\documentclass{article}
\usepackage{amsmath,amsfonts,amscd,amssymb,bm,cite,amsthm}
\usepackage{mathrsfs,listings,url}
\usepackage{graphics,color,ulem}
\usepackage{float}
\usepackage{subfigure}
\usepackage[colorlinks=true, bookmarksopen,
pdfauthor={CST},
pdfcreator={pdftex},
pdfsubject={algorithms},
linkcolor={blue},
anchorcolor={black},
citecolor={firebrick},
filecolor={magenta},
menucolor={black},
plainpages=false,pdfpagelabels,
urlcolor={db} ]{hyperref}
\usepackage{comment}
\usepackage{verbatim}
\usepackage{marginnote}
\usepackage{float}
\usepackage[makeroom]{cancel}
\usepackage[pdftex]{graphicx}
\usepackage[section]{placeins}
\usepackage{mathptmx}
\usepackage{cases}
\usepackage{subeqnarray}
\usepackage{hhline}
\usepackage[linesnumbered,ruled,vlined]{algorithm2e}
\usepackage{xcolor}
\usepackage{tikz}
\usepackage{caption} 
\usepackage{fullpage}
\usepackage{enumitem}
\usepackage{authblk}
\usepackage{amsthm}  % Needed for theorem-like environments
\newtheorem{lemma}{Lemma}
\newtheorem{proposition}{Proposition}
\newtheorem{theorem}{Theorem}
\newtheorem{corollary}{Corollary}

\excludecomment{confidential}

\restylefloat{table}
\allowdisplaybreaks
\definecolor{firebrick}{rgb}{0.698,0.133,0.133}
\numberwithin{equation}{section}

\begin{document}

\title{Long-Time H1-Stability of the Cauchy One-Leg $\theta$-method for the Navier-Stokes Equations}
\author{Isabel Barrio Sanchez\thanks{Department of Mathematics,
            University of Pittsburgh,
            Pittsburgh, PA
            15260, 
            USA. Email: isb42@pitt.edu},
            Catalin Trenchea\thanks{Department of Mathematics,
            University of Pittsburgh,
            Pittsburgh, PA
            15260, 
            USA. Email: trenchea@pitt.edu}, and
            Wenlong Pei\thanks{Department of Mathematics,
            The Ohio State University,
            Columbus, OH
            43210, 
            USA. Email: pei.176@osu.edu}}

\date{}                     %% if you don't need date to appear

\renewcommand\Affilfont{\itshape\small}
\maketitle

\begin{abstract}
In this paper we study the long-time stability of the Cauchy one-leg $\theta$-methods for the two-dimensional Navier-Stokes equations. We establish the uniform dissipativity in $H^1$, in the sense that the semi-discrete-in-time approximations possess a global attractor for a small enough time step, using the discrete Gr\"{o}nwall lemma and the discrete uniform Gr\"{o}nwall lemma.
\end{abstract}

% Section 1 - Introduction
\section{Introduction and background}
The incompressible Navier-Stokes equations (NSE) are the fundamental mathematical model for viscous fluid flow, having a wide range of applications such as aircraft aerodynamics, weather prediction, and blood flow simulation \cite{MR609732,MR1855030,MR1441312,MR2808162,vallis2006atmospheric,washington2005introduction,MR2777984,Gunz02,Furs00,FormaggiaQuarteroniVeneziani2009}. We consider the two-dimensional case
\begin{align}\label{NSE}
    &u_t+u\cdot\nabla u -\nu\Delta u +\nabla p = f,
    \\
    & \nabla\cdot u=0,
    \notag
\end{align}
where $\Omega\subset\mathbb{R}^2$ is a domain with boundary $\partial\Omega$ of class $C^2$, $u(t,x): [0,\infty)\times\Omega \to \mathbb{R}^2$ denotes the velocity, $p(t,x): [0,\infty) \times\Omega  \to \mathbb{R}$ is the pressure, $\nu$ is the kinematic viscosity, and $f\in L^{\infty}(\mathbb{R}_+;L^2(\Omega)^2)$ represents the body forces applied to the fluid. Furthermore, we consider an initial condition $u_0=u(0,x)\in H^1_0(\Omega)$ and a no-slip boundary condition $u|_{\partial\Omega}=0$.

An understanding of the behavior of \eqref{NSE} for long times is essential for many applications, especially the ones concerning climate modeling. As shown in \cite{MR1441312,MR1855030,MR1116181}, in the continuous, two-dimensional case, the $L^2(\Omega)$ norm of the solution is uniformly bounded in time 
\begin{align}\label{eq:ctsL2}
    \|u(t)\|_{L^2(\Omega)}^2 \leq \|u_0\|^2e^{-\nu \lambda_1t}+\frac{1}{\nu^2\lambda_1^2}(1-e^{-\nu \lambda_1 t})\|f\|_{L^{\infty}(\mathbb{R}_+;L^2(\Omega)^2)}^2,
\end{align}
where $\lambda_1$ is the first eigenvalue of the Stokes operator.
Moreover, as stated in \cite{MR2217369,MR1441312}, $u$ can be uniformly bounded in $H^1_0(\Omega)$ by a function which depends on the initial condition
\begin{align}\label{ctsH1init}
    \|\nabla u(t)\|^2\leq K(\|\nabla u_0\|,\|f\|_{L^{\infty}(\mathbb{R}_+;L^2(\Omega)^2)}).
\end{align}
The dependence on the initial condition is transient, i.e., we have that for sufficiently large times
\begin{align}\label{ctsH1}
    \|\nabla u(t)\|^2\leq K(\|f\|_{L^{\infty}(\mathbb{R}_+;L^2(\Omega)^2)}),
\end{align}
demonstrating the existence of a continuous global attractor. Indeed, the uniform bound \eqref{eq:ctsL2} in $L^2(\Omega)$ provides the existence of an absorbing set for the dynamical system, while the additional $H^1(\Omega)$ bounds \eqref{ctsH1init}-\eqref{ctsH1} yield the compactness needed to guarantee the existence of a global attractor.

For a numerical model to be reliable, it is essential that it captures the behavior of its continuous counterpart. Therefore, it is important to develop numerical methods that exhibit long-time stability. Starting from the work of Tone and Wirosoetisno \cite{MR2217369} for the Backward Euler (BE) semi-discrete in time approximation, numerous studies have proved this result for a variety of numerical schemes, both for the semi-discretization in time \cite{MR2888307,han2024longtimestablesavbdf2numerical,MR4552256,MR2217369,wang2026unconditionallylongtimestablevariablestep,MR2945615,MR2552226} and fully discrete formulations \cite{MR3652175,MR3316969,MR3864126,MR3770328,cheng2025longtimestabilityconvergenceanalysis,MR3564779,MR2888307,MR3592110,MR2340671,MR4468789,MR4342664,MR1116181}. Most of the results that employ a full space-time discretization require a time-step restriction depending on the mesh size, namely a Courant-Friedrichs-Lewy (CFL) condition. Our goal is to extend the long-time stability result to the class of Cauchy one-leg $\theta$-methods, using only a semi-discretization in time.

The Cauchy one-leg $\theta$-method discretizes an initial value problem of the form $y'(t)=f(t,y)$ by computing
\begin{align}\label{thetamethod}
    \frac{y^{n+1}-y^n}{\tau}=f(t_{n+\theta},y^{n+\theta}),
\end{align}
where $t_{n+\theta} = t_n + \theta\tau$, $
y^{n+\theta}=\theta y^{n+1}+(1-\theta) y^n$, and
$\theta\in[0,1]$.

For more information and properties of this scheme, we refer the reader to \cite{MR1299729,MR3944203,MR4265875,MR4398357,MR4505477,MR4092601}. Note that if $\theta=1$, this method corresponds to the BE method, for which long-time stability of the semi-discretization in time was proved in \cite{MR2217369}. Whenever $\theta=\frac{1}{2}$, this method becomes the symplectic midpoint, or Crank-Nicolson (CN) method, which is a second-order scheme. We note that the long-time stability of the full space-time discretization of CN was proved in \cite{MR2340671} under a CFL condition. For $\theta\in\left(\frac{1}{2},1\right)$, \eqref{thetamethod} is an implicit method that interpolates between CN and BE, providing some numerical dissipation. This method is second-order accurate whenever $\theta=\frac{1}{2}+\mathcal{O}(\tau)$. In \cite{MR3864126}, the long-time stability of a second-order, fully discrete, linearly extrapolated method, which interpolates between CN and BDF2, is proved also under a CFL condition. The paper \cite{MR3864126} extends the results for the cases $\theta=1$  \cite{MR3652175}  and $\theta=\frac{1}{2}$ \cite{MR3316969}. As in our analysis, that stability proof applies only to $\theta>\frac{1}{2}$. In \cite{MR3944203}, where the authors use the one-leg $\theta$-method to construct discrete space-time approximations which converge to weak solutions to NSE which are suitable in the sense of Scheffer and Caffarelli-Kohn-Nirenberg. Their result also requires that $\theta>\frac{1}{2}$.

Following \cite{MR4092601}, we can rewrite the Cauchy one-leg $\theta$-method \eqref{thetamethod} as a two-step scheme consisting of a BE step and a linear extrapolation (which is equivalent to Forward Euler (FE)). This allows us to leverage some of the techniques developed for BE in \cite{MR2217369}, so part of our argument for the $H^1$ bound will mimic that proof. The method \eqref{thetamethod} applied to \eqref{NSE} reads: given initial condition $u_0$, compute $u^{n+1}$ for all $n\geq 0$ as follows
\begin{align} \label{eq:BENSE}
\frac{1}{\theta\tau}(u^{n+\theta}-u^n)-&\nu\Delta u^{n+\theta}+u^{n+\theta}\cdot
\nabla u^
{n+\theta}+\nabla p^{n+\theta}=f^{n+\theta},\\\notag
&\nabla\cdot u^{n+\theta}=0,
\end{align}
\begin{equation}\label{eq:FENSE}
u^{n+1}=\frac{1}{\theta}u^{n+\theta}-\bigg(\frac{1}{\theta}-1\bigg)u^n.
\end{equation}

In this paper, we establish the uniform dissipativity of $u^n$, or the long-time stability of \eqref{eq:BENSE}-\eqref{eq:FENSE} in $L^2(\Omega)$ and $H^1(\Omega)$, hence proving that a discrete global attractor exists. In Section \ref{sec:math-prelim}, we introduce the mathematical background necessary for our analysis. Then we proceed to prove the $L^2(\Omega)$ and $H^1(\Omega)$ stability in sections \ref{sec:L2bound} and \ref{sec:H1bound} respectively.

\section{Mathematical preliminaries}\label{sec:math-prelim}
In this section we present some of the mathematical tools and notations used in the analysis.

We consider the following solenoidal Sobolev spaces
\begin{align}
    V &= \{ v\in H^1_0(\Omega)^2, \nabla \cdot v = 0 \}
    \\
    H &= \{ v\in L^2(\Omega)^2, \nabla\cdot v=0, v\cdot \hat{n}=0\; \text{on} \;\partial\Omega\},
\end{align}
where $\hat{n}$ is the outward unit normal on $\partial\Omega$. The space $H$ is endowed with the norm and inner product of $L^2(\Omega)^2$, which we denote by $\|\cdot\|$ and $(\cdot,\cdot)$. The space $V$ is endowed with the $H^1(\Omega)$ semi-inner product and norm, which define a norm and inner product on $V$. We assume that $f\in L^{\infty}(\mathbb{R}_+,H)$, and
\begin{align}
    \|f\|_{\infty} &= \|f\|_{L^{\infty}(\mathbb{R}_+,H)}.
\end{align}
We denote by $A$ the Stokes operator, a linear, unbounded, positive definite operator from $V$ to $V'$ such that 
\[
\langle Au,v\rangle_{V',V} = ( \nabla u,\nabla v) \qquad \forall u,v\in V.
\]
The domain of $A$, denoted by $D(A)$, satisfies
\[D(A) = H^2(\Omega)^2\cap V, \qquad D(A)\subset V\subset H.\]
For more information on operator $A$, we refer the reader to \cite{MR1441312}.
Denoting the first eigenvalue of $A$ as $\lambda_1>0$, we recall the Poincar\'{e}-Friedrichs inequality
 \begin{equation}
     \label{PF}
     \|u\|\leq \frac{1}{\sqrt{\lambda_1}}\|\nabla u\| \qquad \forall u\in V.
 \end{equation}
We also denote $b(u,v,w)=( u\cdot \nabla v,w)$, which is a continuous, trilinear, and skew-symmetric form
\begin{align}\label{triformzero}
    b(u,v,w) = -b(u,w,v) \qquad \forall u,v,w\in V.
\end{align}
We define a bilinear operator $B$ from $V\times V$ to $V'$ by
\[
\langle B(u,v),w\rangle_{V',V}=b(u,v,w)\qquad\forall u,v,w\in V.
\]
Then \eqref{NSE} can be rewritten as \cite{MR3533002}
\[
\frac{du}{dt}+\nu Au+B(u,u)=f, \qquad u(0)=0,
\]
and the discrete Cauchy one-leg method \eqref{eq:BENSE}-\eqref{eq:FENSE} becomes
\begin{align} \label{eq:STAR}
\frac{1}{\tau}(u^{n+1}-u^n)+&\nu A u^{n+\theta}+B(u^{n+\theta}, u^
{n+\theta})=f^{n+\theta}.
\end{align}
Finally, we recall the Ladyzhenskaya inequality in two spatial dimensions \cite{MR2808162,MR108963,MR108962}
\begin{align}\label{ladyzhenskaya}
    \|u\|_{L^4(\Omega)}\leq 2^{-1/4}\|u\|^{1/2}\|\nabla u\|^{1/2}.
\end{align}

\begin{confidential}
{\color{red}

    Suppose we want to solve the equation $y'(t)=f(t,y)$, where $0\leq t\leq b$. Discretize the interval of integration by a mesh $0=t_0<t_1<t_2<\dots<t_N=b$, and let $\tau = t_{n+1}-t_n$, the time step. For the sake of simplicity, we will assume constant time step.\\
Let $t_{n+\theta}=t_n+\theta\tau$ and $y_{n+\theta}=\theta y_{n+1}+(1-\theta)y_n$, where $\theta\in[\frac{1}{2},1]$, $y_n\approx y(t_n)$.  Cauchy's one-legged $\theta$-method goes as follows: 
\[\frac{y^{n+1}-y^n}{\tau}=f(t^{n+\theta},y^{n+\theta}).\]
As shown in \cite{BEFE}, we can rewrite the above method by performing Backward Euler (BE) step and a linear extrapolation:

 \[\text{BE:\quad}\frac{y^{n+\theta}-y^n}{\theta\tau}=f(t^{n+\theta},y^{n+\theta})\]
 \[\text{LE:\quad } y^{n+1}=\frac{1}{\theta}y^{n+\theta}-\bigg(\frac{1}{\theta}-1\bigg)y^n\]
 
Note that the Linear Extrapolation step is equivalent to a Forward Euler step. Why is that?

From the BE step, we have that 
\[
\frac{u^{n+\theta}-u^n}{\theta\tau}=f^{n+\theta},
\]
so $u^n=u^{n+\theta}-\theta\tau f^{n+\theta}$. We can just plug thet into LE to get FE:

\begin{align*}
    u^n+1 &= \frac{1}{\theta}u^{n+\theta}-\left(\frac{1-\theta}{\theta}\right)u^n
    \\
    & =  \frac{1}{\theta}u^{n+\theta}-\left(\frac{1-\theta}{\theta}\right)(u^{n+\theta}-\theta\tau f^{n+\theta})
    \\
    &
    = u^{n+\theta}+(1-\theta)\tau f^{n+\theta}
\end{align*}

So we get FE:

\[
\frac{u^{n+1}-u^{n+\theta}}{(1-\theta)\tau}=f^{n+\theta}.
\]

Applying this scheme to NSE in intro, we get:
\begin{align} 
\frac{1}{\theta\tau}(u^{n+\theta}-u^n)-&\nu\Delta u^{n+\theta}+u^{n+\theta}\cdot
\nabla u^
{n+\theta}+\nabla p^{n+\theta}=f^{n+\theta}\\
&\nabla\cdot u^{n+\theta}=0
\end{align}
\begin{equation}
u^{n+1}=\frac{1}{\theta}u^{n+\theta}-\bigg(\frac{1}{\theta}-1\bigg)u^n
\end{equation}
}
\end{confidential}

% Section 2 - L2-Stability
\section{H-Stability}\label{sec:L2bound}
To prove the long-time stability of $u^n$ in $V$, we first establish a bound in $H$, which is a discrete analogue of \eqref{eq:ctsL2}. We also collect several auxiliary results essential for the $V$ stability analysis. 

This section is organized as follows. In Lemma \ref{lemma:BEstab}, we derive an energy estimate for the BE step \eqref{eq:BENSE}. In Lemma \ref{lemma:energy-equality}, we use \eqref{eq:FENSE} to reformulate the result in Lemma \ref{lemma:BEstab} in terms of $u^{n+1}$ and $u^n$ instead of $u^{n+\theta}$. We then include a remark explaining why the result of this paper does not hold for $\theta=\frac{1}{2}$. In Proposition \ref{prop:L2bound}, we establish an $H$-bound on $u^n$, and in Corollary \ref{cor:absorbing}, we prove the existence of an absorbing set in $H$. Finally, in Lemmas \ref{lemma:gradutheta} and \ref{lemma:gradu1}, we obtain $L^2(H^1(\Omega))$ bounds on the fractional time values $u^{n+\theta}$ and the integer values $u^n$, respectively.

In the first result, we use the kinematic dissipation and the numerical dissipation of the BE step \eqref{eq:STAR} to bound $\|u^{n+\theta}\|$ in terms of $\|u^{n}\|$ and $\|f\|_{\infty}$.
\begin{lemma}\label{lemma:BEstab}
For every $n\geq0$ and $\theta\in[0,1]$ we have
\begin{equation}\label{eq:energy0}
\left(1+\lambda_1 \nu\theta\tau\right)\|u^{n+\theta}\|^2-\|u^{n}\|^2+\|u^{n+\theta}-u^n\|^2\leq \frac{1}{\nu}\theta\tau \frac{1}{\lambda_1}\|f\|^2_{\infty}.
\end{equation}
\end{lemma}
\begin{proof}
    Testing \eqref{eq:STAR} with $\theta\tau u^{n+\theta}$ in $H$ and using \eqref{triformzero}, the divergence theorem, the polarization identity, and the H\"{o}lder, Young, and Poincar\'{e}-Friedrichs \eqref{PF} inequalities we obtain
    \begin{confidential}
    {\color{red}
    \begin{align}
        ( u^{n+\theta},u^{n+\theta})-( u^n,u^{n+\theta}) -\theta\tau\nu( \Delta u^{n+\theta},u^{n+\theta}) +\theta\tau ( \nabla p^{n+\theta},u^{n+\theta}) = \theta\tau( f^{n+\theta},u^{n+\theta}).
    \end{align}
    Using the divergence theorem and polarization identity,
    \begin{align}
        \frac{1}{2}\|u^{n+\theta}\|^2
        &
        -\frac{1}{2}\|u^n\|^2
        +\frac{1}{2}\|u^{n+\theta}-u^n\|^2 
        -\nu\theta\tau\int_{\partial\Omega}\nabla u^{n+\theta}u^{n+\theta}\cdot \hat{n} 
        + \nu\theta\tau\|\nabla u^{n+\theta}\|^2
        \\
        &+\theta\tau\int_{\partial\Omega}p^{n+\theta}u^{n+\theta}\cdot\hat{n} -\theta\tau\int_{\Omega}p^{n+\theta}\nabla\cdot u^{n+\theta} 
        = 
        \theta\tau (f^{n+\theta},u^{n+\theta})
    \end{align}

    And using the fact that $\nabla \cdot u^{n+\theta}=0$ and $u=0$ on $\partial\Omega$, multiplying by 2, and using the H\"{o}lder, Young, and Poincar\'{e}-Friedrichs inequalities:

    \begin{align}
        \|u^{n+\theta}\|^2
        &
        -\|u^n\|^2
        +\|u^{n+\theta}-u^n\|^2  
        + 2\nu\theta\tau\|\nabla u^{n+\theta}\|^2
        \\
        &
        \leq
        2\theta\tau \|f^{n+\theta}\|\| u^{n+\theta}||
        \\
        &
        \leq
        \frac{\theta\tau}{\nu}\frac{1}{\lambda_1}\|f^{n+\theta}\|^2 + \nu\theta\tau\|\nabla u^{n+\theta}||^2
        \end{align}

\begin{align}
        \frac{1}{2}\|u^{n+\theta}\|^2
        -\frac{1}{2}\|u^n\|^2
        +\frac{1}{2}\|u^{n+\theta}-u^n\|^2 
        + \nu\theta\tau\|\nabla u^{n+\theta}\|^2
        = 
        \theta\tau (f^{n+\theta},u^{n+\theta})
    \end{align}

    Multiplying by 2 and using the H\"{o}lder, Young, and Poincar\'{e}-Friedrichs inequalities:
        }
\end{confidential}
   \begin{confidential}
   {\color{red}
    \begin{align}
\bigg(1+\lambda_1 \nu\theta\tau\bigg)&\|u^{n+\theta}\|^2-\|u^{n}\|^2+\|u^{n+\theta}-u^n\|^2\leq \frac{1}{\nu}\theta\tau \frac{1}{\lambda_1}\|f^{n+\theta}\|^2\\
&\leq \frac{1}{\nu}\theta\tau \frac{1}{\lambda_1}\|f\|^2_{\infty}
\end{align}

\begin{align}
\bigg(1+\lambda_1 \nu\theta\tau\bigg)&\|u^{n+\theta}\|^2-\|u^{n}\|^2+\|u^{n+\theta}-u^n\|^2\leq \frac{1}{\nu}\theta\tau \frac{1}{\lambda_1}\|f\|^2_{\infty}
\end{align}
}
\end{confidential}
    \begin{align}
        \|u^{n+\theta}\|^2
        &
        -\|u^n\|^2
        +\|u^{n+\theta}-u^n\|^2  
        + \nu\theta\tau\|\nabla u^{n+\theta}\|^2
        \leq
        \frac{\theta\tau}{\nu}\frac{1}{\lambda_1}\|f^{n+\theta}\|^2.
    \end{align}
    Using the Poincar\'{e}-Friedrichs inequality \eqref{PF} gives the relation \eqref{eq:energy0}.
\end{proof}

In the rest of the paper we assume the following time-step restriction 
\begin{equation}\label{eq:step-condition}
    \tau\leq \frac{1}{\lambda_1\nu} =: \kappa_1 .
\end{equation}
The next step is to write the inequality \eqref{eq:energy0} in terms of $u^{n+1}$ and $u^n$ using the extrapolation \eqref{eq:FENSE}. To simplify the presentation, we use the following notations
\begin{align}
\label{eq:alpha}
\alpha_\theta
& 
:=
 \theta 
 - \frac{1}{2} ( 2\theta - 1 ) ( 1 -  \theta )    \lambda_1 \nu \tau 
+ 
\frac{1}{2}
\Big[ ( 2 \theta  - 1 ) ( 1 - \theta ) \Big( 4\theta 
 - ( 1 - \theta )  ( 2 \theta  + 1 )   \lambda_1 \nu  \tau \Big) 
 %\Big]^{1/2}
%\Big( 
\lambda_1 \nu  \tau 
%\Big)^{1/2}
\Big]^{1/2}
,
\\
\label{eq:epsilon} 
\varepsilon_\theta  
& 
:= 
\lambda_1 \nu ( 2\theta - 1)  \tau
,
\\
a_\theta 
&
:= 
\frac{1}{2} \Big\{
( 2 \theta - 1 ) \Big( 4 \theta  - \lambda_1 \nu ( 2\theta + 1 ) ( 1 - \theta )  \tau
 \Big)^{1/2}
- 
\Big( 
\lambda_1 \nu ( 1 - \theta )  \tau
\Big)^{1/2}
\Big\}
,
\\
\label{eq:b}
b_\theta 
& 
:=
\frac{1}{2} \Big\{
( 2 \theta - 1 ) \Big( 4 \theta  - \lambda_1 \nu ( 2\theta + 1 ) ( 1 - \theta )  \tau
 \Big)^{1/2}
+ 
\Big( 
\lambda_1 \nu ( 1 - \theta )  \tau
\Big)^{1/2}
\Big\}
.
\end{align}

We note that, under the assumption \eqref{eq:step-condition} on the size of the time step, for all $\theta\in\left(\frac{1}{2},1\right)$, we have that  
\begin{align}\label{1star}
    \alpha_\theta
\in (\mbox{$\frac{1}{2}$} , \mbox{$\frac{3}{2}$}], \qquad \varepsilon_\theta > 0.
\end{align}

The next result rearranges the left-hand-side of \eqref{eq:energy0} in order to extend the dissipation result for the fractional time $\|u^{n+\theta}\|$ to the integer time $\|u^n\|$.

\begin{confidential}
    {\color{red}
\begin{align}
\label{eq:alpha-epsilon-ab-bounds}
\alpha
\in (\mbox{$\frac{1}{2}$} , \mbox{$\frac{3}{2}$}] 
,
\quad
\varepsilon > 0, 
\quad 
a, b \in \mathbb{R}
\qquad 
\forall \theta \in (\mbox{$\frac{1}{2} , 1$}).
\end{align}

    Now we prove the bounds on $\alpha$, that is, $\frac{1}{2}<\alpha\leq\frac{3}{2}$. We note that these bounds are not sharp, but they are good enough for us.\\
The proofs follow by direct computation:\\
Recall 
\begin{align*}
    \alpha & = 
 \theta 
 - \frac{1}{2} ( 2\theta - 1 ) ( 1 -  \theta )    \lambda_1 \nu \tau 
\\
& \qquad
+ 
\frac{1}{2}
\Big[ ( 2 \theta  - 1 )(1-\theta) \Big( 4\theta 
 - ( 1 - \theta )  ( 2 \theta  + 1 )   \lambda_1 \nu  \tau \Big) 
\lambda_1 \nu  \tau 
\Big]^{1/2}
\end{align*}
We will show that 
\begin{align*}
    & - \frac{1}{2} ( 2\theta - 1 ) ( 1 -  \theta )    \lambda_1 \nu \tau 
\\
& \qquad
+ 
\frac{1}{2}
\Big[ ( 2 \theta  - 1 )(1-\theta) \Big( 4\theta 
 - ( 1 - \theta )  ( 2 \theta  + 1 )   \lambda_1 \nu  \tau \Big) 
\lambda_1 \nu  \tau 
\Big]^{1/2} \geq 0
\end{align*}
If $\theta=1$, then the above expression equals 0, so the inequality is true.\\
Now if $\theta\in(\frac{1}{2},1)$:\\
\begin{align*}
 & - \frac{1}{2} ( 2\theta - 1 ) ( 1 -  \theta )    \lambda_1 \nu \tau 
\\
& \qquad
+ 
\frac{1}{2}
\Big[ ( 2 \theta  - 1 )(1-\theta) \Big( 4\theta 
 - ( 1 - \theta )  ( 2 \theta  + 1 )   \lambda_1 \nu  \tau \Big) 
\lambda_1 \nu  \tau 
\Big]^{1/2} \geq 0
\\
&
\iff
\\
&
\Bigg[(2\theta-1)(1-\theta)\lambda_1 \nu  \tau \Big( 4\theta 
 - ( 1 - \theta )  ( 2 \theta  + 1 )   \lambda_1 \nu  \tau \Big) \Bigg]^{1/2}
 \geq
 (2\theta-1)(1-\theta)\lambda_1 \nu  \tau
 \\
 &
 \iff
 \\
 &
 (2\theta-1)(1-\theta)\lambda_1 \nu  \tau \Big( 4\theta 
 - ( 1 - \theta )  ( 2 \theta  + 1 )   \lambda_1 \nu  \tau \Big)
 \geq
 (2\theta-1)^2(1-\theta)^2\lambda_1^2\nu^2  \tau^2
 \\
 &
 \iff
 \\
 &
 \Big( 4\theta 
 - ( 1 - \theta )  ( 2 \theta  + 1 )   \lambda_1 \nu  \tau \Big)
 \geq
 (2\theta-1)(1-\theta)\lambda_1 \nu  \tau
 \\
 &
 \iff
 \\
 &
 4\theta -\lambda_1 \nu  \tau\theta+2\theta^2\lambda_1 \nu  \tau-\lambda_1 \nu  \tau
 \geq
 3\theta \lambda_1 \nu  \tau -2 \theta^2 \lambda_1 \nu  \tau-\lambda_1 \nu  \tau
 \\
 &
 \iff
 \\
 &
 4\theta^2\lambda_1 \nu  \tau-4\lambda_1 \nu  \tau\theta +4\theta \geq 0
 \\
 &
 \iff
 \\
 &
 \theta(\lambda_1 \nu  \tau\theta-\lambda_1 \nu  \tau+1)\geq 0
 \\
 &
 \iff
 \\
 &
 \lambda_1 \nu  \tau(\theta-1)\geq -1
\end{align*}

And note that, by the time-step restriction, $0\leq \lambda_1 \nu  \tau\leq 1$. And for 
$\theta\in(1/2,1)$, $(\theta-1)\geq 1/2-1=-\frac{1}{2}$. So $\lambda_1 \nu  \tau(\theta-1)\geq -1$, and the lower bound is true.

 So
\begin{align*}
    \alpha &= \theta - \frac{1}{2} ( 2\theta - 1 ) ( 1 -  \theta )    \lambda_1 \nu \tau 
\\
& \qquad
+ 
\frac{1}{2}
\Big[ ( 2 \theta  - 1 ) \Big( 4\theta 
 - ( 1 - \theta )  ( 2 \theta  + 1 )   \lambda_1 \nu  \tau \Big) \Big]^{1/2}
\Big( 
\lambda_1 \nu ( 1 - \theta ) \tau 
\Big)^{1/2} \geq \theta \geq \frac{1}{2}.
\end{align*}
So we have the lower bound. Now we want to prove that $\alpha\leq\frac{3}{2}$ for $\theta\in(\frac{1}{2},1)$. \\
That is, we want to show:
\begin{align*}
    \alpha &= \theta - \frac{1}{2} ( 2\theta - 1 ) ( 1 -  \theta )    \lambda_1 \nu \tau 
\\
& \qquad
+ 
\frac{1}{2}
\Big[ ( 2 \theta  - 1 )(1-\theta)\lambda_1 \nu  \tau \Big( 4\theta 
 - ( 1 - \theta )  ( 2 \theta  + 1 )   \lambda_1 \nu  \tau \Big) 
\Big]^{1/2} 
\leq
\frac{3}{2}
\end{align*}
It suffices to show that
\begin{align*}
\theta 
+ 
\frac{1}{2}
\Big[ ( 2 \theta  - 1 )(1-\theta)\lambda_1 \nu  \tau \Big( 4\theta 
 - ( 1 - \theta )  ( 2 \theta  + 1 )   \lambda_1 \nu  \tau \Big) 
\Big]^{1/2} \leq \frac{3}{2} 
\end{align*}
Furthermore, it suffices to show that
\begin{align*}
\frac{1}{2}
\Big[ ( 2 \theta  - 1 )(1-\theta)\lambda_1 \nu  \tau \Big( 4\theta 
 - ( 1 - \theta )  ( 2 \theta  + 1 )   \lambda_1 \nu  \tau \Big) 
\Big]^{1/2} \leq \frac{1}{2} 
\end{align*}
Then, we would have that
\begin{align*}
\theta 
+& 
\frac{1}{2}
\Big[ ( 2 \theta  - 1 )(1-\theta)\lambda_1 \nu  \tau \Big( 4\theta 
 - ( 1 - \theta )  ( 2 \theta  + 1 )   \lambda_1 \nu  \tau \Big) 
\Big]^{1/2} 
\\
&
\leq 
1+\frac{1}{2}
\Big[ ( 2 \theta  - 1 )(1-\theta)\lambda_1 \nu  \tau \Big( 4\theta 
 - ( 1 - \theta )  ( 2 \theta  + 1 )   \lambda_1 \nu  \tau \Big) 
\Big]^{1/2} \leq \frac{1}{2}+1
\end{align*}
Which would imply our desired result:
\begin{align*}
\theta 
+ 
\frac{1}{2}
\Big[ ( 2 \theta  - 1 )(1-\theta)\lambda_1 \nu  \tau \Big( 4\theta 
 - ( 1 - \theta )  ( 2 \theta  + 1 )   \lambda_1 \nu  \tau \Big) 
\Big]^{1/2} \leq \frac{3}{2} 
\end{align*}

So let's show that
\begin{align*}
\frac{1}{2}
\Big[ ( 2 \theta  - 1 )(1-\theta)\lambda_1 \nu  \tau \Big( 4\theta 
 - ( 1 - \theta )  ( 2 \theta  + 1 )   \lambda_1 \nu  \tau \Big) 
\Big]^{1/2} \leq \frac{1}{2} .
\end{align*}

\begin{align*}
\frac{1}{2}&
\Big[ ( 2 \theta  - 1 )(1-\theta)\lambda_1 \nu  \tau \Big( 4\theta 
 - ( 1 - \theta )  ( 2 \theta  + 1 )   \lambda_1 \nu  \tau \Big) 
\Big]^{1/2} \leq \frac{1}{2} 
\\
&
\iff
\\
&
\Big[ ( 2 \theta  - 1 )(1-\theta)\lambda_1 \nu  \tau \Big( 4\theta 
 - ( 1 - \theta )  ( 2 \theta  + 1 )   \lambda_1 \nu  \tau \Big) 
\Big]^{1/2} \leq 1
\\
&
\iff
\\
&
( 2 \theta  - 1 )(1-\theta)\lambda_1 \nu  \tau \Big( 4\theta 
 - ( 1 - \theta )  ( 2 \theta  + 1 )   \lambda_1 \nu  \tau \Big) 
\leq 1
\end{align*}

And we have that
\begin{align*}
    ( 2 \theta  - 1 )&(1-\theta)\lambda_1 \nu  \tau \Big( 4\theta 
 - ( 1 - \theta )  ( 2 \theta  + 1 )   \lambda_1 \nu  \tau \Big)
 \\
 &
 \leq
 ( 2 \theta  - 1 )(1-\theta)\lambda_1 \nu  \tau 4\theta 
 \\
 &
 \leq( 2 \theta  - 1 )(1-\theta) 4\theta 
\end{align*}

Now I'll show that the maximum of this expression (when $\theta\in(1/2,1)$) is less than 1.

Let $f(\theta)=( 2 \theta  - 1 )(1-\theta) 4\theta =-8\theta^3+12\theta^2-4\theta$. Then $f'(\theta)=-24\theta^2+24\theta-4$. Using the quadratic formula, we get that the two roots or the derivative are $\theta_1=\frac{1}{2}+\frac{\sqrt{3}}{6}\in(1/2,1)$ and $\theta_2=\frac{1}{2}-\frac{\sqrt{3}}{6}\notin(1/2,1)$. So the maximum that concerns us happens at $\theta_1$. We can check it is actually a maximum because $f''(\theta)=-42\theta+24\leq 0 \; \forall\theta\in(1/2,1)$. And $f(\theta_1)=\frac{2\sqrt{3}}{9}\leq 1$, which is what we wanted to show. So $\alpha\leq\frac{3}{2}$ for all $\theta\in\left(\frac{1}{2},1\right)$.

The next result gives an expression of the left hand side of \eqref{eq:energy0} in terms of $u^{n+1}$ and $u^n$.

    }
\end{confidential}

\begin{lemma}
\label{lemma:energy-equality}
Under the time-step assumption \eqref{eq:step-condition}, for every $\theta\in\left(\frac{1}{2},1\right)$, there exist constants (depending on $\theta$) $\alpha_\theta,\varepsilon_\theta>0$ and $a_\theta,b_\theta\in\mathbb{R}$ given in \eqref{eq:alpha}-\eqref{eq:b} such that
\begin{align}
\label{eq:energy1}
( 1 + \lambda_1 \nu \theta \tau ) \|u^{n+\theta}\|^2 - \|u^n\|^2 + \|u^{n+\theta} - u^n\|^2
=
(\alpha_\theta + \varepsilon_\theta) \| u^{n+1} \|^2  - \alpha_\theta \| u^{n} \|^2 + \| a_\theta u^{n+1} - b_\theta u^n \|^2\quad \forall n\geq0.
\end{align}
\end{lemma}
\begin{proof}
The conclusion follows by direct calculation. First, using the extrapolation \eqref{eq:FENSE}
 into the left hand side of \eqref{eq:energy0} gives
\begin{align*}
&
\left( 1 + \lambda_1 \nu \theta \tau \right) \|u^{n+\theta}\|^2 - \|u^n\|^2 + \|u^{n+\theta} - u^n\|^2
\\
& 
=
 \theta^2  \left( 2 + \lambda_1 \nu \theta \tau \right) \| u^{n+1} \|^ 2  
+ 2\theta \left[  (1 - \theta) \left( 1 + \lambda_1 \nu \theta \tau \right)   - \theta  \right]
( u^{n+1} , u^{n} )
\\
& \qquad
+
	( \theta - 1) \left[ 
\left( 1 + \lambda_1 \nu \theta \tau \right) ( \theta - 1 ) + ( \theta + 1 ) \right] 
\| u^{n} \|^ 2  .  
\end{align*}
\begin{confidential}
    {\color{red}
    \begin{align*}
&
\Big( 1 + \lambda_1 \nu \theta \tau \Big) \|u^{n+\theta}\|^2 - \|u^n\|^2 + \|u^{n+\theta} - u^n\|^2
\\
& 
=
\Big( 1 + \lambda_1 \nu \theta \tau \Big) \| \theta u^{n+1} + (1 - \theta) u^n \|^2 - \|u^n\|^2 + \| \theta u^{n+1} + (1 - \theta) u^n  - u^n\|^2
\\
& 
=
\Big( 1 + \lambda_1 \nu \theta \tau \Big) \| \theta u^{n+1} + (1 - \theta) u^n \|^2 
- \|u^n\|^2 
+ \theta ^2 \|  u^{n+1} - u^n\|^2
\\
& 
=
\Big( 1 + \lambda_1 \nu \theta \tau \Big) \Big( \theta^2 \| u^{n+1} \|^ 2  + 2\theta(1 - \theta) ( u^{n} ,   u^{n+1} ) + (1 - \theta)^2 \| u^n \|^2 \Big)
\\
&
-  \|u^n\|^2 
+ \theta ^2 \Big( \|  u^{n+1} \|^2 - 2( u^{n+1} , u^{n} )  + \| u^n\|^2 \Big)
\\
& 
=
\bigg[ \Big( 1 + \lambda_1 \nu \theta \tau \Big)  \theta^2 + \theta^2 \bigg] \| u^{n+1} \|^ 2  
+ \bigg[  2\theta(1 - \theta) \Big( 1 + \lambda_1 \nu \theta \tau \Big)   - 2 \theta^2  \bigg]
( u^{n+1} , u^{n} )
\\
& 
\qquad
+\bigg[ 
\Big( 1 + \lambda_1 \nu \theta \tau \Big) ( 1 - \theta)^2 + (\theta^2 - 1) \bigg] 
\| u^{n} \|^ 2  .  
\\
& 
=
 \theta^2  \Big( 2 + \lambda_1 \nu \theta \tau \Big) \| u^{n+1} \|^ 2  
+ 2\theta \bigg[  (1 - \theta) \Big( 1 + \lambda_1 \nu \theta \tau \Big)   - \theta  \bigg]
( u^{n+1} , u^{n} )
\\
& \qquad
+
	( \theta - 1) \bigg[ 
\Big( 1 + \lambda_1 \nu \theta \tau \Big) ( \theta - 1 ) + ( \theta + 1 ) \bigg] 
\| u^{n} \|^ 2  .  
\end{align*}
    }
\end{confidential}

Next, we note that by \eqref{1star} the constants $\alpha_\theta,\varepsilon_\theta>0$, $a_\theta,b_\theta\in\mathbb{R}$ from \eqref{eq:alpha}-\eqref{eq:b} satisfy
\begin{align} 
\label{eq:exp0}
\alpha_\theta + \varepsilon_\theta + a_\theta^2 
 &= 
 \theta^2  \left( 2 + \lambda_1 \nu \theta \tau \right) ,
\\ 
b_\theta^2 - \alpha _\theta
&= 
	( \theta - 1) \left[ 
\left( 1 + \lambda_1 \nu \theta \tau \right) ( \theta - 1 ) + ( \theta + 1 ) \right] ,
\\
2 a_\theta b_\theta 
&= 
\label{eq:expend}
2\theta \left[  \theta - (1 - \theta) \left( 1 + \lambda_1 \nu \theta \tau \right) \right],
\end{align}
\begin{confidential}
    {\color{red}
    These expressions are satisfied by the constants $\alpha,\varepsilon,a,b$ we picked in \eqref{eq:alpha}-\eqref{eq:b}. The choice of these constants follows a long computation that is shown in Appendix \ref{sec:appendix1}.
    }
\end{confidential}
and therefore
\begin{align}
\label{eq:goal1}
& 
(\alpha_\theta + \varepsilon_\theta) \| u^{n+1} \|^2  - \alpha_\theta \| u^{n} \|^2 + \| a_\theta u^{n+1} - b_\theta u^n \|^2
\\
& 
= 
(\alpha_\theta + \varepsilon_\theta + a_\theta^2 ) \| u^{n+1} \|^2  + ( b_\theta^2 - \alpha_\theta)  \| u^{n} \|^2  -  2a_\theta b_\theta  ( u^{n+1} , u^n )
\notag
\\
& 
\notag
= \theta^2  \left( 2 + \lambda_1 \nu \theta \tau \right) \| u^{n+1} \|^ 2  
+ 
	( \theta - 1) \left[ 
\left( 1 + \lambda_1 \nu \theta \tau \right) ( \theta - 1) + ( \theta + 1) \right] 
\| u^{n} \|^ 2 \\
&\qquad+ 
2\theta \left[  (1 - \theta) \left( 1 + \lambda_1 \nu \theta \tau \right)   - \theta  \right]
( u^{n+1} , u^{n} ),
\notag
\end{align}
which finally gives \eqref{eq:energy1}. 
\end{proof}
\noindent

\textbf{Remark}. We emphasize that, when $\theta=\frac{1}{2}$, there is no $\alpha_\theta,\varepsilon_\theta>0$ such that \eqref{eq:energy1} holds. We first note that for \eqref{eq:energy1} to hold, the relations \eqref{eq:exp0}-\eqref{eq:expend} are necessary and sufficient. Adding these expressions we have that 
\begin{align*}
& 
(a_\theta + b_\theta)^2 + \varepsilon_\theta
 = 
4 \theta^2 - 1
 + ( 2\theta - 1 )^2   \left( 1 + \lambda_1 \nu \theta \tau \right),
\end{align*}
which, for $\theta=\frac{1}{2}$, shows there is no $\varepsilon_\theta>0$ satisfying \eqref{eq:energy1}.

\begin{confidential}
At this point, it is important to mention why this argument will not work for the case when $\theta=\frac{1}{2}$, which corresponds to the midpoint rule. 

The goal of our argument is to pick $\alpha,\varepsilon>0$ and $a,b\in\mathbb{R}$ such that 
    {\color{red}
    \begin{align*}
    (\alpha + \varepsilon)&\|u_{n+1}\|^2-\alpha\|u_{n}\|^2+\|a_nu_{n+1}-b_nu_n\|^2\\
    &=(1+\delta)\|u_{n+\theta}\|^2-\|u_{n}\|^2+\|u_{n+1}-u_n\|^2\\
    &\leq C\|f\|_{\infty}^2
\end{align*}

\begin{align*}
    (\alpha + \varepsilon)&\|u_{n+1}\|^2-\alpha\|u_{n}\|^2+\|a_nu_{n+1}-b_nu_n\|^2
    \leq C\|f\|_{\infty}^2
\end{align*}
for some constant $C$, so that we can have 
\[\|u_{n+1}\|^2\leq\frac{1}{1+\varepsilon/\alpha}\|u_{n}\|^2+C\frac{1}{\alpha+\varepsilon}\|f\|_{\infty}^2\]

Which, applying the fact that \[\frac{1}{1+\varepsilon/\alpha}\leq e^{-K\varepsilon}\] for some constant $K$ and recursion, gets us the $L^2$ bound.\\

Adding the expressions \eqref{eq:exp0}-\eqref{eq:expend} gives
\begin{align*}
& 
(a + b)^2 + \varepsilon
 = 
4 \theta^2 - 1
 + ( 2\theta - 1 )^2   \Big( 1 + \lambda_1 \nu \theta \tau \Big).
\end{align*}

If $\theta=\frac{1}{2}$, \[(a+b)^2+\varepsilon=0\]
So we can't find $\varepsilon>0$ that satisfy our conditions. 
   }
\end{confidential}

\begin{confidential}
    {\color{red}
    \begin{align*}
& 
(a + b)^2 + \varepsilon
= 
 \theta^2  \Big( 2 + \lambda_1 \nu \theta \tau \Big) 
+
	( \theta - 1) \bigg[ 
\Big( 1 + \lambda_1 \nu \theta \tau \Big) ( \theta - 1 ) + ( \theta + 1 ) \bigg] 
\\
&
 + 2\theta \bigg[  \theta - (1 - \theta) \Big( 1 + \lambda_1 \nu \theta \tau \Big) \bigg]
 \\
 & 
 = 
 \theta^2 
 +  \theta^2  \Big( 1 + \lambda_1 \nu \theta \tau \Big) 
	+ \Big( 1 + \lambda_1 \nu \theta \tau \Big) ( 1 - \theta)^2 + ( \theta^2 - 1 )
\\
&
 + 2\theta^2 
 - 2\theta (1 - \theta) \Big( 1 + \lambda_1 \nu \theta \tau \Big)
 \\
 & 
 = 
4 \theta^2 -1
 + \big[  \theta^2  + ( 1 - \theta)^2  - 2\theta (1 - \theta)  \big] 
 \Big( 1 + \lambda_1 \nu \theta \tau \Big) 
 \\
 &
 = 
4 \theta^2 - 1
 + ( 2\theta - 1 )^2   \Big( 1 + \lambda_1 \nu \theta \tau \Big),
\end{align*}

\begin{align*}
& 
(a + b)^2 + \varepsilon
 = 
4 \theta^2 - 1
 + ( 2\theta - 1 )^2   \Big( 1 + \lambda_1 \nu \theta \tau \Big),
\end{align*}

When $\theta=\frac{1}{2}$, the $\theta$-method corresponds to the Crank-Nicolson scheme, which is second-order accurate. It is important to note that $\theta=\frac{1}{2}$ is not the only case where the $\theta$-method is second-order, since this property is also satisfied when $\theta=\frac{1}{2}+\tau$.\\
 }
\end{confidential}

\begin{confidential}
{\color{red}
(Proof of second order accuracy:)
We have \[\frac{y^{n+1}-y^n}{\tau}=f(t_{n+\theta},y^{n+\theta})\]
\[y^{n+\theta}=\theta y^{n+1}+(1-\theta)y^n\]
\[t_{n+\theta}=t_n+\theta\tau\]
By Taylor's theorem,
\[y^{n+1}=y_n+\tau y_n'+\frac{1}{2}\tau^2y''_n+O(\tau^3)\]
\begin{align*}
f(t_{n+\theta},y^{n+\theta})&=y'_n+\theta\tau f_t(t_n,y^n)+\theta\tau f_y(t_n,y^n)(y'_n+\frac{1}{2}y''_n)+O(\tau^2)\\
&=y'_n+\theta\tau y''_n+O(\tau^2)
\end{align*}
So \begin{align*}
    \frac{y^{n+1}-y^n}{\tau}&-f(t_{n+\theta},y^{n+\theta})\\ &=\frac{\tau y'_n+\frac{1}{2}\tau^2y''_n+O(\tau^3)}{\tau}-y'_n-\theta\tau y''_n+O(\tau^2)\\
    &= \frac{1}{2}\tau y''_n-\tau\theta y''_n+O(\tau^2)\\
    &= \Big(\frac{1}{2}\tau-\tau\theta\Big) y''_n+O(\tau^2)\\
\end{align*}
If $\theta = \frac{1}{2}+\tau$, we get second-order accuracy.\\ 
}
\end{confidential}

The next proposition is the main result of this section: a uniform bound of the semidiscrete-in-time velocity $u^n$ in $H$ in terms of the initial data $u_0$ and the forcing term $f$.

\begin{proposition}\label{prop:L2bound}
Assume that the time-step condition \eqref{eq:step-condition} holds.
Then
\begin{align} 
\label{eq:energy2}
\| u^{n} \|^2 
&
\leq 
e^{ - \frac{1}{15} \lambda_1 \nu( 2\theta - 1)n\tau }
\| u_{0} \|^2
+
3 \frac{1}{\lambda_1\nu^2} \frac{ 1 }{ ( 2 \theta - 1 ) }
  \| f \|^2_{\infty} 
 \qquad \forall n\geq 0
,
\end{align}
and there is a constant 
\[K_1(\|u_0\|,\|f\|_{\infty}) := \|u_0\|^2+ 3 \frac{1}{\lambda_1\nu^2} \frac{ 1 }{ ( 2 \theta - 1 ) }
  \|f\|^2_{\infty}\] such that 
\begin{equation}
\label{eq:energy2.1}
    \|u^{n}\|^2\leq K_1 \qquad \forall n\geq 0.
\end{equation}
\end{proposition}

\begin{proof}
    From \eqref{eq:energy0} and \eqref{eq:energy1} we have
\begin{confidential}
    {\color{red}
    \begin{align*}
&
(\alpha + \varepsilon) \| u^{n+1} \|^2  - \alpha \| u^{n} \|^2 + \| a u^{n+1} - b u^n \|^2
\\
& 
=
\Big( 1 + \lambda_1 \nu \theta \tau \Big) \|u^{n+\theta}\|^2 - \|u^n\|^2 + \|u^{n+\theta} - u^n\|^2
\\
& 
\leq
\frac{1}{\nu}\theta\tau  \| f\|_{\infty}^2
,
\end{align*}
hence
\begin{align*}
&
\| u^{n+1} \|^2 
\leq
\frac{\alpha}{\alpha + \varepsilon} \| u^{n} \|^2
+
\frac{1}{\nu} \frac{1}{\alpha + \varepsilon}  \theta \tau  \| f\|_{\infty}^2
,
\end{align*}
and therefore
    }
\end{confidential}
\begin{confidential}
    {\color{red}
\begin{align*}
&
(\alpha + \varepsilon) \| u^{n+1} \|^2  - \alpha \| u^{n} \|^2 + \| a u^{n+1} - b u^n \|^2
\leq
\frac{1}{\nu}\theta\tau  \| f\|_{\infty}^2
,
\end{align*}
hence
    }    
\end{confidential}
\begin{align*}
&
\| u^{n+1} \|^2 
\leq
\frac{1}{1 + \varepsilon_\theta / \alpha_\theta} \| u^{n} \|^2
+
\frac{1}{\nu} \frac{1}{\alpha_\theta + \varepsilon_\theta}  \theta \tau  \| f\|_{\infty}^2,
\end{align*}
and inductively,
\begin{confidential}
{\color{red}
\begin{align}
\| u^{n+1} \|^2 
&
\leq
\frac{1}{1 + \varepsilon / \alpha} \| u^{n} \|^2
+
\frac{1}{\nu} \frac{1}{\alpha + \varepsilon}  \theta \tau  \| f\|_{\infty}^2
\notag
\\
&
\leq
\frac{1}{1 + \varepsilon / \alpha} 
\Big(
\frac{1}{1 + \varepsilon / \alpha } \| u^{n-1} \|^2
+
\frac{1}{\nu} \frac{1}{\alpha + \varepsilon}  \theta \tau  \| f \|^2_{\infty}
\Big)
+
\frac{1}{\nu} \frac{1}{\alpha + \varepsilon}  \theta \tau  \| f\|_{\infty}^2
\notag
\\
&
=
\frac{1}{1 + \varepsilon / \alpha} 
\frac{1}{1 + \varepsilon / \alpha } \| u^{n-1} \|^2
\notag
\\
& 
\quad
+
\frac{1}{1 + \varepsilon / \alpha} 
\frac{1}{\nu} \frac{1}{\alpha + \varepsilon}  \theta \tau  \| f \|^2_{\infty}
+
\frac{1}{\nu} \frac{1}{\alpha + \varepsilon}  \theta \tau  \| f\|_{\infty}^2
\notag
\\
&
\leq
\frac{1}{1 + \varepsilon / \alpha} 
\frac{1}{1 + \varepsilon / \alpha } 
\Big(
\frac{1}{1 + \varepsilon / \alpha} \| u^{n-2} \|^2
+
\frac{1}{\nu} \frac{1}{\alpha + \varepsilon}  \theta \tau  \| f \|^2_{\infty}
\Big)
\notag
\\
& 
\quad
+
\frac{1}{1 + \varepsilon / \alpha} 
\frac{1}{\nu} \frac{1}{\alpha + \varepsilon}  \theta \tau  \| f \|^2_{\infty}
+
\frac{1}{\nu} \frac{1}{\alpha + \varepsilon}  \theta \tau  \| f\|_{\infty}^2
\notag
\\
&
=
\frac{1}{1 + \varepsilon / \alpha} 
\frac{1}{1 + \varepsilon / \alpha } 
\frac{1}{1 + \varepsilon / \alpha} \| u^{n-2} \|^2
\notag
\\
&
\quad
+
\frac{1}{1 + \varepsilon / \alpha} 
\frac{1}{1 + \varepsilon / \alpha } 
\frac{1}{\nu} \frac{1}{\alpha + \varepsilon}  \theta \tau  \| f \|^2_{\infty}
\notag
\\
& 
\quad
+
\frac{1}{1 + \varepsilon / \alpha} 
\frac{1}{\nu} \frac{1}{\alpha + \varepsilon}  \theta \tau  \| f \|^2_{\infty}
\notag
\\
& 
\quad
+
\frac{1}{\nu} \frac{1}{\alpha + \varepsilon}  \theta \tau  \| f\|_{\infty}^2
\notag
\\
& 
=
\\
&\qquad \vdots 
\notag
\\
&
\label{eq:induction conclusion}
\leq
\Bigg(\frac{1}{1 + \varepsilon / \alpha}\Bigg)^{n+1}
\| u_{0} \|^2
\\
&
\quad
+
\Bigg(\frac{1}{1 + \varepsilon / \alpha} \Bigg)^n
\frac{1}{\nu} \frac{1}{\alpha + \varepsilon}  \theta \tau  \| f \|_{\infty}^2
\notag
\\
& 
\qquad
\vdots
\notag
\\
&
\quad
+
\Bigg(\frac{1}{1 + \varepsilon / \alpha} \Bigg)^2 
\frac{1}{\nu} \frac{1}{\alpha + \varepsilon}  \theta \tau  \| f \|^2_{\infty}
\notag
\\
& 
\quad
+
\frac{1}{1 + \varepsilon / \alpha} 
\frac{1}{\nu} \frac{1}{\alpha + \varepsilon}  \theta \tau  \| f \|^2_{\infty}
\notag
\\
& 
\quad
+
\frac{1}{\nu} \frac{1}{\alpha + \varepsilon}  \theta \tau  \| f\|_{\infty}^2
\\
&
=
\Big(\frac{1}{1 + \varepsilon / \alpha} \Big)^{n+1}
\| u_{0} \|^2
\\
&
+
\sum_{j=0}^{n} \Big(\frac{1}{1 + \varepsilon / \alpha} \Big)^{n-j} \frac{1}{\nu} \frac{1}{\alpha+\varepsilon} \theta \tau  \| f \|^2_{\infty},
%\notag
\end{align}
therefore
}
\end{confidential}
\begin{align*}
\| u^{n+1} \|^2 
&
\leq
\Big(\frac{1}{1 + \varepsilon_\theta / \alpha_\theta} \Big)^{n+1}
\| u_{0} \|^2
+
\sum_{j=1}^{n+1} \Big(\frac{1}{1 + \varepsilon_\theta / \alpha_\theta} \Big)^{n-j+1} \frac{1}{\nu} \frac{1}{\alpha_\theta+\varepsilon_\theta} \theta \tau  \| f \|^2_{\infty}.
\end{align*}
\begin{confidential}
    {\color{red}
    Recall that from \eqref{eq:step-condition} and \eqref{eq:alpha-epsilon-ab-bounds} we have that
\begin{align*}
\varepsilon / \alpha 
< 2 \varepsilon 
= 2 \lambda_1  \nu (2 \theta - 1)\tau
\leq 2 \lambda_1  \nu \tau
\leq 2,
\end{align*}
    }
\end{confidential}
\noindent
Now let $g(x) := e^{- \frac{1}{10} x} - \frac{1}{1+x}$, with $g'(x)\geq 0$ for all $x\in[0,2]$, $g(0)=0$, and $g(x)\geq 0$ for all $x\in [0,2]$. Then by \eqref{1star} we have 
 \begin{align*}
\frac{1}{1 + \varepsilon_\theta / \alpha_\theta} 
\leq e^{-\frac{1}{10} \frac{\varepsilon_\theta }{ \alpha_\theta } }
\leq e^{-\frac{1}{10}\frac{2}{3} \varepsilon_\theta  }
= e^{-\frac{1}{15} \varepsilon_\theta  }.
 \end{align*}
\begin{confidential}
 {\color{red}
 \begin{align*}
& 
g(x) := e^{- \mu x} - \frac{1}{1+x}, \qquad \forall x\in [0,2],
\\
& 
g(0) = 0,
\\
&
g' (x) 
= - \mu e^{- \mu x} + \frac{1}{(1+x)^2}.
 \end{align*}
 I need to pick $\mu$ such that $g'(x)\geq 0$ for all $x\in[0,2]$. Let's pick $\mu=\frac{1}{10}$. I need to show 
 \[\frac{1}{10}e^{-\frac{1}{10}x}\leq\frac{1}{(1+x)^2}\] 
 for all $x\in[0,2]$.
 \begin{align*}
     &\frac{1}{10}e^{-\frac{1}{10}x}\leq\frac{1}{(1+x)^2}\\
     &\iff e^{-\frac{1}{10}x}\leq\frac{10}{(1+x)^2}\\
     &\iff -\frac{1}{10}x\leq\ln\Big(\frac{10}{(1+x)^2}\Big)\\
     &\iff -x\leq10\ln\Big(\frac{10}{(1+x)^2}\Big)\\
     &\iff -10\ln\Big(\frac{10}{(1+x)^2}\Big)\leq x\\
 \end{align*}
 Since $x$ is at worst 0, it suffices to show that $-10\ln\Big(\frac{10}{(1+x)^2}\Big)\leq0$. It suffices to show that $\ln\Big(\frac{10}{(1+x)^2}\Big)\geq0$.
 \begin{align*}
     &\ln\Big(\frac{10}{(1+x)^2}\Big)\geq0\\
     &\iff \frac{10}{(1+x)^2}\geq 1\\
     &\iff 10\geq (1+x)^2\\
 \end{align*}
 And this is true for all $x\in[0,2]$. So we can pick $\mu=\frac{1}{10}$.
 Now since $g'(x)\geq 0$, $g(x)$ is increasing. Note that $g'(x)\geq 0$ for all $x\in[0,2]$ and $g(0)=0$. Thus, $g(x)\geq 0$ for all $x\in [0,2]$. So we get
 \begin{align*}
\frac{1}{1 + \varepsilon / \alpha} 
\leq e^{-\frac{1}{10} \frac{\varepsilon }{ \alpha } }
\leq e^{-\frac{1}{10}\frac{2}{3} \varepsilon  }
= e^{-\frac{2}{30} \varepsilon  }
,
 \end{align*}
 (Since \[\alpha\leq\frac{3}{2}\Rightarrow\frac{2}{3}\varepsilon\leq\frac{\varepsilon}{\alpha}\] 
 and $e^{-\nu x}$ is decreasing).
 }
\end{confidential}
\noindent 
Therefore,
\begin{align*}
\| u^{n+1} \|^2 
&
\leq
\Big(\frac{1}{1 + \varepsilon_\theta / \alpha_\theta} \Big)^{n+1}
\| u_{0} \|^2
+
\sum_{j=1}^{n+1} \Big(\frac{1}{1 + \varepsilon_\theta / \alpha_\theta} \Big)^{n-j+1} \frac{1}{\nu} \frac{1}{\alpha_\theta+\varepsilon_\theta} \theta \tau  \| f \|^2_{\infty}
\\
&
\leq
e^{ - \frac{1}{15} (n+1)\varepsilon_\theta}
\| u_{0} \|^2
+
\frac{1}{\nu}\tau  \| f \|^2_{\infty}  
\Big(\sum_{j=0}^n \Big(  \frac{1}{1 + 2/3\varepsilon_\theta } \Big)^{j} \Big) \frac{1}{1/2+2/3\varepsilon_\theta}
\\
& 
=
e^{ - \frac{1}{15} (n+1)\varepsilon_\theta}
\| u_{0} \|^2
\\
& \quad
+
\frac{1}{\nu}\tau  \| f \|^2_{\infty}  
\frac{1}{1/2+2/3\varepsilon_\theta}
\Big( 1- \Big(\frac{1}{1 + 2/3\varepsilon_\theta } \Big)^{n+1} \Big)
\Bigg(\frac{1}{1-\frac{1}{1+2/3\varepsilon_\theta}}\Bigg)
\\
& 
\leq
e^{ - \frac{1}{15} (n+1)\varepsilon_\theta}
\| u_{0} \|^2
+
\frac{1}{\nu}\tau  \| f \|^2_{\infty}  
\Big( 1- \Big(\frac{1}{1 + 2/3\varepsilon_\theta } \Big)^{n+1} \Big)
\Bigg(\frac{2}{2/3\varepsilon_\theta}\Bigg)
\\
& 
\leq
e^{ - \frac{1}{15} (n+1)\tau\nu\lambda_1 (2\theta-1)}
\| u_{0} \|^2
+
3\frac{1}{\lambda_1\nu^2}\frac{1}{2\theta-1}  \| f \|^2_{\infty}, 
\end{align*}
\begin{confidential}
{\color{red}
\begin{align*}
\| u^{n+1} \|^2 
&
\leq
\Big(\frac{1}{1 + \varepsilon / \alpha} \Big)^{n+1}
\| u_{0} \|^2
\\
&
+
\sum_{j=1}^{n+1} \Big(\frac{1}{1 + \varepsilon / \alpha} \Big)^{n-j+1} \frac{1}{\nu} \frac{1}{\alpha+\varepsilon} \theta \tau  \| f \|^2_{\infty}
\\
&
\leq
e^{ - \frac{2}{30} (n+1)\varepsilon}
\| u_{0} \|^2
\\
& \quad
+
 \| f \|^2_{\infty} 
\sum_{j=1}^{n+1} \Big(  \frac{1}{1 + \varepsilon / \alpha} \Big)^{n-j+1} \frac{1}{\nu} \frac{1}{\alpha+\varepsilon}
 \tau 
\\
&
\leq
e^{ - \frac{2}{30} (n+1)\varepsilon}
\| u_{0} \|^2
\\
& \quad
+
 \| f \|^2_{\infty} 
\sum_{j=1}^{n+1} \Big(  \frac{1}{1 + \varepsilon / \alpha} \Big)^{n-j+1} \frac{1}{\nu} \frac{1}{1/2+\varepsilon}
 \tau 
\\
& 
=
e^{ - \frac{1}{15} (n+1)\varepsilon}
\| u_{0} \|^2
\\
& \quad
+
\frac{1}{\nu}\tau  \| f \|^2_{\infty} 
\Big(\sum_{j=1}^{n+1} \Big(  \frac{1}{1 + \varepsilon / \alpha} \Big)^{n-j+1}  \frac{1}{1/2+\varepsilon}
 \Big)
\\
& 
\leq
e^{ - \frac{1}{15} (n+1)\varepsilon}
\| u_{0} \|^2
\\
& \quad
+
\frac{1}{\nu}\tau  \| f \|^2_{\infty}  
\Big(\sum_{j=1}^{n+1} \Big(  \frac{1}{1 + 2/3\varepsilon } \Big)^{n-j+1}\Big) \frac{1}{1/2+2/3\varepsilon}
\\
& 
=
e^{ - \frac{1}{15} (n+1)\varepsilon}
\| u_{0} \|^2
\\
& \quad
+
\frac{1}{\nu}\tau  \| f \|^2_{\infty}  
\Big(\sum_{j=0}^n \Big(  \frac{1}{1 + 2/3\varepsilon } \Big)^{j} \Big) \frac{1}{1/2+2/3\varepsilon}
\\
& 
\text{Note that  } \sum_{i=0}^nx^i=\frac{1-x^{n+1}}{1-x} \text{ for } x\neq 1\\
& 
=
e^{ - \frac{1}{15} (n+1)\varepsilon}
\| u_{0} \|^2
\\
& \quad
+
\frac{1}{\nu}\tau  \| f \|^2_{\infty}  
\frac{1}{1/2+2/3\varepsilon}
\Big( 1- \Big(\frac{1}{1 + 2/3\varepsilon } \Big)^{n+1} \Big)
\Bigg(\frac{1}{1-\frac{1}{1+2/3\varepsilon}}\Bigg)
\\
& 
=
e^{ - \frac{1}{15} (n+1)\varepsilon}
\| u_{0} \|^2
\\
& \quad
+
\frac{1}{\nu}\tau  \| f \|^2_{\infty}  
\frac{1}{1/2+2/3\varepsilon}
\Big( 1- \Big(\frac{1}{1 + 2/3\varepsilon } \Big)^{n+1} \Big)
\Bigg(\frac{1+2/3\varepsilon}{2/3\varepsilon}\Bigg)
\\
& 
\leq
e^{ - \frac{1}{15} (n+1)\varepsilon}
\| u_{0} \|^2
\\
& \quad
+
\frac{1}{\nu}\tau  \| f \|^2_{\infty}  
\Big( 1- \Big(\frac{1}{1 + 2/3\varepsilon } \Big)^{n+1} \Big)
\Bigg(\frac{2}{2/3\varepsilon}\Bigg)
\\
& 
=
e^{ - \frac{1}{15} (n+1)\tau\nu\lambda_1 (2\theta-1)}
\| u_{0} \|^2
\\
& \quad
+
\frac{1}{\nu}\tau  \| f \|^2_{\infty}  
\Big( 1- \Big(\frac{1}{1 + 2/3\tau\nu\lambda_1 (2\theta-1) } \Big)^{n+1} \Big)
\Bigg(\frac{2}{2/3\tau\nu\lambda_1 (2\theta-1)}\Bigg)
\\
& 
=
e^{ - \frac{1}{15} (n+1)\tau\nu\lambda_1 (2\theta-1)}
\| u_{0} \|^2
\\
& \quad
+
3\frac{1}{\lambda_1\nu^2}\frac{1}{2\theta-1}  \| f \|^2_{\infty}  
\Big( 1- \Big(\frac{1}{1 + 2/3\tau\nu\lambda_1 (2\theta-1) } \Big)^{n+1} \Big)
\\
& 
\leq
e^{ - \frac{1}{15} (n+1)\tau\nu\lambda_1 (2\theta-1)}
\| u_{0} \|^2
\\
& \quad
+
3\frac{1}{\lambda_1\nu^2}\frac{1}{2\theta-1}  \| f \|^2_{\infty} 
\end{align*}
Hence,  \eqref{eq:energy2} is satisfied.\\
}
\end{confidential}
\noindent
which yields \eqref{eq:energy2} and \eqref{eq:energy2.1}.
\end{proof}

The following bound, which follows by direct computation from \eqref{eq:energy2}, proves the existence of an absorbing set in $H$.

\begin{corollary}\label{cor:absorbing}
Under the small time-step assumption \eqref{eq:step-condition}, provided the computational time interval $n\tau$ is sufficiently large
\begin{align*}
& 
 n\tau
\geq 
15 \frac{1}{\lambda_1}
\frac{1}{\nu} \frac{1}{( 2\theta - 1)}
\ln \bigg(
 \nu^2\lambda_1 ( 2 \theta - 1 ) 
\frac{ \| u_{0} \|^2 }{  \| f \|^2_{\infty} }
\bigg)
,
\end{align*}
we have that 
\begin{align}
\label{eq:energy6} 
\| u^{n} \|^2 
&
\leq 
4 \frac{1}{\lambda_1\nu^2} \frac{ 1 }{( 2 \theta - 1 ) }
  \| f \|^2_{\infty} 
.
\end{align}
\end{corollary}

\begin{confidential}
{\color{red}  
\begin{proof}

From \eqref{eq:energy2} we have
\begin{align*} 
\| u^{n+1} \|^2 
&
\leq 
e^{ - \frac{1}{15} \lambda_1 \nu ( 2\theta - 1) 
(n+1)\tau }
\| u_{0} \|^2
\\
& \quad
+
3 \frac{1}{\lambda_1\nu^2} \frac{ 1 }{ ( 2 \theta - 1 ) }
  \| f \|^2_{\infty} 
\end{align*}
If $\|u_0\|=0$, the conclusion of the corollary is automatically true. Otherwise, it is true provided
\begin{align*} 
& 
e^{ - \frac{1}{15} \lambda_1 \nu ( 2\theta - 1) 
(n+1)\tau }
\| u_{0} \|^2
\leq 
 \frac{1}{\lambda_1\nu^2} \frac{ 1 }{ ( 2 \theta - 1 ) }
  \| f \|^2_{\infty} 
,
\end{align*}
or
\begin{align*} 
& 
e^{ - \frac{1}{15} \lambda_1 \nu ( 2\theta - 1) 
(n+1)\tau }
\leq 
\frac{1}{\| u_{0} \|^2}
 \frac{1}{\lambda_1\nu^2} \frac{ 1 }{( 2 \theta - 1 ) }
  \| f \|^2_{\infty} 
,
\\
& 
e^{ \frac{1}{15} \lambda_1 \nu ( 2\theta - 1) 
(n+1)\tau }
\geq 
 \nu^2\lambda_1 ( 2 \theta - 1 ) 
\frac{ \| u_{0} \|^2 }{  \| f \|^2_{\infty} }
,
\\
& 
\frac{1}{15} \lambda_1 \nu ( 2\theta - 1) (n+1)\tau 
\geq 
\ln \bigg(
 \nu^2\lambda_1 ( 2 \theta - 1 ) 
\frac{ \| u_{0} \|^2 }{  \| f \|^2_{\infty} }
\bigg)
,
\\
& 
 (n+1)\tau
\geq 
15 \frac{1}{\lambda_1}
\frac{1}{\nu ( 2\theta - 1)}
\ln \bigg(
 \nu^2\lambda_1 ( 2 \theta - 1 ) 
\frac{ \| u_{0} \|^2 }{  \| f \|^2_{\infty} }
\bigg)
,
\end{align*}
i.e., 
\begin{align*} 
\| u^{n+1} \|^2 
&
\leq 
4 \frac{1}{\lambda_1\nu^2} \frac{ 1 }{( 2 \theta - 1 ) }
  \| f \|^2_{\infty} 
,
\notag
\end{align*}
provided
\begin{align*}
& 
 (n+1)\tau
\geq 
15 \frac{1}{\lambda_1}
\frac{1}{\nu} \frac{1}{( 2\theta - 1)}
\ln \bigg(
 \nu^2\lambda_1 ( 2 \theta - 1 ) 
\frac{ \| u_{0} \|^2 }{  \| f \|^2_{\infty} }
\bigg)
.
\end{align*}

The bound \eqref{eq:energy6} follows directly from \eqref{eq:energy2}.
\end{proof}
}
\end{confidential}

The next two lemmas are needed in Section \ref{sec:H1bound} to prove the existence of an absorbing set in $V$. First, we derive an $L^2(H^1(\Omega))$ bound on the fractional time values $u^{n+\theta}$.

\begin{lemma}\label{lemma:gradutheta}
    Assume that the time-step restriction \eqref{eq:step-condition} holds. Then we have the following bounds
    \begin{align}
\label{eq:energy3}
& 
\nu\tau \sum_{i=0}^{n} \| \nabla u^{i+\theta}\|^2 
\leq
\|u_0\|^2  
+ 
\frac{1}{\nu}  (n+1)\tau  \frac{1}{\lambda_1}\| f \|_{\infty}^2
\qquad \forall n\geq 0
,
\end{align}
and for any integer $p\geq0$

\begin{align}
\label{eq:energy4}
& 
\nu\tau \sum_{i=n}^{n+p} \| \nabla u^{i+\theta}\|^2 
\leq
\|u^n\|^2  
+ 
\frac{1}{\nu}  (p+1)\tau  \frac{1}{\lambda_1}\| f \|_{\infty}^2 
\qquad \forall n\geq 0
.
\end{align}
\end{lemma}

\begin{proof}

\begin{confidential}
{\color{red}
    Recall what we are trying to solve, where the linear extrapolation is equivalent to a Forward Euler step:
\begin{align}\label{eq:BE}
    \frac{1}{\theta\tau}(u^{n+\theta}-u^n) 
    - \nu\Delta u^{n+\theta}
    + u^{n+\theta}\cdot\nabla u^{n+\theta} 
    + \nabla p^{n+\theta} &= f^{n+\theta}  \quad \text{in }\Omega \\
    \nabla\cdot u^{n+\theta} &= 0 \quad \text{in }\Omega\\
    u^{n+\theta}=0 \quad \text{on } \partial\Omega
\end{align}
\begin{equation}\label{eq:FE}
    \frac{1}{(1-\theta)\tau}(u^{n+1}-u^{n+\theta})-\nu\Delta u^{n+\theta}+u^{n+\theta}\cdot\nabla u^{n+\theta}+\nabla p^{n+\theta}=f^{n+\theta} \\    
\end{equation}
}
\end{confidential}
Testing \eqref{eq:BENSE} with $\theta\tau u^{n+\theta}$, and \eqref{eq:FENSE} with $(1-\theta)\tau u^{n+\theta}$ in $H$, using the polarization identity, and adding both equations yields

\begin{confidential}
    {\color{red}
\[\frac{1}{2}\left(\|u^{n+\theta}\|^2-\|u^{n}\|^2+\|u^{n+\theta}-u^{n}\|^2\right)
+\nu\theta\tau\|\nabla u^{n+\theta}\|^2
=\theta\tau( f^{n+\theta},u^{n+\theta})\]
and
\[\frac{1}{2}\left(\|u^{n+1}\|^2-\|u^{n+\theta}\|^2-\|u^{n+1}-u^{n+\theta}\|^2\right)
+\nu(1-\theta)\tau\|\nabla u^{n+\theta}\|^2
=(1-\theta)\tau( f^{n+\theta},u^{n+\theta}).\]
}
\end{confidential}

\[
\frac{1}{2}\Big(\|u^{n+1}\|^2-\|u^{n}\|^2\Big)
+\frac{1}{2}\Big(\|u^{n+\theta}-u^{n}\|^2-\|u^{n+1}-u^{n+\theta}\|^2\Big)
+\nu\tau\|\nabla u^{n+\theta}\|^2
=\tau( f^{n+\theta},u^{n+\theta}).
\]
The following is a key step to prove the long-time stability of the Cauchy one-leg $\theta$-method \eqref{eq:STAR}. We recall that, by \eqref{eq:FENSE}, 
$u^{n+\theta}=\theta u^{n+1}+(1-\theta)u^n$. Therefore,
\begin{confidential}
    {\color{red}
    \begin{align*}
    &\|u^{n+\theta}-u^{n}\|^2-\|u^{n+1}-u^{n+\theta}\|^2\\
    &= \|\theta u^{n+1}+(1-\theta)u^n-u^{n}\|^2-\|u^{n+1}-\theta u^{n+1}-(1-\theta)u^n\|^2\\
    &= \theta^2\|u^{n+1}-u^{n}\|^2-(1-\theta)^2\|u^{n+1}-u^{n}\|^2\\
    &= [\theta^2-(1-\theta)^2]\|u^{n+1}-u^{n}\|^2\\
    &= (2\theta-1)\|u^{n+1}-u^{n}\|^2
\end{align*}
    }
\end{confidential}
\begin{align*}
    \|u^{n+\theta}-u^{n}\|^2-\|u^{n+1}-u^{n+\theta}\|^2
    &= \|\theta u^{n+1}+(1-\theta)u^n-u^{n}\|^2-\|u^{n+1}-\theta u^{n+1}-(1-\theta)u^n\|^2\\
    &= \theta^2\|u^{n+1}-u^{n}\|^2-(1-\theta)^2\|u^{n+1}-u^{n}\|^2\\
    &= (2\theta-1)\|u^{n+1}-u^{n}\|^2,
\end{align*}
which, by using the H\"{o}lder, Young, and Poincar\'{e}-Friedrichs inequalities yields
\begin{confidential}
    {\color{red}
    \begin{align*}
& 
\frac{1}{2} \big( \|u^{n+1}\|^2 - \|u^n\|^2 \big) 
+ \frac{1}{2} (2\theta -1) \|u^{n+1} - u^n\|^2 
+ \nu \tau \| \nabla u^{n+\theta}\|^2 
= \tau ( f^{n+\theta} , u^{n+\theta} )
\\
& 
\leq
\tau \| f^{n+\theta}\|  \| u^{n+\theta}\|
\leq
\frac{1}{2}\nu \tau \|\nabla u^{n+\theta}\|^2
+ \frac{1}{2\nu} \tau \frac{1}{\lambda_1}\| f^{n+\theta}\|^2 
,
\end{align*}
hence
    }
\end{confidential}
\begin{align}
\label{eq:energy5}
\|u^{n+1}\|^2 - \|u^n\|^2 
+  (2\theta -1) \|u^{n+1} - u^n\|^2 
+ \nu \tau \| \nabla u^{n+\theta}\|^2 
&\leq
\frac{1}{\nu} \tau \frac{1}{\lambda_1}\| f^{n+\theta}\|^2 \leq \frac{1}{\nu} \tau  \frac{1}{\lambda_1}\| f\|_{\infty}^2.
\end{align}
Summation in $n$ implies \eqref{eq:energy3}, and similarly, summation for $i=n:n+p$, gives \eqref{eq:energy4}.
\begin{confidential}
{\color{red}
\begin{align*}
& 
\|u^{n+1}\|^2  
+ \sum_{i=0}^{n} (2\theta -1) \|u^{i+1} - u^i\|^2 
+ \nu \sum_{i=0}^{n} \tau \| \nabla u^{i+\theta}\|^2 
\\
& 
\leq
\|u_0\|^2  
+ 
\frac{1}{\nu} \sum_{i=0}^{n}  \tau  \frac{1}{\lambda_1}\| f \|_{\infty}^2 
\\
& 
=
\|u_0\|^2  
+ 
\frac{1}{\nu}(n+1)\tau  \frac{1}{\lambda_1}\| f \|_{\infty}^2  
\end{align*}
\begin{align*}
& 
\|u^{n+1}\|^2  
+ \sum_{i=0}^{n} (2\theta -1) \|u^{i+1} - u^i\|^2 
+ \nu \sum_{i=0}^{n} \tau \| \nabla u^{i+\theta}\|^2 
\leq
\|u_0\|^2  
+ 
\frac{1}{\nu}  (n+1)\tau  \frac{1}{\lambda_1}\| f \|_{\infty}^2,
\end{align*}
$ \forall n+1 \geq 0$, or
}
\end{confidential}
\begin{confidential}
{\color{red}
\begin{align*}
& 
\|u^{n+p+1}\|^2  
+ \sum_{i=n}^{n+p} (2\theta -1) \|u^{i+1} - u^i\|^2 
+ \nu \sum_{i=n}^{n+p} \tau \| \nabla u^{i+\theta}\|^2 
\\
& 
\leq
\|u^n\|^2  
+ 
\frac{1}{\nu} \tau (p+1) \frac{1}{\lambda_1}\| f \|_{\infty}^2 
\end{align*}

\begin{align*}
& 
\|u^{n+p+1}\|^2  
+ \sum_{i=n}^{n+p} (2\theta -1) \|u^{i+1} - u^i\|^2 
+ \nu \sum_{i=n}^{n+p} \tau \| \nabla u^{i+\theta}\|^2 
\leq
\|u^n\|^2  
+ 
\frac{1}{\nu}  (p+1)\tau  \frac{1}{\lambda_1}\| f \|_{\infty}^2,
\end{align*}
$\forall n+1 \geq 0$, 
or

\eqref{eq:energy4}.
}
\end{confidential}
\end{proof}

Secondly, we derive an $L^2(H^1(\Omega))$ bound on the integer time values $u^n$. We introduce the following notation 
\begin{align*}
    C_1 &:= \left(2K_1^2(2\theta^2-2\theta+1)\right)^{1/4},
    \\
    K_2  &:= \left(\frac{1}{2}+\left|\frac{C_1^4(1-\theta)^4\theta^3}{\nu^3(2\theta-1)^3}
    -\frac{\nu}{2\theta}\right|\frac{1}{\nu}\right)\|u_0\|^2  
    + \frac{\tau\nu}{8\theta}(4\theta^2-6\theta+3)\|\nabla u_0\|^2,
    \\
    K_3  &:= \left(\frac{4\theta }{\nu (2\theta-1)}+\left|\frac{C_1^4(1-\theta)^4\theta^3}{\nu^3(2\theta-1)^3}-\frac{\nu}{2\theta}\right|\frac{1}{\nu^2}\right)\frac{1}{\lambda_1}\|f\|_{\infty}^2,\\
    K_4 &:= \frac{1}{2}+ \left|\frac{C_1^4(1-\theta)^4\theta^3}{\nu^3(2\theta-1)^3}-\frac{\nu}{2\theta}\right|\frac{1}{\nu}.
\end{align*}

\begin{lemma}\label{lemma:gradu1}
Assume that the time-step restriction \eqref{eq:step-condition} holds. Then we have
    \begin{align}\label{eq:boundnablaunp1}
    &\frac{\nu(2\theta-1)}{16\theta}\tau\sum_{i=0}^{n}\|\nabla u^{i+1}\|^2
    \leq
     K_2  + (n+1)\tau K_3 \qquad \forall n\geq 0,
    \end{align}
    and for any integer $p\geq0$
    \begin{align}\label{eq:boundnablaunp1.1}
    \frac{\nu(2\theta-1)}{16\theta}\tau\sum_{i=n}^{n+p}\|\nabla u^{i+1}\|^2
    \leq
    4 K_4 K_1 + K_3(p+1)\tau + \frac{\tau\nu}{8\theta}(4\theta^2-6\theta+3)\|\nabla u^n\|^2 \qquad \forall n\geq 0.
\end{align}
\end{lemma}

\begin{proof}

Multiplying \eqref{eq:BENSE} by $\theta\tau$ and \eqref{eq:FENSE} by $(1-\theta)\tau$ and adding yields
\begin{confidential}
{\color{red}
\begin{align} 
u^{n+\theta}-u^n-\theta\tau\nu\Delta u^{n+\theta}+\theta\tau u^{n+\theta}\cdot
\nabla u^
{n+\theta}+\theta\tau\nabla p^{n+\theta}=\theta \tau f^{n+\theta}
\end{align}
\begin{align}
u^{n+1}-u^{n+\theta}-(1-\theta)\tau\nu\Delta u^{n+\theta}+(1-\theta)\tau u^{n+\theta}\cdot
\nabla u^
{n+\theta}+(1-\theta)\tau \nabla p^{n+\theta}=(1-\theta)\tau f^{n+\theta}
\end{align}
}
\end{confidential}
\begin{align} 
u^{n+1}-u^n-\tau\nu\Delta u^{n+\theta}+\tau u^{n+\theta}\cdot
\nabla u^
{n+\theta}+\tau\nabla p^{n+\theta}= \tau f^{n+\theta}.
\end{align}
Testing with $u^{n+1}$ in $H$ and using the polarization identity we have
\begin{align}
\label{eq:2.52}
     \frac{1}{2}\|u^{n+1}\|^2-&\frac{1}{2}\|u^n\|^2 + \frac{1}{2} \|u^{n+1}-u^n\|^2 +\nu\tau(\nabla u^{n+\theta},\nabla u^{n+1})
    \\
    &
    \notag
    = -\tau b(u^{n+\theta},u^{n+\theta}, u^{n+1}) + \tau ( f^{n+\theta},u^{n+1}).
\end{align}

\begin{confidential}
{\color{red}
Note we have $\nabla\cdot u^{n+\theta}=0$, and $u^{n+\theta}=\theta u^{n+1}+(1-\theta)u^n$. So $\nabla\cdot u^{n+1}=\frac{\theta-1}{\theta}\nabla\cdot u^n$. So since $\nabla\cdot u_0 = 0$, we must have that $\nabla\cdot u^n=0$ for all $n$. The same argument holds to argue that $u^n=0$ on the boundary.

\begin{align}
    \|u^{n+1}\|^2
    -( u^n, u^{n+1})
    -\tau\nu(\Delta u^{n+\theta}, u^{n+1}) +\tau b(u^{n+\theta},u^{n+\theta}, u^{n+1}) + \tau ( \nabla p^{n+\theta}, u^{n+1}) = \tau ( f^{n+\theta}, u^{n+1} )
\end{align}

Let's look at each term of this expression:

\[
-( u^n, u^{n+1}) = -\frac{1}{2}\|u^n\|^2-\frac{1}{2}\|u^{n+1}\|^2+\frac{1}{2}\|u^{n+1}-u^n\|^2
\]

\[( \nabla p^{n+\theta},u^{n+1}) = ( p^{n+\theta}, \nabla\cdot u^{n+1}) = 0\]

}
\end{confidential}

We now explain how to handle each of the last three terms of the previous expression. 

The first key step in the proof deals with the kinematic dissipation term $\nu\tau(\nabla u^{n+\theta},\nabla u^{n+1})$, which we rewrite
\begin{align*}
    (1-\theta)( \nabla u^{n},\nabla u^{n+1}) = \frac{1}{2\theta} \left(-\theta^2\|\nabla u^{n+1}\|^2-(1-\theta)^2\|\nabla u^{n}\|^2+\|\theta\nabla u^{n+1}+(1-\theta)\nabla u^n\|^2\right),
\end{align*}
\begin{confidential}
{\color{red}
Expanding the RHS:

\begin{align}
    \frac{1}{2\theta} &\left(-\theta^2\|\nabla u^{n+1}\|^2-(1-\theta)^2\|\nabla u^{n}\|^2+\|\theta\nabla u^{n+1}+(1-\theta)\nabla u^n\|^2\right)
    \\
    &
    =
    \frac{1}{2\theta}\left[-\theta^2\|\nabla u^{n+1}\|^2-(1-\theta)^2\|\nabla u^{n}\|^2+\theta^2\|\nabla u^{n+1}\|^2+(1-\theta)^2\|\nabla u^{n}\|^2+2\theta(1-\theta)( \nabla u^n,\nabla u^{n+1})\right]
    \\
    &
    =
    \frac{1}{2\theta}2\theta(1-\theta)( \nabla u^n,\nabla u^{n+1})
    \\
    &
    =(1-\theta)( \nabla u^n,\nabla u^{n+1})
\end{align}
}
\end{confidential}
and therefore,
\begin{align}
     ( \nabla u^{n+\theta}, \nabla u^{n+1}) &
    = ( \nabla (\theta u^{n+1}+(1-\theta)u^n), \nabla u^{n+1})\notag
    \\
    &
    = \theta \|\nabla u^{n+1}\|^2 + (1-\theta) ( \nabla u^n, \nabla u^{n+1})\notag
\\
&
    =
    \theta \|\nabla u^{n+1}\|^2 + \frac{1}{2\theta} \left(-\theta^2\|\nabla u^{n+1}\|^2-(1-\theta)^2\|\nabla u^{n}\|^2+\|\theta\nabla u^{n+1}+(1-\theta)\nabla u^n\|^2\right)\notag
    \\
    &
    \label{eq:2.63}
    =\frac{\theta}{2}\|\nabla u^{n+1}\|^2-\frac{(1-\theta)^2}{2\theta}\|\nabla u^n\|^2+\frac{1}{2\theta}\|\nabla u^{n+\theta}\|^2.
\end{align}
\begin{confidential}
    {\color{red}
    So
\begin{align}\label{eq:boundbandf}
    \frac{1}{2}\|u^{n+1}\|^2-&\frac{1}{2}\|u^n\|^2 + \frac{1}{2} \|u^{n+1}-u^n\|^2 +\frac{\tau\nu\theta}{2}\|\nabla u^{n+1}\|^2-\frac{\tau\nu(1-\theta)^2}{2\theta}\|\nabla u^n\|^2 + \frac{\nu\tau}{2\theta}\|\nabla u^{n+\theta}\|^2
    \\
    &
    = -\tau b(u^{n+\theta},u^{n+\theta}, u^{n+1}) + \tau ( f^{n+\theta},u^{n+1})
\end{align}
    }
\end{confidential}
Next, we bound the body force term using the H\"{o}lder, Young, and Poincar\'{e}-Friedrichs inequalities. This yields
\begin{align}\label{eq:2.66}
    \tau( f^{n+\theta}, u^{n+1}) \leq \tau \|f^{n+\theta}\|\| u^{n+1}\| \leq \frac{\tau}{2\delta}\frac{1}{\lambda_1}\|f^{n+\theta}\|^2+\tau\frac{\delta}{2}\|\nabla u^{n+1}\|^2
\end{align}
for any $\delta>0$, to be chosen later.

The third key step deals with the non-linear convective term. Using the skew-symmetry property \eqref{triformzero}, we have 
\begin{align*}
    -\tau b(u^{n+\theta},u^{n+\theta}, u^{n+1}) 
    &= -\tau b(u^{n+\theta},\theta u^{n+1}+(1-\theta)u^n, u^{n+1}) 
    \\
    &
    = -\tau \theta b(u^{n+\theta}, u^{n+1}, u^{n+1}) -\tau (1-\theta)b(u^{n+\theta},u^n, u^{n+1}) 
    \\
    &
    = -\tau (1-\theta)b(u^{n+\theta},u^n, u^{n+1}). 
\end{align*}
Using H\"{o}lder's inequality with $p=2$, $q=4$, $r=4$ and the Ladyzhenskaya inequality \eqref{ladyzhenskaya} we obtain
\begin{align*}
    -\tau (1-\theta)b(u^{n+\theta},u^n, u^{n+1}) &
    \leq (1-\theta)\tau 2^{-1/2} \|\nabla u^n\|\|u^{n+\theta}\|^{1/2}\|\nabla u^{n+\theta}\|^{1/2}\|u^{n+1}\|^{1/2}\|\nabla u^{n+1}\|^{1/2}.
\end{align*}
\begin{confidential}
{\color{red}
Thus, for all $\theta\in\left(\frac{1}{2},1\right)$,

\begin{align}\label{eq:trick1}
    \frac{1}{2}\|u^{n+1}\|^2-&\frac{1}{2}\|u^n\|^2 + \frac{1}{2} \|u^{n+1}-u^n\|^2 +\frac{\tau\nu\theta}{2}\|\nabla u^{n+1}\|^2-\frac{\tau\nu(1-\theta)^2}{2\theta}\|\nabla u^n\|^2 + \frac{\nu\tau}{2\theta}\|\nabla u^{n+\theta}\|^2
    \\
    &
    = -\tau b(u^{n+\theta},u^{n+\theta}, u^{n+1}) + \tau ( f^{n+\theta},u^{n+1})
    \\
    &
    = -\tau b(u^{n+\theta},\theta u^{n+1}+(1-\theta)u^n, u^{n+1}) + \tau ( f^{n+\theta},u^{n+1})
    \\
    &
    = -\tau \theta b(u^{n+\theta}, u^{n+1}, u^{n+1}) -\tau (1-\theta)b(u^{n+\theta},u^n, u^{n+1}) + \tau ( f^{n+\theta},u^{n+1})
    \\
    &
    = -\tau (1-\theta)b(u^{n+\theta},u^n, u^{n+1}) + \tau ( f^{n+\theta},u^{n+1})
    \\
    &
    \leq 
    -\tau (1-\theta)b(u^{n+\theta},u^n, u^{n+1}) + \frac{\tau}{2\delta} \frac{1}{\lambda_1}\|f^{n+\theta}\|^2+\tau\frac{\delta}{2}\|\nabla u^{n+1}\|^2
\end{align}

And we can deal with the trilinear term using H\"{o}lder's inequality with $p=2$, $q=4$, $r=4$ and Ladyzhenskaya's inequality in 2D ($\|u\|_{L^4(\Omega)}\leq 2^{-1/4}\|u\|^{1/2}\|\nabla u\|^{1/2}$):
\begin{align}
    -\tau (1-\theta)b(u^{n+\theta},u^n, u^{n+1}) &
    = -(1-\theta)\tau\int_{\Omega}(u^{n+\theta}\cdot \nabla)u^nu^{n+1}
    \\
    &
    \leq(1-\theta) \tau\int_{\Omega} |u^{n+\theta}||\nabla u^n||u^{n+1}|
    \\
    &
    \leq (1-\theta)\tau \|\nabla u^n\|\|u^{n+\theta}\|_{L^4(\Omega)}\|u^{n+1}\|_{L^4(\Omega)}
    \\
    &
    \leq (1-\theta)\tau \|\nabla u^n\|2^{-1/4}\|u^{n+\theta}\|^{1/2}\|\nabla u^{n+\theta}\|^{1/2}2^{-1/4}\|u^{n+1}\|^{1/2}\|\nabla u^{n+1}\|^{1/2}
    \\
    &
    = (1-\theta)\tau 2^{-1/2} \|\nabla u^n\|\|u^{n+\theta}\|^{1/2}\|\nabla u^{n+\theta}\|^{1/2}\|u^{n+1}\|^{1/2}\|\nabla u^{n+1}\|^{1/2}
\end{align}
Now recall that $\|u^{n}\|^2\leq K_1$ for all $n$. Then
\begin{align}
    \|u^{n+\theta}\|^2 & 
    = \|\theta u^{n+1}+(1-\theta)u^n\|^2
    \\
    &
    \leq
    2\theta^2\|u^{n+1}\|^2+2(1-\theta)^2\|u^n\|^2
    \\
    &
    \leq 2(\theta^2+(1-\theta)^2)K_1
    \\
    &
    = 2(2\theta^2-2\theta+1)K_1
\end{align}
}
\end{confidential}
Now due to \eqref{eq:energy2.1}, we have
\begin{align*}
    \|u^{n+\theta}\|^2  
    = \|\theta u^{n+1}+(1-\theta)u^n\|^2
    \leq
    2\theta^2\|u^{n+1}\|^2+2(1-\theta)^2\|u^n\|^2
    \leq 2(2\theta^2-2\theta+1)K_1,
\end{align*}
hence
\begin{align*}
    \|u^{n+1}\|^{1/2}\|u^{n+\theta}\|^{1/2} \leq C_1.
\end{align*}
Therefore, applying Young's inequality twice, we get that for arbitrary $\epsilon,x>0$ (to be made precise later),
\begin{confidential}
    {\color{red}
    \begin{align}
    -\tau (1-\theta)b(u^{n+\theta},u^n, u^{n+1}) &
    \leq (1-\theta)\tau 2^{-1/2} \|\nabla u^n\|\|u^{n+\theta}\|^{1/2}\|\nabla u^{n+\theta}\|^{1/2}\|u^{n+1}\|^{1/2}\|\nabla u^{n+1}\|^{1/2}
    \\
    &
    \leq (1-\theta)\tau 2^{-1/2}C_1\|\nabla u^n\|\|\nabla u^{n+\theta}\|^{1/2}\|\nabla u^{n+1}\|^{1/2}
    \\
    &
    \leq (1-\theta)\tau 2^{-1/2}C_1\left(\frac{\epsilon}{2}\|\nabla u^n\|^2 +\frac{1}{2\epsilon}\|\nabla u^{n+\theta}\|\|\nabla u^{n+1}\|\right)
    \\
    &
    \leq (1-\theta)\tau 2^{-1/2}C_1\left(\frac{\epsilon}{2}\|\nabla u^n\|^2 + \frac{1}{2\epsilon}\left(\frac{x}{2}\|\nabla u^{n+1}\|^2+\frac{1}{2x}\|\nabla u^{n+\theta}\|^2\right)\right)
    \\
    &
    =
    \frac{(1-\theta)\tau}{\sqrt{2}}C_1\frac{\epsilon}{2}\|\nabla u^n\|^2 
    +\frac{(1-\theta)\tau}{\sqrt{2}}C_1\frac{x}{4\epsilon}\|\nabla u^{n+1}\|^2
    +
    \frac{(1-\theta)\tau}{\sqrt{2}}C_1\frac{1}{4x\epsilon}\|\nabla u^{n+\theta}\|^2
\end{align}
    }
\end{confidential}
\begin{align}
    -\tau& (1-\theta)b(u^{n+\theta},u^n, u^{n+1}) 
    \leq (1-\theta)\tau 2^{-1/2}C_1\|\nabla u^n\|\|\nabla u^{n+\theta}\|^{1/2}\|\nabla u^{n+1}\|^{1/2}\notag
    \\
    &
    \label{eq:2.93}
    \leq 
    \frac{(1-\theta)\tau}{\sqrt{2}}C_1\frac{\epsilon}{2}\|\nabla u^n\|^2 
    +\frac{(1-\theta)\tau}{\sqrt{2}}C_1\frac{x}{4\epsilon}\|\nabla u^{n+1}\|^2
    +
    \frac{(1-\theta)\tau}{\sqrt{2}}C_1\frac{1}{4x\epsilon}\|\nabla u^{n+\theta}\|^2.
\end{align}
Plugging \eqref{eq:2.63}, \eqref{eq:2.66}, and \eqref{eq:2.93} back into \eqref{eq:2.52} we obtain
\begin{align}\label{eq: expr1}
    &\frac{1}{2}\|u^{n+1}\|^2-\frac{1}{2}\|u^n\|^2 + \frac{1}{2} \|u^{n+1}-u^n\|^2 
    \\
    &
    \qquad+\left(\frac{\tau\nu\theta}{2}-\frac{(1-\theta)\tau}{\sqrt{2}}C_1\frac{x}{4\epsilon}-\frac{\delta\tau}{2}\right)\|\nabla u^{n+1}\|^2
    -\left(\frac{\tau\nu(1-\theta)^2}{2\theta}+\frac{(1-\theta)\tau}{\sqrt{2}}C_1\frac{\epsilon}{2}\right)\|\nabla u^n\|^2 \notag
    \\
    &
    \qquad+ \left(\frac{\nu\tau}{2\theta}-\frac{(1-\theta)\tau}{\sqrt{2}}C_1\frac{1}{4\epsilon x}\right)\|\nabla u^{n+\theta}\|^2 \notag
    \\
    &
    \leq \notag
    \frac{\tau}{2\delta}\frac{1}{\lambda_1}\|f^{n+\theta}\|^2.
\end{align}
\begin{confidential}
    {\color{red}
    In order to have a telescoping sum, but still have a term containing $\|\nabla u^{n+1}\|$ in the left-hand side, I want to pick $\varepsilon,C,\delta,x>0$ such that 
\begin{align}\label{eq:2.98}
   &\left(\frac{\tau\nu\theta}{2}-\frac{(1-\theta)\tau}{\sqrt{2}}C_1\frac{x}{4\epsilon}-\frac{\delta\tau}{2}\right)\|\nabla u^{n+1}\|^2
    -\left(\frac{\tau\nu(1-\theta)^2}{2\theta}+\frac{(1-\theta)\tau}{\sqrt{2}}C_1\frac{\epsilon}{2}\right)\|\nabla u^n\|^2 
    \\
    &
    = 
    \label{eq:2.99}
    \tau C\|\nabla u^{n+1}\|^2+ \left(\frac{\tau\nu\theta}{2}-\frac{(1-\theta)\tau}{\sqrt{2}}C_1\frac{x}{4\epsilon}+\frac{\delta\tau}{2}-\tau C\right)\|\nabla u^{n+1}\|^2
    -\left(\frac{\tau\nu(1-\theta)^2}{2\theta}+\frac{(1-\theta)\tau}{\sqrt{2}}C_1\frac{\epsilon}{2}\right)\|\nabla u^n\|^2 
\end{align}
 and
 \begin{equation}\label{eq:cond2verify}
0\leq \frac{\tau\nu\theta}{2}-\frac{(1-\theta)\tau}{\sqrt{2}}C_1\frac{x}{4\epsilon}-\frac{\delta\tau}{2}-\tau C = \frac{\tau\nu(1-\theta)^2}{2\theta}+\frac{(1-\theta)\tau}{\sqrt{2}}C_1\frac{\epsilon}{2}
 \end{equation}

By direct computation, we can prove that the choice
    }
\end{confidential}
With the values $\varepsilon = \frac{\nu(2\theta-1)}{2\theta(1-\theta)C_1\sqrt{2}}$, $x = \frac{\nu^2(2\theta-1)^2}{2\theta^2(1-\theta)^2C_1^2}$, and $\delta = \frac{\nu(2\theta-1)}{8\theta}$, we see that \eqref{eq: expr1} writes as 
\begin{confidential}
    {\color{red}
    \begin{align}
    \epsilon &= \frac{\nu(2\theta-1)}{2\theta(1-\theta)C_1\sqrt{2}}
    \\
    x &= \frac{\nu^2(2\theta-1)^2}{2\theta^2(1-\theta)^2C_1^2}
    \\
\delta &= \frac{\nu(2\theta-1)}{8\theta}
\\
C &=\frac{\nu(2\theta-1)}{16\theta}
\end{align}
satisfies the statement above.
    }
\end{confidential}
 \begin{confidential}
 {\color{red}
 So I want to have 
 \[
 \frac{\nu\theta}{2}-\frac{\nu(1-\theta)^2}{2\theta}
 = 
 \frac{(1-\theta)}{\sqrt{2}}C_1\frac{\epsilon}{2}
 +\frac{(1-\theta)}{\sqrt{2}}C_1\frac{x}{4\epsilon}
 +\frac{\delta}{2} +C
 \]

 Let $\delta=\frac{(1-\theta)}{\sqrt{2}}C_1\frac{\rho}{2}$, $C=\frac{(1-\theta)}{\sqrt{2}}C_1\frac{D}{4}$. We then want to get 

 \begin{align}
     &\frac{\nu\theta}{2}-\frac{\nu(1-\theta)^2}{2\theta}
 = 
 \frac{(1-\theta)}{\sqrt{2}}C_1\left(\frac{\epsilon}{2}+\frac{\rho}{4}+\frac{D}{4}+\frac{x}{4\epsilon}\right)
 \\
 &
 \iff\\
 &
 \frac{\nu(2\theta-1)}{2\theta}=\frac{(1-\theta)}{\sqrt{2}}C_1\left(\frac{\epsilon}{2}+\frac{\rho}{4}+\frac{D}{4}+\frac{x}{4\epsilon}\right)
 \end{align}

Letting $\rho=D=\epsilon$,
\begin{align}
 &
 \frac{\nu(2\theta-1)}{2\theta}=\frac{(1-\theta)}{\sqrt{2}}C_1\left(\epsilon+\frac{x}{4\epsilon}\right)
 \\
 &
 \iff
 \\
 &
 \frac{\nu(2\theta-1)}{4\theta}+\frac{\nu(2\theta-1)}{4\theta}
 =\frac{(1-\theta)}{\sqrt{2}}C_1\epsilon+\frac{(1-\theta)}{\sqrt{2}}C_1\frac{x}{4\epsilon}
 \end{align}

 Pick $\epsilon$ such that $\frac{\nu(2\theta-1)}{4\theta}=\frac{(1-\theta)}{\sqrt{2}}C_1\epsilon$. So, for $\theta\in(1/2,1)$,
 \[
 \epsilon = \frac{\nu(2\theta-1)}{2\theta(1-\theta)C_1\sqrt{2}}
 \]
I want to note here that this is the exact moment where this doesn't work for Backward Euler $\theta=1$.

Now I pick $x$ such that $\frac{\nu(2\theta-1)}{4\theta}=\frac{(1-\theta)}{\sqrt{2}}C_1\frac{x}{4\epsilon}$ So, for $\theta\in(1/2,1)$,

\[
x = \frac{\nu^2(2\theta-1)^2}{2\theta^2(1-\theta)^2C_1^2}
\]

And thus, we get that

\[
\delta = \frac{(1-\theta)}{\sqrt{2}}C_1\frac{\epsilon}{2}= \frac{\nu(2\theta-1)}{8\theta}
\]

\[
C = \frac{(1-\theta)}{\sqrt{2}}C_1\frac{\epsilon}{4}= \frac{\nu(2\theta-1)}{16\theta}
\]

Now let's verify that \eqref{eq:cond2verify} holds. Plugging in the values of $x,\epsilon,\delta,C$, we get that
\[
\frac{\nu\theta}{2}-\frac{(1-\theta)}{\sqrt{2}}C_1\frac{x}{4\epsilon}-\frac{\delta}{2}- C = \frac{\theta\nu}{2}-\frac{\nu(2\theta-1)}{4\theta}-\frac{\nu(2\theta-1)}{8\theta} = \frac{\nu}{8\theta}(4\theta^2-6\theta+3)>0
\]
And
\[
\frac{\nu(1-\theta)^2}{2\theta}+\frac{(1-\theta)}{\sqrt{2}}C_1\frac{\epsilon}{2} = \frac{\nu}{8\theta}(4\theta^2-6\theta+3)
\]
So our conditions are satisfied. 
We can now compute the quantities in \eqref{eq:2.98}-\eqref{eq:2.99}:

\begin{align}
\frac{\nu\theta}{2}-\frac{(1-\theta)}{\sqrt{2}}C_1\frac{x}{4\epsilon}-\frac{\delta}{2}- C &=\frac{\nu(1-\theta)^2}{2\theta}+\frac{(1-\theta)}{\sqrt{2}}C_1\frac{\epsilon}{2} 
\\
&= \frac{\nu}{8\theta}(4\theta^2-6\theta+3)>0
\end{align}

\[
\frac{\nu}{2\theta}-\frac{(1-\theta)}{\sqrt{2}}C_1\frac{1}{4\epsilon x}=\frac{\nu}{2\theta}-\frac{C_1^4(1-\theta)^4\theta^3}{\nu^3(2\theta-1)^3}
\]
}
\end{confidential}
\begin{align}\label{eq:energy7}
    &\frac{1}{2}\|u^{n+1}\|^2-\frac{1}{2}\|u^n\|^2 + \frac{1}{2} \|u^{n+1}-u^n\|^2 
    \\
    &
    \notag
    \qquad +\frac{\tau\nu(2\theta-1)}{16\theta}\|\nabla u^{n+1}\|^2
    +\frac{\tau\nu}{8\theta}(4\theta^2-6\theta+3)\|\nabla u^{n+1}\|^2
    -\frac{\tau\nu}{8\theta}(4\theta^2-6\theta+3)\|\nabla u^n\|^2 
    \\
    &
    \leq \notag
    \frac{\tau 4\theta}{\nu (2\theta-1)} \frac{1}{\lambda_1}\|f^{n+\theta}\|^2
    + \tau\left(\frac{C_1^4(1-\theta)^4\theta^3}{\nu^3(2\theta-1)^3}-\frac{\nu}{2\theta}\right)\|\nabla u^{n+\theta}\|^2.
\end{align}
\begin{confidential}
    {\color{red}
    Adding from $i=0$ to $n$, and using the bound \eqref{eq:energy3},
\begin{align}
    \frac{1}{2}\|u^{n+1}\|^2&+ \frac{1}{2} \sum_{i=0}^{n}\|u^{i+1}-u^i\|^2 +\frac{\tau\nu}{8\theta}(4\theta^2-6\theta+3)\|\nabla u^{n+1}\|^2
    \\
    &
    +\frac{\tau\nu(2\theta-1)}{16\theta}\sum_{i=0}^{n}\|\nabla u^{i+1}\|^2
    \\
    &
    \leq
    \frac{1}{2}\|u_0\|^2  + \frac{\tau\nu}{8\theta}(4\theta^2-6\theta+3)\|\nabla u_0\|^2 + \frac{\tau 4\theta}{\nu (2\theta-1)}\sum_{i=0}^{n} \frac{1}{\lambda_1}\|f^{n+\theta}\|^2
    \\
    &
    + \left|\frac{C_1^4(1-\theta)^4\theta^3}{\nu^3(2\theta-1)^3}-\frac{\nu}{2\theta}\right|\frac{1}{\nu}\|u_0\|^2+\left|\frac{C_1^4(1-\theta)^4\theta^3}{\nu^3(2\theta-1)^3}-\frac{\nu}{2\theta}\right|\frac{1}{\nu^2} (n+1)\tau  \frac{1}{\lambda_1}\|f\|_{\infty}^2
    \\
    &
    \leq
    \frac{1}{2}\|u_0\|^2  
    + \frac{\tau\nu}{8\theta}(4\theta^2-6\theta+3)\|\nabla u_0\|^2 
    + \frac{4\theta }{\nu (2\theta-1)}(n+1)\tau \frac{1}{\lambda_1}\|f\|_{\infty}^2
    \\
    &
    + \left|\frac{C_1^4(1-\theta)^4\theta^3}{\nu^3(2\theta-1)^3}
    -\frac{\nu}{2\theta}\right|\frac{1}{\nu}\|u_0\|^2
    +\left|\frac{C_1^4(1-\theta)^4\theta^3}{\nu^3(2\theta-1)^3}-\frac{\nu}{2\theta}\right|\frac{1}{\nu^2} (n+1)\tau  \frac{1}{\lambda_1}\|f\|_{\infty}^2
    \\
    &
    =
    \left(\frac{1}{2}+\left|\frac{C_1^4(1-\theta)^4\theta^3}{\nu^3(2\theta-1)^3}
    -\frac{\nu}{2\theta}\right|\frac{1}{\nu}\right)\|u_0\|^2  
    \\
    &
    + \frac{\tau\nu}{8\theta}(4\theta^2-6\theta+3)\|\nabla u_0\|^2 
    \\
    &
    + \left(\frac{4\theta }{\nu (2\theta-1)}+\left|\frac{C_1^4(1-\theta)^4\theta^3}{\nu^3(2\theta-1)^3}-\frac{\nu}{2\theta}\right|\frac{}{\nu^2}\right)(n+1)\tau \frac{1}{\lambda_1}\|f\|_{\infty}^2
    \\
    &
    := K_2  + (n+1)\tau K_3 
\end{align}
    }
\end{confidential}
Adding from $i=0$ to $n$, and using the $L^2(H^1(\Omega))$ bound \eqref{eq:energy3} on the fractional time values yields
\begin{align*}
    &\frac{1}{2}\|u^{n+1}\|^2+ \frac{1}{2} \sum_{i=0}^{n}\|u^{i+1}-u^i\|^2 +\frac{\tau\nu}{8\theta}(4\theta^2-6\theta+3)\|\nabla u^{n+1}\|^2
    +\frac{\tau\nu(2\theta-1)}{16\theta}\sum_{i=0}^{n}\|\nabla u^{i+1}\|^2
    \\
    &
    \leq
    \left(\frac{1}{2}+\left|\frac{C_1^4(1-\theta)^4\theta^3}{\nu^3(2\theta-1)^3}
    -\frac{\nu}{2\theta}\right|\frac{1}{\nu}\right)\|u_0\|^2  
    + \frac{\tau\nu}{8\theta}(4\theta^2-6\theta+3)\|\nabla u_0\|^2 
    \\
    &
    \qquad + \left(\frac{4\theta }{\nu (2\theta-1)}+\left|\frac{C_1^4(1-\theta)^4\theta^3}{\nu^3(2\theta-1)^3}-\frac{\nu}{2\theta}\right|\frac{}{\nu^2}\right)(n+1)\tau \frac{1}{\lambda_1}\|f\|_{\infty}^2,
\end{align*}
which implies \eqref{eq:boundnablaunp1}. Finally, from \eqref{eq:energy7}, adding from $i=n$ to $n+p$, we obtain
\begin{align*}
    &\frac{1}{2}\|u^{n+p+1}\|^2+ \frac{1}{2} \sum_{i=n}^{n+p}\|u^{i+1}-u^i\|^2 +\frac{\tau\nu}{8\theta}(4\theta^2-6\theta+3)\|\nabla u^{n+p+1}\|^2
    +\frac{\tau\nu(2\theta-1)}{16\theta}\sum_{i=n}^{n+p}\|\nabla u^{i+1}\|^2
    \\
    &
    \leq
    \frac{1}{2}\|u^n\|^2  + \frac{\tau\nu}{8\theta}(4\theta^2-6\theta+3)\|\nabla u^n\|^2 + \frac{\tau 4\theta}{\nu (2\theta-1)}(p+1)\frac{1}{\lambda_1}\|f\|_{\infty}^2
    \\
    &
    \qquad + \left|\frac{C_1^4(1-\theta)^4\theta^3}{\nu^3(2\theta-1)^3}-\frac{\nu}{2\theta}\right|\frac{1}{\nu}\|u^n\|^2+\left|\frac{C_1^4(1-\theta)^4\theta^3}{\nu^3(2\theta-1)^3}-\frac{\nu}{2\theta}\right|\frac{1}{\nu^2} (p+1)\tau  \frac{1}{\lambda_1}\|f\|_{\infty}^2
    \\
    &
    \leq
    4 K_4 K_1 + K_3(p+1)\tau + \frac{\tau\nu}{8\theta}(4\theta^2-6\theta+3)\|\nabla u^n\|^2,
\end{align*}
which, in particular, yields \eqref{eq:boundnablaunp1.1}.
\end{proof}

% Section 3 - H1 - Bound

\section{V-Stability}\label{sec:H1bound}
This section establishes our main result, namely that the Cauchy one-leg $\theta$-method \eqref{eq:STAR} is $V$-stable. Our goal is to prove bounds analogous to the continuous-in-time bounds \eqref{ctsH1init} and \eqref{ctsH1}. We follow the approach of \cite{MR2217369}, and briefly outline the key steps here.

In Lemma \ref{lemmaind}, we show that, provided a suitable time-step restriction holds at step $n$, then there is a bound on $\|\nabla u^{n+1}\|$ in terms of $\|\nabla u^{n}\|$. In Lemma \ref{lemma:DGL}, we recall the Discrete Gr\"{o}nwall Lemma. In Proposition \ref{prop2}, we establish finite-time stability using Lemma \ref{lemmaind} and Lemma \ref{lemma:DGL}. In Lemma \ref{lemma:DUGL}, we recall the discrete Uniform Gr\"{o}nwall Lemma. Finally, in Theorem \ref{th:theorem}, we prove the long-time $V$-stability.
\begin{confidential}
    {\color{red}
    We first prove a proposition that states that, if a certain time-step restriction is met for some $n$, then there is a bound $\|u^{n+1}\|$ depending on the norm at previous time steps. We then use this result and the Discrete Gr\"{o}nwall Lemma to prove inductively that ther is a finite-time bound. Finally, we use the Discrete Uniform Gr\"{o}nwall Lemma to prove long-time stability. 
    }
\end{confidential}

To set the stage for the first lemma, we argue as in the proof of \eqref{eq:energy5} in Lemma \ref{lemma:gradutheta}. We test \eqref{eq:BENSE} with $-\theta\tau\Delta u^{n+\theta}$ and \eqref{eq:FENSE} with $-(1-\theta)\tau \Delta u^{n+\theta}$ in $V$, use the polarization identity, the H\"{o}lder and Young inequalities, and add to obtain
\begin{align*}
& 
\frac{1}{2} \big( \|\nabla u^{n+1}\|^2 - \|\nabla u^n\|^2 \big) 
+ \frac{1}{2} (2\theta -1) \|\nabla (u^{n+1} - u^n)\|^2 
+ \frac{3}{4}\nu \tau \| \Delta u^{n+\theta} \|^2 
\leq
\frac{1}{\nu} \tau \|f\|^2_{\infty} 
- \tau  b(u^{n+\theta},u^{n+\theta},-\Delta u^{n+\theta}).
\end{align*}

Next, we use the H\"{o}lder inequality with $p=4$, $q=4$ , $r=2$, the Ladyzhenskaya inequality \eqref{ladyzhenskaya}, and the Young inequality to bound the non-linear term, and we get
\begin{confidential}
    {\color{red}
    \begin{align*}
&\text{(Using H\"{o}lder's inequality with p=4,q=4,r=2)}
\\&\leq\frac{1}{\nu} \tau \|f\|^2_{\infty} 
+ \tau \|u^{n+\theta}\|_{L^4(\Omega)}\|\nabla u^{n+\theta}\|_{L^4(\Omega)}\|\Delta u^{n+\theta}\|
\\
&
\Big(\text{Using Ladyzhenskaya's inequality in 2D from \cite{MR2808162}} \;\;\|u\|_{L^4(\Omega)}\leq 2^{-1/4}\|u\|_{L^2(\Omega)}^{1/2}\|\nabla u\|_{L^2(\Omega)}^{1/2}\Big)
\\
&
\leq
\frac{1}{\nu} \tau \|f\|^2_{\infty} 
+ \tau 2^{-1/4}\|u^{n+\theta}\|^{1/2}\|\nabla u^{n+\theta}\|^{1/2}
2^{-1/4}\|\nabla u^{n+\theta}\|^{1/2}\|\Delta u^{n+\theta}\|^{1/2}\|\Delta u^{n+\theta}\|
\\
&
=
\frac{1}{\nu} \tau \|f\|^2_{\infty} 
+ \frac{\tau}{\sqrt{2}} \|u^{n+\theta}\|^{1/2}\|\nabla u^{n+\theta}\|\|\Delta u^{n+\theta}\|^{3/2}z\frac{1}{z}
\\
&
\text{(Where $z>0$ will be picked later. Now using Young's inequality with $p=4$, $q=\frac{4}{3}$)}
\\
&
\leq
\frac{1}{\nu} \tau \|f\|^2_{\infty} 
+ \tau\frac{1}{4}\frac{1}{4}\frac{1}{z^4} \|u^{n+\theta}\|^{2}\|\nabla u^{n+\theta}\|^4
+\tau\frac{3}{4}z^{4/3}\|\Delta u^{n+\theta}\|^{2}
\\
&
\text{Now picking $z^4=\frac{\nu^3}{27}$}
\\
&
\leq
\frac{1}{\nu} \tau \|f\|^2_{\infty} 
+ \tau\frac{1}{16}\frac{27}{\nu^3} \|u^{n+\theta}\|^{2}\|\nabla u^{n+\theta}\|^4
+\tau\frac{\nu}{4}\|\Delta u^{n+\theta}\|^{2}
\end{align*}
    }
\end{confidential}
\begin{align}
\label{eq: energyineq1}
    &\frac{1}{2} \big( \|\nabla u^{n+1}\|^2 - \|\nabla u^n\|^2 \big) 
+ \frac{1}{2} (2\theta -1) \|\nabla (u^{n+1} - u^n)\|^2 
+ \frac{1}{2}\nu \tau \| \Delta u^{n+\theta} \|^2 
\\
&
\notag
\leq
\frac{1}{\nu} \tau \|f\|^2_{\infty} 
+ \tau\frac{27}{16}\frac{1}{\nu^3} \|u^{n+\theta}\|^{2}\|\nabla u^{n+\theta}\|^4.
\end{align}

\begin{confidential}
    {\color{red}
    Noting $(a+b)^2\leq 2a^2+2b^2$, And noting $(a+b)^4\leq 2(2^2a^4+2^2b^4)$,
    }
\end{confidential}

Since $u^{n+\theta}=\theta u^{n+1}+(1-\theta)u^n$, we have
\begin{align*}
    \|u^{n+\theta}\|^2&=\|\theta u^{n+1}+(1-\theta)u^n\|^2
    \leq
    2(\|u^{n+1}\|^2+\|u^{n}\|^2),
    \\
    \|\nabla u^{n+\theta}\|^4&=\|\theta \nabla u^{n+1}+(1-\theta)\nabla u^n\|^4
    \leq
    2^3(\|\nabla u^{n+1}\|^4+\|\nabla u^{n}\|^4).
\end{align*}
\begin{confidential}
    {\color{red}
    Since $u^{n+\theta}=\theta u^{n+1}+(1-\theta)u^n$, and noting $(a+b)^2\leq 2a^2+2b^2$,
\begin{align*}
    \|u^{n+\theta}\|^2&=\|\theta u^{n+1}+(1-\theta)u^n\|^2
    \\
    &
    \leq
    2\theta^2\|u^{n+1}\|^2+2(1-\theta)^2\|u^n\|^2
    \\
    &
    \leq
    2(\|u^{n+1}\|^2+\|u^{n}\|^2)
    \leq
    4K_1
\end{align*}
    }
\end{confidential}
\begin{confidential}
    {\color{red}
    \begin{align*}
    \|\nabla u^{n+\theta}\|^4&=\|\theta \nabla u^{n+1}+(1-\theta)\nabla u^n\|^4
    \\
    &
    \leq
    2(2^2\theta^4\|\nabla u^{n+1}\|^4+2^2(1-\theta)^4\|\nabla u^n\|^4)
    \\
    &
    \leq
    2^3(\|\nabla u^{n+1}\|^4+\|\nabla u^{n}\|^4)
\end{align*}
    }
\end{confidential}
Therefore, \eqref{eq: energyineq1} yields
\begin{align*}
        &  \|\nabla u^{n+1}\|^2 - \|\nabla u^n\|^2  
+  (2\theta -1) \|\nabla (u^{n+1} - u^n)\|^2 
+ \nu \tau \| \Delta u^{n+\theta} \|^2 
\leq
\frac{2}{\nu} \tau \|f\|^2_{\infty} 
+ \tau\frac{54}{\nu^3}  2K_1\|\nabla u^{n+1}\|^4
+\tau\frac{54}{\nu^3}  2K_1\|\nabla u^{n}\|^4,
\end{align*}
\begin{confidential}
    {\color{red}
    \begin{align*}
        &  \|\nabla u^{n+1}\|^2 - \|\nabla u^n\|^2  
+  (2\theta -1) \|\nabla (u^{n+1} - u^n)\|^2 
+ \nu \tau \| \Delta u^{n+\theta} \|^2 
\\
&
\leq
\frac{2}{\nu} \tau \|f\|^2_{\infty} 
+ \tau\frac{3^3}{2^3}\frac{1}{\nu^3} \|u^{n+\theta}\|^{2}\|\nabla u^{n+\theta}\|^4
\\
&
\leq
\frac{2}{\nu} \tau \|f\|^2_{\infty} 
+ \tau\frac{3^3}{2^3}\frac{1}{\nu^3}  2(\|u^{n+1}\|^2+\|u^{n}\|^2)2^3(\|\nabla u^{n+1}\|^4+\|\nabla u^{n}\|^4)
\\
&
=
\frac{2}{\nu} \tau \|f\|^2_{\infty} 
+ \tau3^3\cdot 2\frac{1}{\nu^3}  (\|u^{n+1}\|^2+\|u^{n}\|^2)\|\nabla u^{n+1}\|^4
+\tau3^3\cdot 2\frac{1}{\nu^3}  (\|u^{n+1}\|^2+\|u^{n}\|^2)\|\nabla u^{n}\|^4
\\
&
\leq
\frac{2}{\nu} \tau \|f\|^2_{\infty} 
+ \tau\frac{54}{\nu^3}  2K_1\|\nabla u^{n+1}\|^4
+\tau\frac{54}{\nu^3}  2K_1\|\nabla u^{n}\|^4
\end{align*}
    }
\end{confidential}
and moreover
\begin{align}\label{eq:quadraticeq}
    0 &\leq  
 \tau\frac{54}{\nu^3}  2K_1\|\nabla u^{n+1}\|^4
-\|\nabla u^{n+1}\|^2
+\|\nabla u^n\|^2
+\tau\frac{54}{\nu^3}  2K_1\|\nabla u^{n}\|^4
+\frac{2}{\nu} \tau \|f\|^2_{\infty}.
\end{align}

Before we state the first lemma of this section, we define some notations. Note that for any $\theta\in\left(\frac{1}{2},1\right]$ there exists two positive constants $\delta_1,\varepsilon_1>0$ such that
\begin{align}
    C_2 &:=2^{-1/2}\theta(1+2\theta)\sqrt{K_1},\\
    C_3 &:=\nu\theta^2+\frac{\nu}{2}\theta^2 -\epsilon_1-\nu\theta(2-\theta)\frac{\delta_1}{2}>0,\\
    C_4 &:=\frac{1}{\epsilon_1}C_2^2-\nu(1-\theta)^2+\nu- \frac{\nu}{2}\theta^2+\nu\theta(2-\theta)\frac{1}{2\delta_1}>0.
\end{align}

\begin{lemma}\label{lemmaind}
    Suppose that the time-step restriction \eqref{eq:step-condition} holds, and assume that, for some $n$, we have
    \begin{align}\label{eq:bound}
    \tau\frac{108}{\nu^3}K_1
\left[\left(\frac{2}{\nu C_3}\frac{1}{\lambda_1} +\frac{2}{\nu^2}\frac{1}{\lambda_1} \right)\|f\|^2_{\infty}+\left(1+\frac{C_4}{C_3}\right)\|\nabla u^{n}\|^2+\frac{108}{\nu^4}\frac{1}{\lambda_1} K_1\|\nabla u^{n}\|^4\right]
\leq\frac{1}{5}.
    \end{align}
    Then 
    \begin{align}\label{eq:2.22}
    \|\nabla u^{n+1}\|^2 &\leq  \|\nabla u^{n}\|^2\left(1+\tau\frac{108}{\nu^3}K_1\|\nabla u^{n}\|^2\right)\left[1+2\tau\frac{108}{\nu^3}K_1\left(\|\nabla u^{n}\|^2+\tau\frac{108}{\nu^3}K_1\|\nabla u^{n}\|^4\right)\right] + \frac{18}{5\nu}\tau \|f\|^2_{\infty}.
    \end{align}
\end{lemma}

\begin{proof}
We begin by noting that the right-hand side of \eqref{eq:quadraticeq} is a quadratic polynomial in $\|\nabla u^{n+1}\|^2$, hence either
\begin{equation}\label{eq:lessthanbound}
    \|\nabla u^{n+1}\|^2\leq\frac{1-\sqrt{\Delta_n}}{2K_1\tau\frac{108}{\nu^3}},
\end{equation}
or
\begin{equation}\label{eq: greaterthanbound}
    \|\nabla u^{n+1}\|^2\geq\frac{1+\sqrt{\Delta_n}}{2K_1\tau\frac{108}{\nu^3}},
\end{equation}
where
\[
\Delta_n=1-4\tau\frac{108}{\nu^3} K_1\left(\|\nabla u^n\|^2
+\tau\frac{108}{\nu^3}K_1\|\nabla u^{n}\|^4
+\frac{2}{\nu} \tau \|f\|^2_{\infty}\right).
\]
Moreover, we note that \eqref{eq:step-condition} and \eqref{eq:bound} imply $\Delta_n> 0$. We now show that only \eqref{eq:lessthanbound} holds.

\begin{confidential}
    {\color{red}
    Recall from \eqref{eq:BENSE} and \eqref{eq:FENSE},
\begin{align*}
&\frac{1}{\theta\tau}(u^{n+\theta}-u^n)-\nu\Delta u^{n+\theta}+u^{n+\theta}\cdot\nabla u^{n+\theta}+\nabla p^{n+\theta}=f^{n+\theta}
\\
&
\nabla\cdot u^{n+\theta}=0
\end{align*}
    }
\end{confidential}

Testing \eqref{eq:BENSE} with $2\tau(u^{n+\theta}-u^n)$ in $H$ and using the divergence theorem and polarization identity gives
\begin{confidential}
    {\color{red}
    \begin{align*}
    &\frac{2}{\theta}\|u^{n+\theta}-u^n\|^2
    -2\nu\tau(\Delta u^{n+\theta},u^{n+\theta}-u^n)
    +2\tau( u^{n+\theta}\cdot\nabla u^{n+\theta},u^{n+\theta}-u^n)
    =2\tau( f^{n+\theta},u^{n+\theta}-u^n)
\end{align*}
\begin{align*}
    &\frac{2}{\theta}\|u^{n+\theta}-u^n\|^2
    +2\nu\tau(\nabla u^{n+\theta},\nabla u^{n+\theta}-\nabla u^n)
    +2\tau( u^{n+\theta}\cdot\nabla u^{n+\theta},u^{n+\theta}-u^n)
    =2\tau( f^{n+\theta},u^{n+\theta}-u^n)
\end{align*}
From here, let's denote 
\[
b(u,v,w)=( u\cdot\nabla v,w)
\]
Since $\|a\|^2-\|b\|^2+\|a-b\|^2=2( a,a-b)$ (polarization identity),
    }
\end{confidential}
\begin{align}\label{eq:label1}
     &\frac{2}{\theta}\|u^{n+\theta}-u^n\|^2
    +\nu\tau\|\nabla u^{n+\theta}\|^2
    -\nu\tau\|\nabla u^n\|^2
    +\nu\tau\|\nabla u^{n+\theta}-\nabla u^n\|^2
    +2\tau b(u^{n+\theta},u^{n+\theta},u^{n+\theta}-u^n)
    \\
    &
    \notag
    =2\tau( f^{n+\theta},u^{n+\theta}-u^n).
\end{align}
Applying the properties of the trilinear form and the H\"{o}lder, Ladyzhenskaya, and Young inequalities yields
\begin{confidential}
    {\color{red}
    \begin{align*}
    b(u^{n+\theta},u^{n+\theta},u^{n+\theta}-u^n) &= b(u^{n+\theta},u^{n+\theta}, -u^n)\\
    &= b(u^{n+\theta},u^n,u^{n+\theta})
\end{align*}
We want to write $u^{n+\theta}$ in terms of $u^{n+1}$ and $u^n$. Using properties defined in the Mathematical Preliminaries section and Young's inequality:
\begin{align*}
    b(u^{n+\theta},u^n,u^{n+\theta}) 
    &
    =
    b(\theta u^{n+1}+(1-\theta) u^n,u^n,\theta u^{n+1}+(1-\theta )u^n)
    \\
    &
    =
    b(\theta u^{n+1},u^n,\theta u^{n+1})
    \\
    &
    +
    b(\theta u^{n+1},u^n,(1-\theta) u^{n})
    \\
    &
    +
    b((1-\theta) u^{n},u^n,\theta u^{n+1})
    \\
    &
    +
    b((1-\theta) u^{n},u^n,(1-\theta) u^{n})
    \\
    &
    =
    \theta^2 b(u^{n+1},u^n,u^{n+1})+(1-\theta)\theta b(u^n,u^n,u^{n+1})
    \\
    &
    = 
    \theta^2 b(u^{n+1},u^n,u^{n+1})+\theta b(u^n,u^n,u^{n+1})-\theta^2b(u^n,u^n,u^{n+1})
    \\
    &
    =
    \theta^2 b(u^{n+1},u^n,u^{n+1})-\theta b(u^n,u^{n+1},u^{n})+\theta^2 b(u^n,u^{n+1},u^{n})
    \\
    &
    \leq
    2^{-1/2}\theta^2\|u^{n+1}\|^{1/2}\|\nabla u^{n+1}\|^{1/2}\|\nabla u^{n}\|\|\nabla u^{n+1}\|^{1/2}\|u^{n+1}\|^{1/2}
    \\
    &
    +
    2^{-1/2}\theta\|u^{n}\|^{1/2}\|\nabla u^{n}\|^{1/2}\|\nabla u^{n+1}\|\|\nabla u^{n}\|^{1/2}\|u^{n}\|^{1/2}
    \\
    &
    +
    2^{-1/2}\theta^2\|u^{n}\|^{1/2}\|\nabla u^{n}\|^{1/2}\|\nabla u^{n+1}\|\|\nabla u^{n}\|^{1/2}\|u^{n}\|^{1/2}
    \\
    &
    =
    2^{-1/2}\theta^2\|u^{n+1}\|\|\nabla u^{n+1}\|\|\nabla u^{n}\|+2^{-1/2}\theta^2\|u^{n}\|\|\nabla u^{n}\|\|\nabla u^{n+1}\|+2^{-1/2}\theta\|u^{n}\|\|\nabla u^{n}\|\|\nabla u^{n+1}\|
    \\
    &
    \leq
    2^{-1/2}\theta^2\sqrt{K_1}\|\nabla u^{n+1}\|\|\nabla u^{n}\|+2^{-1/2}\theta^2\sqrt{K_1}\|\nabla u^{n}\|\|\nabla u^{n+1}\|+2^{-1/2}\theta\sqrt{K_1}\|\nabla u^{n}\|\|\nabla u^{n+1}\|
    \\
    &
    = 
    2^{-1/2}(2\theta^2\sqrt{K_1}+\theta\sqrt{K_1})\|\nabla u^{n}\|\|\nabla u^{n+1}\|
    \\
    &
    \leq
    \frac{1}{2\epsilon}C_2^2\|\nabla u^n\|^2+\frac{\epsilon}{2}\|\nabla u^{n+1}\|^2
\end{align*}
Where  $\epsilon>0$ will be picked later.\\
    }
\end{confidential}
\begin{align*}
    &b(u^{n+\theta},u^{n+\theta},u^{n+\theta}-u^n) 
    = b(u^{n+\theta},u^n,u^{n+\theta})
    \\
    &
    =
    b(\theta u^{n+1}+(1-\theta) u^n,u^n,\theta u^{n+1}+(1-\theta )u^n)
    \\
    &
    =
    \theta^2 b(u^{n+1},u^n,u^{n+1})-\theta b(u^n,u^{n+1},u^{n})+\theta^2 b(u^n,u^{n+1},u^{n})
    \\
    &
    \leq
    2^{-1/2}\theta^2\|u^{n+1}\|\|\nabla u^{n+1}\|\|\nabla u^{n}\|+2^{-1/2}\theta^2\|u^{n}\|\|\nabla u^{n}\|\|\nabla u^{n+1}\|+2^{-1/2}\theta\|u^{n}\|\|\nabla u^{n}\|\|\nabla u^{n+1}\|
    \\
    &
    \leq
    2^{-1/2}(2\theta^2\sqrt{K_1}+\theta\sqrt{K_1})\|\nabla u^{n}\|\|\nabla u^{n+1}\|
    \\
    &
    \leq
    \frac{1}{2\epsilon_1}C_2^2\|\nabla u^n\|^2+\frac{\epsilon_1}{2}\|\nabla u^{n+1}\|^2.
\end{align*}
We use the H\"{o}lder, Young, and Poincar\'{e}-Friedrichs inequalities to bound the forcing term
\begin{align*}
    2\tau( f^{n+\theta},u^{n+\theta}-u^n) 
    \leq \frac{\nu\tau}{2}\|\nabla(u^{n+\theta}-u^n)\|^2+\tau\frac{2}{\nu} \frac{1}{\lambda_1}\|f^{n+\theta}\|^2.
\end{align*}
\begin{confidential}
    {\color{red}
    \begin{align*}
    2\tau( f^{n+\theta},u^{n+\theta}-u^n) & \leq 2\tau\|f^{n+\theta}\|\|u^{n+\theta}-u^n\|\\
    &
    \leq \frac{\nu\tau}{2}\|\nabla(u^{n+\theta}-u^n)\|^2+\tau\frac{2}{\nu} \frac{1}{\lambda_1}\|f^{n+\theta}\|^2
\end{align*}
    }
\end{confidential}
Plugging these two inequalities  into \eqref{eq:label1} implies
\begin{align}\label{eq:label2}
    &\frac{2}{\theta}\|u^{n+\theta}-u^n\|^2
    +\nu\tau\|\nabla u^{n+\theta}\|^2
    -\nu\tau\|\nabla u^n\|^2
    +\nu\tau\|\nabla u^{n+\theta}-\nabla u^n\|^2
    \\
    &
    \notag
    \leq
    \frac{2\tau}{2\epsilon_1}C_2^2\|\nabla u^n\|^2+\frac{2\tau\epsilon_1}{2}\|\nabla u^{n+1}\|^2+\frac{\nu\tau}{2}\|\nabla(u^{n+\theta}-u^n)\|^2+\frac{2\tau}{\nu} \frac{1}{\lambda_1}\|f^{n+\theta}\|^2.
\end{align}
Using $u^{n+\theta}=\theta u^{n+1}+(1-\theta)u^n$, we write the norms of $u^{n+\theta}$ in terms of $u^{n+1}$ and $u^n$
\begin{confidential}
    {\color{red}
    \begin{align*}
    \|\nabla u^{n+\theta}\|^2 &= \| \theta\nabla u^{n+1}+(1-\theta)\nabla u^n\|^2
    \\
    &
    = \theta^2 \|\nabla u^{n+1}\|^2 + (1-\theta)^2\|\nabla u^{n}\|^2+2\theta(1-\theta)(\nabla u^{n+1},\nabla u^n)
\end{align*}
\begin{align*}
    \|\nabla u^{n+\theta}-\nabla u^n\|^2 &= \| \theta\nabla u^{n+1}-\theta\nabla u^n\|^2
    \\
    &
    = \theta^2 \|\nabla u^{n+1}\|^2 + \theta^2\|\nabla u^{n}\|^2+2\theta^2(\nabla u^{n+1},\nabla u^n)
\end{align*}
    }
\end{confidential}
\begin{align*}
    \|\nabla u^{n+\theta}\|^2 
    &= \theta^2 \|\nabla u^{n+1}\|^2 + (1-\theta)^2\|\nabla u^{n}\|^2+2\theta(1-\theta)(\nabla u^{n+1},\nabla u^n),\\
    \|\nabla u^{n+\theta}-\nabla u^n\|^2 
    &= \theta^2 \|\nabla u^{n+1}\|^2 + \theta^2\|\nabla u^{n}\|^2+2\theta^2(\nabla u^{n+1},\nabla u^n),
\end{align*}
so \eqref{eq:label2} becomes
\begin{confidential}
    {\color{red}
    \begin{align*}
     &\frac{2}{\theta}\|u^{n+\theta}-u^n\|^2\\
    &+\nu\tau\theta^2 \|\nabla u^{n+1}\|^2 + \nu\tau(1-\theta)^2\|\nabla u^{n}\|^2+2\nu\tau\theta(1-\theta)(\nabla u^{n+1},\nabla u^n)\\
    &-\nu\tau\|\nabla u^n\|^2\\
    &+\frac{\nu\tau}{2}\theta^2 \|\nabla u^{n+1}\|^2 + \frac{\nu\tau}{2}\theta^2\|\nabla u^{n}\|^2+\nu\tau\theta^2(\nabla u^{n+1},\nabla u^n)
    \\
    &
    \leq
    \frac{\tau}{\epsilon}C_2^2\|\nabla u^n\|^2+\tau\epsilon\|\nabla u^{n+1}\|^2+\frac{2\tau}{\nu} \frac{1}{\lambda_1}\|f^{n+\theta}\|^2
\end{align*}
Rearranging,
    }
\end{confidential}
\begin{align*}
    &\frac{2}{\theta}\|u^{n+\theta}-u^n\|^2+(\nu\tau\theta^2+\frac{\nu\tau}{2}\theta^2 -\tau\epsilon)\|\nabla u^{n+1}\|^2 
    + (\nu\tau(1-\theta)^2-\nu\tau+ \frac{\nu\tau}{2}\theta^2-\frac{\tau}{\epsilon}C_2^2)\|\nabla u^{n}\|^2
    \\
    &
    \leq
    \frac{2\tau}{\nu} \frac{1}{\lambda_1}\|f^{n+\theta}\|^2-\nu\tau\theta(2-\theta)(\nabla u^{n+1},\nabla u^n).
\end{align*}
Moreover, using the H\"{o}lder, Young, and Poincar\'{e}-Friedrichs inequalities, we obtain
\begin{confidential}
    {\color{red}
    \begin{align*}
    &\left(\nu\tau\theta^2+\frac{\nu\tau}{2}\theta^2 -\tau\epsilon\right)\|\nabla u^{n+1}\|^2 \\
    &
    \leq
    \frac{2\tau}{\nu} \frac{1}{\lambda_1}\|f^{n+\theta}\|^2\\
    &+ \left(\frac{\tau}{\epsilon}C_2^2-\nu\tau(1-\theta)^2+\nu\tau- \frac{\nu\tau}{2}\theta^2\right)\|\nabla u^{n}\|^2\\
    &
    +\nu\tau\theta(2-\theta)\|\nabla u^{n+1}\|\|\nabla u^n\|\\
    &
    \leq
    \frac{2\tau}{\nu} \frac{1}{\lambda_1}\|f^{n+\theta}\|^2\\
    &+ (\frac{\tau}{\epsilon}C_2^2-\nu\tau(1-\theta)^2+\nu\tau- \frac{\nu\tau}{2}\theta^2)\|\nabla u^{n}\|^2\\
    &
    +\nu\tau\theta(2-\theta)\frac{\delta}{2}\|\nabla u^{n+1}\|^2+\nu\tau\theta(2-\theta)\frac{1}{2\delta}\|\nabla u^n\|^2\\
\end{align*}
where $\delta$ will be picked later. 

\begin{align*}
    (\nu\tau\theta^2+\frac{\nu\tau}{2}\theta^2 -\tau\epsilon)\|\nabla u^{n+1}\|^2 
    &\leq
    \frac{2\tau}{\nu} \frac{1}{\lambda_1}\|f^{n+\theta}\|^2
    \\&+ (\frac{\tau}{\epsilon}C_2^2-\nu\tau(1-\theta)^2+\nu\tau- \frac{\nu\tau}{2}\theta^2)\|\nabla u^{n}\|^2\\
    &
    +\nu\tau\theta(2-\theta)\frac{\delta}{2}\|\nabla u^{n+1}\|^2+\nu\tau\theta(2-\theta)\frac{1}{2\delta}\|\nabla u^n\|^2\\
\end{align*}

So finally we get
    }
\end{confidential}
\begin{align*}
    &\left(\nu\tau\theta^2+\frac{\nu\tau}{2}\theta^2 -\tau\epsilon_1-\nu\tau\theta(2-\theta)\frac{\delta_1}{2}\right)\|\nabla u^{n+1}\|^2 
    \\
    &
    \leq
    \frac{2\tau}{\nu} \frac{1}{\lambda_1}\|f\|^2_{\infty}+ \left(\frac{\tau}{\epsilon_1}C_2^2-\nu\tau(1-\theta)^2+\nu\tau- \frac{\nu\tau}{2}\theta^2+\nu\tau\theta(2-\theta)\frac{1}{2\delta_1}\right)\|\nabla u^{n}\|^2,
\end{align*}
\begin{confidential}
    {\color{red}
    We can fix $\epsilon =\epsilon_1$ and $\delta=\delta_1$ small enough such that 
\[C_3:=\nu\theta^2+\frac{\nu}{2}\theta^2 -\epsilon_1-\nu\theta(2-\theta)\frac{\delta_1}{2}>0\]
\[C_4:=\frac{1}{\epsilon_1}C_2^2-\nu(1-\theta)^2+\nu- \frac{\nu}{2}\theta^2+\nu\theta(2-\theta)\frac{1}{2\delta_1}>0\]
    }
\end{confidential}
so
\begin{align*}
    \|\nabla u^{n+1}\|^2&\leq\frac{2}{\nu C_3} \frac{1}{\lambda_1}\|f\|^2_{\infty}+\frac{C_4}{C_3}\|\nabla u^{n}\|^2.
\end{align*}
\begin{confidential}
    {\color{red}
    Our assumption for the proposition is that for some $n$, we have
\begin{align}\label{eq:boundp1}
\tau\frac{108}{\nu^3}K_1
\left[\left(\frac{2}{\nu C_3} +\frac{2}{\nu^2}\frac{1}{\lambda_1} \right)\|f\|^2_{\infty}+\left(1+\frac{C_4}{C_3}\right)\|\nabla u^{n}\|^2+\frac{108}{\nu^4}\frac{1}{\lambda_1} K_1\|\nabla u^n\|^4\right]
\leq\frac{1}{5}
\end{align}
    }
\end{confidential}
Now, from assumption \eqref{eq:bound} we have 
\begin{align*}
    2K_1\tau\frac{108}{\nu^3}\|\nabla u^{n+1}\|^2&\leq2K_1\tau\frac{108}{\nu^3} \left(\frac{2}{\nu C_3}\frac{1}{\lambda_1} \|f\|^2_{\infty}+\frac{C_4}{C_3}\|\nabla u^{n}\|^2\right)< 1,
\end{align*}
which contradicts \eqref{eq: greaterthanbound}. Therefore, we must have \eqref{eq:lessthanbound}
\begin{align}\label{eq:starmar27}
    \|\nabla u^{n+1}\|^2&\leq\frac{1-\sqrt{\Delta_n}}{2K_1\tau\frac{108}{\nu^3}}
    \\
    &
    \notag
    = \frac{1-\big(1-4\tau\frac{108}{\nu^3} K_1\big(\|\nabla u^n\|^2
+\tau\frac{108}{\nu^3}K_1\|\nabla u^{n}\|^4
+\frac{2}{\nu} \tau  
\|f\|^2_{\infty}\big)\big)^{1/2}}{2K_1\tau\frac{108}{\nu^3}}
\\
&
\notag
= \frac{2\big(\|\nabla u^n\|^2
+\tau\frac{108}{\nu^3}K_1\|\nabla u^{n}\|^4
+\frac{2}{\nu} \tau  \|f\|^2_{\infty}\big)}{1+(1-x)^{1/2}},
\end{align}
where 
\[
x = 4\tau\frac{108}{\nu^3} K_1\left(\|\nabla u^n\|^2
+\tau\frac{108}{\nu^3}K_1\|\nabla u^{n}\|^4
+\frac{2}{\nu} \tau  \|f\|^2_{\infty}\right).
\]

We now proceed to proving the conclusion \eqref{eq:2.22} of the lemma. Using the bounds \eqref{eq:step-condition} and \eqref{eq:bound}, we get
\begin{align*}
    x &= 4\tau\frac{108}{\nu^3} K_1\left(\|\nabla u^n\|^2
+\tau\frac{108}{\nu^3}K_1\|\nabla u^{n}\|^4
+\frac{2}{\nu} \tau  \|f\|^2_{\infty}\right)
\\
&
\leq 
 4\tau\frac{108}{\nu^3} K_1\left(\|\nabla u^n\|^2
+\frac{1 }{\lambda_1\nu}\frac{108}{\nu^3}K_1\|\nabla u^{n}\|^4
+\frac{2}{\nu} \frac{1 }{\lambda_1\nu}  \|f\|^2_{\infty}\right)\leq
\frac{4}{5},
\end{align*}
and since
\[\frac{2}{1+(1-x)^{1/2}}\leq1+\frac{x}{2},\]
for $0\leq x \leq \frac{4}{5}$, then by \eqref{eq:starmar27} we obtain that
\begin{align}\label{expr2}
    \|\nabla u^{n+1}\|^2&\leq 
    \notag
    \left(\|\nabla u^n\|^2
+\tau\frac{108}{\nu^3}K_1\|\nabla u^{n}\|^4
+\frac{2}{\nu} \tau  \|f\|^2_{\infty}\right)
\\
&
\notag
\qquad
\times\left(1+2\tau\frac{108}{\nu^3} K_1\left(\|\nabla u^n\|^2
+\tau\frac{108}{\nu^3}K_1\|\nabla u^{n}\|^4
+\frac{2}{\nu} \tau \|f\|^2_{\infty}\right)\right)
\\
    &= \|\nabla u^{n}\|^2\left(1+\tau\frac{108}{\nu^3}K_1\|\nabla u^{n}\|^2\right)+\frac{2}{\nu}\tau  \|f\|^2_{\infty}
    \\
    &
    \notag
    \qquad
    +2\tau\frac{108}{\nu^3}K_1\|\nabla u^{n}\|^4\left(1+\tau\frac{108}{\nu^3}K_1\|\nabla u^{n}\|^2\right)^2
    \\
    &
    \qquad
    \notag
    + 2\tau\frac{108}{\nu^3}K_1\|\nabla u^{n}\|^2\left(1+\tau\frac{108}{\nu^3}K_1\|\nabla u^{n}\|^2\right)\frac{2}{\nu}\tau \|f\|^2_{\infty}
    \\
    &
    \qquad
    \notag
    +2\tau\frac{108}{\nu^3}K_1\|\nabla u^{n}\|^2\left(1+\tau\frac{108}{\nu^3}K_1\|\nabla u^{n}\|^2\right)\frac{2}{\nu}\tau \|f\|^2_{\infty}
    \\
    &
    \qquad
    \notag
    + 2\tau\frac{108}{\nu^3}K_1\frac{4}{\nu^2}\tau^2 \|f\|^4_{\infty}.
    \end{align}

     Using the time-step restriction \eqref{eq:step-condition} and the assumption \eqref{eq:bound} in the right hand side of \eqref{expr2} the last three terms equate to
    \begin{align*}
        & \|f\|^2_{\infty}\bigg[2\tau\frac{108}{\nu^3}K_1\frac{2}{\nu}\tau\bigg(2\|\nabla u^{n}\|^2\big(1+\tau\frac{108}{\nu^3}K_1\|\nabla u^{n}\|^2\big)+\frac{2}{\nu}\tau \|f\|^2_{\infty}\bigg)\bigg]
        \leq
        \frac{8}{5\nu}\tau \|f\|^2_{\infty}.
    \end{align*}   
Adding this to the second term in the right-hand side of \eqref{expr2} gives 
    \[
    \frac{2}{\nu}\tau  \|f\|^2_{\infty}+\frac{4}{\nu}\tau \|f\|^2_{\infty}\cdot\frac{2}{5} = \frac{18}{5\nu}\tau \|f\|^2_{\infty}.
    \]
The rest of the terms equal
\begin{confidential}
    {\color{red}
    \begin{align*}
    &\|\nabla u^{n}\|^2\big(1+\tau\frac{108}{\nu^3}K_1\|\nabla u^{n}\|^2\big)\bigg[1+2\tau\frac{108}{\nu^3}K_1\|\nabla u^{n}\|^2\big(1+\tau\frac{108}{\nu^3}K_1\|\nabla u^{n}\|^2\big)\bigg]
    \\
    &
    =
    \|\nabla u^{n}\|^2\big(1+\tau\frac{108}{\nu^3}K_1\|\nabla u^{n}\|^2\big)\bigg[1+2\tau\frac{108}{\nu^3}K_1\left(\|\nabla u^{n}\|^2+\tau\frac{108}{\nu^3}K_1\|\nabla u^{n}\|^4\right)\bigg]
\end{align*}
    }
\end{confidential}
\begin{align*}
    \|\nabla u^{n}\|^2\left(1+\tau\frac{108}{\nu^3}K_1\|\nabla u^{n}\|^2\right)\left[1+2\tau\frac{108}{\nu^3}K_1\left(\|\nabla u^{n}\|^2+\tau\frac{108}{\nu^3}K_1\|\nabla u^{n}\|^4\right)\right].
\end{align*}
Therefore, putting everything together, we obtain the bound \eqref{eq:2.22}.
\begin{confidential}
    {\color{red}
    
So we arrive at the result of the lemma:
\begin{align} 
    \|\nabla u^{n+1}\|^2 &\leq  \|\nabla u^{n}\|^2\big(1+\tau\frac{108}{\nu^3}K_1\|\nabla u^{n}\|^2\big)\bigg[1+2\tau\frac{108}{\nu^3}K_1\left(\|\nabla u^{n}\|^2+\tau\frac{108}{\nu^3}K_1\|\nabla u^{n}\|^4\right)\bigg] + \frac{18}{5\nu}\tau \|f\|^2_{\infty}
\end{align}

    }
\end{confidential}
\end{proof}

For reader's convenience we recall a result similar to the Discrete Gr\"{o}nwall Lemma proved in \cite{MR2217369,MR1116181}.
\begin{lemma}[Discrete Gr\"{o}nwall Lemma]\label{lemma:DGL}
    Given $\tau>0$, positive integer $n_*>0$ and positive sequences $\xi_n,\eta_n,\zeta_n,\alpha_n$ such that
    \begin{equation}\label{eq:lemma6}
        \xi_n\leq\xi_{n-1}(1+\tau\alpha_{n-1})(1+\tau\eta_{n-1})+\tau\zeta_n
    \end{equation}
    for $n=1,\dots n_*$.\\
    We have, for any $n\in\{2,\dots,n_*\}$,
    \begin{equation}\label{1star1}
        \xi_n\leq\xi_0\exp\bigg(\sum_{i=0}^{n-1}\tau\eta_i\bigg)\exp\bigg(\sum_{i=0}^{n-1}\tau\alpha_i\bigg)+\sum_{i=1}^{n-1}\Bigg[\tau\zeta_i\exp\bigg(\sum_{j=i}^{n-1}\tau\eta_j\bigg)\exp\bigg(\sum_{j=i}^{n-1}\tau\alpha_j\bigg)\Bigg]+\tau\zeta_n.
    \end{equation}
\end{lemma}
\begin{proof}
    Using \eqref{eq:lemma6} recursively, we derive
    \[
    \xi_n\leq\xi_0\prod_{i=0}^{n-1}\big[(1+\tau\eta_i)(1+\tau\alpha_i)\big]+\sum_{i=1}^{n-1}\tau\zeta_i\prod_{j=i}^{n-1}\big[(1+\tau\eta_j)(1+\tau\alpha_j)\big]+\tau\zeta_n
    \]
    Using the fact that $1+x\leq e^x$ for all $x\in\mathbb{R}$, the conclusion of the lemma follows.
\end{proof}

%%%%%%%%%%%%%%%%%%%%%%%%%%%%%%%%%%%%%%%%%%%%%%%%%%%%%%%%%%%%%%%%%%%%%%

We are now ready to prove a $V$-bound on $u^n$ for a finite time interval $[0,T]$. For ease of notation, we define the following 
    \begin{align*}
        C_5 &:= \frac{108 K_1 16\theta}{\nu^4(2\theta-1)},\qquad
        C_6 := 3C_5K_2+2C_5^2K_2^2,
        \qquad
        C_7 := 3C_5K_3+4C_5^2K_2K_3\qquad
        C_8 := 2C_5^2K_3^2,
    \end{align*}
    and
    \begin{align}
    \notag
    K_5(\|\nabla u_0\|,\|f\|_\infty,n\tau) := 
    &
     \|\nabla u_0\|^2\exp(C_6)\exp(C_7n\tau)\exp(C_8(n\tau)^2)
     \\
     &
     +n\tau\frac{18}{5\nu}\|f\|_{\infty}^2\exp(C_6)\exp(C_7n\tau)\exp(C_8(n\tau)^2).
     \label{eq:K5}
\end{align}
Also, we denote
\begin{align}
        \kappa_2 &=  \frac{1}{15\frac{108}{\nu^3}K_1(\frac{2}{\nu C_3}\frac{1}{\lambda_1}+\frac{2  }{\lambda_1\nu^2})\|f\|_{\infty}^2}\label{eq:2.38},\\
        \kappa_3 &=\frac{1}{15\frac{108}{\nu^3}K_1(1+C_4/C_3)K_5(\|\nabla u_0\|,\|f\|_{\infty},T)}\label{eq:2.37},\\
        \kappa_4 &= \frac{1}{15\frac{108^2}{\nu^7}K_1^2\frac{1}{\lambda_1} K_5^2(\|\nabla u_0\|,\|f\|_{\infty},T)}.\label{eq:2.38p}
    \end{align}

\begin{proposition}\label{prop2}
    Let $T>0$ and let $K_5(\cdot,\cdot,\cdot)$ be the function defined in \eqref{eq:K5}, which is monotonically increasing in all of its arguments. \\
    Suppose the time step is such that
    \begin{equation}\label{eq:2.36}
    \tau\leq \min\{\kappa_1,\kappa_2(\|f\|_{\infty}),\kappa_3(\|\nabla u_0\|,\|f\|_{\infty},T),\kappa_4(\|\nabla u_0\|,\|f\|_{\infty},T)\}.
    \end{equation}
    Then the bound \eqref{eq:2.22} holds for all $n+1=1,\dots,N:=\lfloor T/\tau\rfloor$ and 
    \begin{align}\label{2star}
        \|\nabla u^n\|^2\leq K_5(\|\nabla u_0\|,\|f\|_{\infty},n\tau) \qquad \forall n=0,\dots,N:=\lfloor T/\tau\rfloor.
    \end{align}
\end{proposition}

\begin{proof}
    Fix $T>0$ and let $\tau$ satisfy \eqref{eq:2.36}. We will use induction on $n$.
    When  $n=0$,
    \[
    \|\nabla u_0\|^2 \leq \|\nabla u_0\|^2 \exp(C_6) = K_5(\|\nabla u_0\|,\|f\|_{\infty},0).
    \]
    By the time-step assumption \eqref{eq:2.36} and the fact that $K_5(\cdot,\cdot,\cdot)$ is monotonically increasing in all of its arguments, the condition \eqref{eq:bound} is satisfied for $n=0$, hence by Lemma \ref{lemmaind} we have
\begin{confidential}
    {\color{red}
    \begin{align}
    \tau\frac{108}{\nu^3}K_1 &
\Bigg[\left(\frac{2}{\nu C_3}\frac{1}{\lambda_1} +\frac{2}{\nu^2}\frac{1}{\lambda_1} \right)\|f\|^2_{\infty}+\left(1+\frac{C_4}{C_3}\right)\|\nabla u^{0}\|^2
\\
&
+\frac{108}{\nu^4}\frac{1}{\lambda_1} K_1\|\nabla u_0\|^4\Bigg]
\\
&
\leq 
\tau\frac{108}{\nu^3}K_1
\Bigg[\left(\frac{2}{\nu C_3}\frac{1}{\lambda_1} +\frac{2}{\nu^2}\frac{1}{\lambda_1} \right)\|f\|^2_{\infty}+\left(1+\frac{C_4}{C_3}\right)K_5(\|\nabla u_0\|,\|f\|_{\infty},0)
\\
&
+\frac{108}{\nu^4}\frac{1}{\lambda_1} K_1K_5^2(\|\nabla u_0\|,\|f\|_{\infty},0)\Bigg]
\\
&
\leq 
\tau\frac{108}{\nu^3}K_1
\Bigg[\left(\frac{2}{\nu C_3}\frac{1}{\lambda_1} +\frac{2}{\nu^2}\frac{1}{\lambda_1} \right)\|f\|^2_{\infty}+\left(1+\frac{C_4}{C_3}\right)K_5(\|\nabla u_0\|,\|f\|_{\infty},T)
\\
&
+\frac{108}{\nu^4}\frac{1}{\lambda_1} K_1K_5^2(\|\nabla u_0\|,\|f\|_{\infty},T)\Bigg]
\\ &
\leq\frac{1}{15}+\frac{1}{15}+\frac{1}{15}
=\frac{1}{5}
    \end{align}
    \begin{align}
    \tau\frac{108}{\nu^3}K_1 &
\Bigg[\left(\frac{2}{\nu C_3}\frac{1}{\lambda_1} +\frac{2}{\nu^2}\frac{1}{\lambda_1} \right)\|f\|^2_{\infty}+\left(1+\frac{C_4}{C_3}\right)\|\nabla u^{0}\|^2
\\
&
+\frac{108}{\nu^4}\frac{1}{\lambda_1} K_1\|\nabla u_0\|^4\Bigg]
\leq\frac{1}{15}+\frac{1}{15}+\frac{1}{15}
=\frac{1}{5}
    \end{align}
    }
\end{confidential}
\begin{align*} 
    \|\nabla u^{1}\|^2 &\leq  \|\nabla u^{0}\|^2\left(1+\tau\frac{108}{\nu^3}K_1\|\nabla u^{0}\|^2\right)\left[1+2\tau\frac{108}{\nu^3}K_1\left(\|\nabla u^{0}\|^2+\tau\frac{108}{\nu^3}K_1\|\nabla u^{0}\|^4\right)\right] + \frac{18}{5\nu}\tau \|f\|^2_{\infty}.
\end{align*}

Now we assume that \eqref{eq:bound} holds for $n=0,\dots,m-1$ for some $m\leq N$. Then, again by Lemma \ref{lemmaind} we have that \eqref{eq:2.22} holds for $n+1=1,\dots,m$. Furthermore, we bound $\|\nabla u^m\|$ by using the Discrete Gr\"{o}nwall Lemma \ref{lemma:DGL}. 
We write \eqref{eq:2.22} in the form \eqref{eq:lemma6} in Lemma \ref{lemma:DGL}, where
\begin{align*}
    \xi_n = \|\nabla u^{n}\|^2, \qquad 
    \alpha_n = \frac{108}{\nu^3}K_1\|\nabla u^{n}\|^2, \qquad
    \eta_n = 2\frac{108}{\nu^3}K_1\bigg(\|\nabla u^{n}\|^2+\tau\frac{108}{\nu^3}K_1\|\nabla u^{n}\|^4\bigg),\qquad
    \zeta_n = \frac{18}{5\nu} \|f\|^2_{\infty}.
\end{align*}
We now compute the sums that appear in the right-hand side of \eqref{1star1}.
\begin{confidential}
    {\color{red}
    \[
\sum_{i=0}^{n-1}\tau\alpha_i = \sum_{i=0}^{n-1}\tau \frac{108}{\nu^3}K_1\|\nabla u^{i}\|^2 = \frac{108}{\nu^3}K_1 \sum_{i=0}^{n-1}\tau \|\nabla u^{i}\|^2.
\]

Note from \eqref{eq:boundnablaunp1}, we have 

 \begin{align}
    \frac{\tau\nu(2\theta-1)}{16\theta}\sum_{n=0}^{N-1}\|\nabla u^{n+1}\|^2
    \leq
    K_2 + N\tau K_3
    \end{align}

So

\[
\frac{108}{\nu^3}K_1 \sum_{n=0}^{N-1}\tau \|\nabla u^{n}\|^2 \leq \frac{108}{\nu^3}K_1\frac{16\theta}{\nu(2\theta-1)}(K_2+N\tau K_3)
\]
    }
\end{confidential}

Using the $L^2(H^1)$ bound \eqref{eq:boundnablaunp1} in Lemma \ref{lemma:gradu1}, we have
\begin{align}\label{march231}
    \sum_{i=0}^{n-1}\tau\alpha_i =  \frac{108}{\nu^3}K_1 \sum_{i=0}^{n-1}\tau \|\nabla u^{i}\|^2\leq \frac{108}{\nu^3}K_1\frac{16\theta}{\nu(2\theta-1)}(K_2+N\tau K_3)
\end{align}
and
\begin{confidential}
    {\color{red}
    \begin{align*}
    \sum_{j=i}^{n-1}\tau \alpha_j & 
    =
    \sum_{j=i}^{n-1}\tau\frac{108}{\nu^3}K_1\|\nabla u^{j}\|^2
    \\
    &
    = 
     \frac{108}{\nu^3}K_1\sum_{j=i}^{n-1}\tau\|\nabla u^{j}\|^2
     \\
     &
     \leq
     \frac{108}{\nu^3}K_1\sum_{j=0}^{n-1}\tau\|\nabla u^{j}\|^2
     \\
     &
     \leq
     \frac{108}{\nu^3}K_1\frac{16\theta}{\nu(2\theta-1)}(K_2+n\tau K_3)
\end{align*}
    }
\end{confidential}
\begin{align}\label{march232}
    \sum_{j=i}^{n-1}\tau \alpha_j & 
    =
     \frac{108}{\nu^3}K_1\sum_{j=i}^{n-1}\tau\|\nabla u^{j}\|^2
     \leq
     \frac{108}{\nu^3}K_1\sum_{j=0}^{n-1}\tau\|\nabla u^{j}\|^2
     \leq
     \frac{108}{\nu^3}K_1\frac{16\theta}{\nu(2\theta-1)}(K_2+n\tau K_3).
\end{align}

Now we use the fact that if $a_i\geq0$, $\sum_{i=0}^{n-1}a_i^2\leq \left(\sum_{i=0}^{n-1}a_i\right)^2$ and the bounds \eqref{march231} and \eqref{march232} to get
    \begin{align*}
    \sum_{i=0}^{n-1} \tau \eta_i
     &
     = 
     \sum_{i=0}^{n-1} \tau 2\frac{108}{\nu^3}K_1\left(\|\nabla u^{i}\|^2+\tau\frac{108}{\nu^3}K_1\|\nabla u^{i}\|^4\right)
     \\
     &
     =
     \sum_{i=0}^{n-1} \tau 2\frac{108}{\nu^3}K_1\|\nabla u^{i}\|^2
     +2\sum_{i=0}^{n-1}\left(\tau\frac{108}{\nu^3}K_1\|\nabla u^{i}\|^2\right)^2
     \\
     &
     \leq
     2\frac{108}{\nu^3}K_1\sum_{i=0}^{n-1} \tau \|\nabla u^{i}\|^2
     +2\left(\sum_{i=0}^{n-1}\tau\alpha_i\right)^2
     \\
     &
     = 
     2\sum_{i=0}^{n-1} \tau \alpha_i
     +2\left(\sum_{i=0}^{n-1}\tau\alpha_i\right)^2
     \\
     &
     \leq
     2 C_5(K_2+n\tau K_3) + 2(C_5(K_2+n\tau K_3))^2
     \\
     &
     =
     2C_5K_2+2C_5n\tau K_3+2C_5^2K_2^2+2C_5^2K_3^2(n\tau)^2+4C_5^2K_2K_3n\tau,
\end{align*}
\begin{confidential}
    {\color{red}
\begin{align*}
    \sum_{i=0}^{n-1} \tau \eta_i
     &
     =
     \sum_{i=0}^{n-1} \tau 2\frac{108}{\nu^3}K_1\|\nabla u^{i}\|^2
     +2\sum_{i=0}^{n-1}\left(\tau\frac{108}{\nu^3}K_1\|\nabla u^{i}\|^2\right)^2
     \\
     &
     \leq
     2\frac{108}{\nu^3}K_1\sum_{i=0}^{n-1} \tau \|\nabla u^{i}\|^2
     +2\left(\sum_{i=0}^{n-1}\tau\alpha_i\right)^2
     \\
     &
     \leq
     2 C_5(K_2+n\tau K_3) + 2(C_5(K_2+n\tau K_3))^2
     \\
     &
     =
     2C_5K_2+2C_5n\tau K_3+2C_5^2K_2^2+2C_5^2K_3^2(n\tau)^2+4C_5^2K_2K_3n\tau
\end{align*}
    }
\end{confidential}
and similarly,
\begin{align*}
    \sum_{j=i}^{n-1} \tau \eta_j & 
    \leq \sum_{j=0}^{n-1} \tau \eta_i
    \leq
    2C_5K_2+2C_5n\tau K_3+2C_5^2K_2^2+2C_5^2K_3^2(n\tau)^2+4C_5^2K_2K_3n\tau.
\end{align*}

Finally, we can apply the Discrete Gr\"{o}nwall Lemma \ref{lemma:DGL} to get the desired bound \eqref{2star} on $m$
\begin{confidential}
    {\color{red}
    \begin{align}
    \|\nabla u^m\|^2 
    &
    \leq
    \|\nabla u_0\|^2 \exp(C_5K_2+C_5K_3m\tau+2C_5K_2+2C_5K_3m\tau +2C_5^2K_2^2+2C_5^2K_3^2(m\tau)^2+4C_5^2K_2K_3m\tau)
    \\
    &
    +\sum_{i=1}^{m-1}\tau\frac{18}{5\nu}\|f\|_{\infty}^2\exp(C_5K_2+C_5K_3m\tau+2C_5K_2+2C_5K_3m\tau +2C_5^2K_2^2+2C_5^2K_3^2(m\tau)^2+4C_5^2K_2K_3m\tau)
    \\
    &
    +\tau\frac{18}{5\nu}\|f\|_{\infty}^2
    \\
    &
    \leq
     \|\nabla u_0\|^2\exp(C_6)\exp(C_7m\tau)\exp(C_8(m\tau)^2)
     \\
     &
     +m\tau\frac{18}{5\nu}\|f\|_{\infty}^2\exp(C_6)\exp(C_7m\tau)\exp(C_8(m\tau)^2)
     \\
     & 
     := K_5(\|\nabla u_0\|,\|f\|_\infty,m\tau) 
\end{align}
    }
\end{confidential}
\begin{align}\label{march233}
    \|\nabla u^m\|^2 
    &
    \leq
     K_5(\|\nabla u_0\|,\|f\|_\infty,m\tau).
\end{align}

We note that the bound $K_5$ in \eqref{march233} depends on the initial discrete value through $\|\nabla u_0\|$ and on $m$, but the dependence on $m$ is only through the time $m\tau$. So as long as $m\leq N = \lfloor T/\tau\rfloor$, the assumption \eqref{eq:bound} is satisfied, and the bound \eqref{eq:2.22} follows from Lemma \ref{lemmaind}.
\begin{confidential}
    {\color{red}
    \begin{align}
    \tau\frac{108}{\nu^3}K_1 &
\left[\left(\frac{2}{\nu C_3} +\frac{2}{\nu^2}\frac{1}{\lambda_1} \right)\|f\|^2_{\infty}+\left(1+\frac{C_4}{C_3}\right)\|\nabla u^{m}\|^2+\frac{108}{\nu^4}\frac{1}{\lambda_1} K_1\|\nabla u^m\|^4\right]
\\
&
\leq 
\frac{1}{15}
+\tau\frac{108}{\nu^3}K_1\left(1+\frac{C_4}{C_3}\right)K_5(\|\nabla u_0\|,\|f\|_{\infty},T)
+
\tau\frac{108}{\nu^3}K_1\frac{108}{\nu^4}\frac{1}{\lambda_1} K_1K_5^2(\|\nabla u_0\|,\|f\|_{\infty},T)
\\
&
\leq 
\frac{1}{15}+\frac{1}{15}+\frac{1}{15}=\frac{1}{5}
    \end{align}
    }
\end{confidential}
\end{proof}

For the reader's convenience, we now recall a result similar to Discrete Uniform Gr\"{o}nwall Lemma in \cite{MR2217369}.
\begin{lemma}[Discrete Uniform Gr\"{o}nwall Lemma]\label{lemma:DUGL}
        Given $\tau>0$, $n_1,n_2,n_*\in\mathbb{N}$ such that $n_1<n_*$ and $n_1+n_2+1\leq n_*$, given positive sequences $\xi_n,\eta_n,\zeta_n,\alpha_n$ such that
    \begin{equation}\label{eq:2.48}
        \xi_n\leq\xi_{n-1}(1+\tau\alpha_{n-1})(1+\tau\eta_{n-1})+\tau\zeta_n \qquad \text{for} \quad n=n_1,\dots n_* ,
    \end{equation}
    and given the bounds
    \begin{align*}
        \sum_{n=n'}^{n'+n_2}\tau\eta_n  \leq a_1(n_1,n_*),\qquad
        \sum_{n=n'}^{n'+n_2}\tau\zeta_n  \leq a_2(n_1,n_*),\qquad
        \sum_{n=n'}^{n'+n_2}\tau\xi_n  \leq a_3(n_1,n_*),\qquad
        \sum_{n=n'}^{n'+n_2}\tau\alpha_n  \leq a_4(n_1,n_*),
    \end{align*}
    for any $n'$ satisfying $n_1\leq n' \leq n_*-n_2$, we have
    \[
    \xi_n\leq \bigg(\frac{a_3(n_1,n_*)}{\tau n_2}+a_2(n_1,n_*)\bigg)e^{a_1(n_1,n_*)}e^{a_4(n_1,n_*)}
    \]
    for any $n$ such that $n_1+n_2+1\leq n\leq n_*.$
\end{lemma}

\begin{proof}
    Let $n_3,n_4\in\mathbb{N}$ be such that $n_1\leq n_3-1\leq n_4\leq n_2+n_3-1\leq n_*-1.$ Using \eqref{eq:2.48} recursively, we derive
\begin{confidential}
        {\color{red}
         \begin{align*}
        \xi_{n_2+n_3}&\leq 
        \xi_{n_4}\prod_{i=n_4}^{n_2+n_3-1}(1+\tau\eta_i)(1+\tau\alpha_i)+\sum_{i=n_4+1}^{n_2+n_3-1}\tau\zeta_i\prod_{j=i}^{n_2+n_3-1}(1+\tau\eta_j)(1+\tau\alpha_j)+\tau\zeta_{n_2+n_3}\\
        &\text{(Using the fact that $1+x\leq e^x$ for all $x\in\mathbb{R}$})\\
        & 
        \leq
        \xi_{n_4}\prod_{i=n_4}^{n_2+n_3-1}e^{\tau\eta_i}e^{\tau\alpha_i}+\sum_{i=n_4+1}^{n_2+n_3-1}\tau\zeta_i\prod_{j=i}^{n_2+n_3-1}e^{\tau\alpha_i}e^{\tau\eta_i}+\tau\zeta_{n_2+n_3}\\
        &
        \leq
        \xi_{n_4}e^{\sum_{i=n_4}^{n_2+n_3-1}\tau\eta_i}e^{\sum_{i=n_4}^{n_2+n_3-1}\tau\alpha_i}+\sum_{i=n_4+1}^{n_2+n_3-1}\tau\zeta_ie^{\sum_{i=n_4}^{n_2+n_3-1}\tau\eta_i}e^{\sum_{i=n_4}^{n_2+n_3-1}\tau\alpha_i}+\tau\zeta_{n_2+n_3}
    \end{align*}
    Note that by the assumptions of the lemma,
    \begin{align*}
        &\sum_{i=n_4}^{n_2+n_3-1}\tau\eta_i\leq\sum_{i=n_4}^{n_4+n_2}\tau\eta_i\leq a_1(n_1,n_*)\\
        &\sum_{i=n_4}^{n_2+n_3-1}\tau\alpha_i\leq\sum_{i=n_4}^{n_4+n_2}\tau\alpha_i\leq a_4(n_1,n_*)\\
        &\sum_{i=n_4}^{n_2+n_3-1}\tau\zeta_i\leq\sum_{i=n_4}^{n_4+n_2}\tau\zeta_i\leq a_2(n_1,n_*)\\
    \end{align*}
        }
\end{confidential}
    \begin{align*}
        \xi_{n_2+n_3}
        \leq
        \xi_{n_4}e^{\sum_{i=n_4}^{n_2+n_3-1}\tau\eta_i}e^{\sum_{i=n_4}^{n_2+n_3-1}\tau\alpha_i}+\sum_{i=n_4+1}^{n_2+n_3-1}\tau\zeta_ie^{\sum_{i=n_4}^{n_2+n_3-1}\tau\eta_i}e^{\sum_{i=n_4}^{n_2+n_3-1}\tau\alpha_i}+\tau\zeta_{n_2+n_3},
    \end{align*}
    so
    \[
     \xi_{n_2+n_3}\leq \xi_{n_4}e^{a_1}e^{a_4}+a_2e^{a_1}e^{a_4}=(\xi_{n_4}+a_2)e^{a_1}e^{a_4}.
    \]
    Multiplying this equation by $\tau$, summing from $n_4=n_3-1$ to $n_2+n_3-2$, and using the assumptions concludes the argument. 
\begin{confidential}
        {\color{red}
        \begin{align*}
        &\tau\xi_{n_2+n_3}\leq \tau(\xi_{n_4}+a_2)e^{a_1}e^{a_4}\\
        & \sum_{n_4=n_3-1}^{n_2+n_3-2}\tau\xi_{n_2+n_3}\leq \bigg(\sum_{n_4=n_3-1}^{n_2+n_3-2}\tau\xi_{n_4}+\sum_{n_4=n_3-1}^{n_2+n_3-2}\tau a_2\bigg)e^{a_1}e^{a_4}\\
        & \tau\xi_{n_2+n_3}n_2\leq \big(a_3+n_2\tau a_2\big)e^{a_1}e^{a_4}\\
         & \xi_{n_2+n_3}\leq \bigg(\frac{a_3}{\tau n_2}+ a_2\bigg)e^{a_1}e^{a_4}
    \end{align*}
    And since $n_1\leq n_3-1\leq n_4\leq n_2+n_3-1\leq n_*-1$, then if n is such that $n_1+n_2+1\leq n\leq n_*$, for any $n_2+n_3$ we can write $n$ in terms of $n_2+n_3$ (we can notice that by adding $n_2+1$ to the first chain of inequalities in this paragraph to get $n_1+n_2+1\leq n_3+n_2$ and adding 1 to get $n_2+n_3\leq n_*$).\\
    So
        }
\end{confidential}  
\end{proof}

We are now ready to give the main result, which is the long-time bound on $\|\nabla u^n\|$. We define the following constants
\begin{align}
    C_9 := 4K_4\frac{108}{\nu^3}\frac{16\theta}{\nu(2\theta-1)},\qquad
    C_{10} := Q_{4}\frac{108}{\nu^3}\frac{16\theta}{\nu(2\theta-1)},
    \qquad
    C_{11} := \frac{108}{\nu^3}\frac{2}{(2\theta-1)}(4\theta^2-6\theta+3).
\end{align}
Also,
\begin{align}\label{2starmar27}
    \rho_0= \frac{1 }{\lambda_1\nu^2(2\theta-1)}\|f\|_{\infty}^2,\qquad T_0=15\frac{1}{\lambda_1} \frac{1}{\nu(2\theta-1)}\ln\left(\frac{\|u_0\|^2}{\rho_0}\right),
\end{align}
and, for some fixed $r$, we let $\rho_1$ be such that 
\begin{align}\label{4star}
    \rho_1^2(&\|f\|_{\infty};r) :=
    \left(\frac{160\theta\rho_0K_4}{\nu r (2\theta-1)}+\frac{40\theta K_3}{\nu(2\theta-1)}+\frac{\nu^3(4\theta^2-6\theta+3)}{324(2\theta-1)r\rho_0}+r \frac{18}{5\nu} \|f\|^2_{\infty}\right)
    \\
    &
    \notag
    \times\exp\Bigg(2C_9 \rho_0^2 +2C_{10}\rho_0r+\frac{4}{15(2\theta-1)}(4\theta^2-6\theta+3)
    +
    8 C_9^2\rho_0^4 + 8 C_{10}^2\rho_0^2r^2
    +\frac{16}{15^2(2\theta-1)^2}(4\theta^2-6\theta+3)^2\Bigg)
    \\
    &
    \notag
    \times \exp \left(C_9 \rho_0^2 +C_{10}\rho_0r+\frac{2}{15(2\theta-1)}(4\theta^2-6\theta+3)\right).
\end{align}
We also let
\begin{align}\label{eq:K62}
    K_6(\|f\|_{\infty}):= K_5(\rho_1,\|f\|_{\infty},r) 
\end{align}
and
\begin{align}\label{eq:K72}
K_7(\|\nabla u_0\|, \|f\|_{\infty}) :=
     \max\{K_5(\|\nabla u_0\|, \|f\|_{\infty}, T_0+r), K_6(\|f\|_{\infty})\}.
\end{align}

\begin{theorem}\label{th:theorem}
    Let $u_0\in V$, $f\in L^{\infty}(\mathbb{R}_+;H)$, $u^n$ be the solution to \eqref{eq:BENSE}-\eqref{eq:FENSE}. Let $r\geq 4\kappa_1$ be arbitrarily fixed and let the time step $\tau$ be small, such that 
    \begin{align}\label{taubound}
    \tau\leq \min\{\kappa_1,\kappa_2(\|f\|_{\infty}),\kappa_3(\|\nabla u_0\|,\|f\|_{\infty},T_0+r),\kappa_4(\|\nabla u_0\|,\|f\|_{\infty},T_0+r),\kappa_3(\rho_1,\|f\|_{\infty},r),\kappa_4(\rho_1,\|f\|_{\infty},r)\}.
    \end{align}
    Then we have 
    \[
    \|\nabla u^n\|^2 \leq K_7(\|\nabla u_0\|,\|f\|_{\infty})
    \]
    for all $n\geq 0$, where $K_{7}(\cdot,\cdot)$ (defined in \eqref{eq:K72}) is a continuous function defined on $\mathbb{R}_+^2$, increasing in both arguments. Moreover, for $K_6(\|f\|_{\infty})$ as given in \eqref{eq:K62} we have
    \[
    \|\nabla u^n\|^2\leq K_6(\|f\|_{\infty})
    \]
    for all $n\geq N_0+N_r := \lfloor T_0/\tau\rfloor + \lfloor r/\tau\rfloor$, i.e., $\|\nabla u^n\|$ is bounded independently of $u_0$ beyond $N_0+N_r$.
\end{theorem}

\begin{proof}
    Let $r\geq 4\kappa_1$ and $\tau$ be such that assumption \eqref{taubound} holds. 

    First, we recall that from the uniform bound \eqref{eq:energy6} and \eqref{2starmar27}, provided $n\tau\geq T_0$,  we have
    \[
    \|u^{n}\|^2\leq 4\rho_0.
    \]
    By the small time step hypothesis \eqref{taubound}, $\tau$ satisfies also \eqref{eq:2.36} with $T=T_0+r$. Then, by Proposition \ref{prop2}, we get that \eqref{eq:2.22} holds for all $n+1=1,\dots,N_0+N_r$, and
    \[
    \|\nabla u^n\|^2\leq K_5(\|\nabla u_0\|,\|f\|_{\infty},n\tau)
    \]
    for all $n=1,\dots,N_0+N_r$. So for $\tau$ satisfying \eqref{taubound},
    \begin{align} 
    \|\nabla u^{n+1}\|^2 &\leq  \|\nabla u^{n}\|^2\big(1+\tau\frac{108}{\nu^3}K_1\|\nabla u^{n}\|^2\big)\bigg[1+2\tau\frac{108}{\nu^3}K_1\left(\|\nabla u^{n}\|^2+\tau\frac{108}{\nu^3}K_1\|\nabla u^{n}\|^4\right)\bigg] + \frac{18}{5\nu}\tau \|f\|^2_{\infty}
    \end{align}
for all $n+1=1,\dots,N_0+N_r$. 

We now apply the discrete uniform Gr\"{o}nwall Lemma \ref{lemma:DUGL} with 
\begin{align*}
    \xi_n = \|\nabla u^{n}\|^2,\qquad
    \alpha_n = \frac{108}{\nu^3}K_1\|\nabla u^{n}\|^2,\qquad
    \eta_n = 2\frac{108}{\nu^3}K_1\left(\|\nabla u^{n}\|^2+\tau\frac{108}{\nu^3}K_1\|\nabla u^{n}\|^4\right),\qquad
    \zeta_n = \frac{18}{5\nu} \|f\|^2_{\infty},
\end{align*}
$n_1=N_0+1$, $n_2 = N_r-2$, $n_*=N_0+N_r$.

We note that, since the sums $a_1(n_1,n_*)$, $a_2(n_1,n_*)$, $a_3(n_1,n_*)$, $a_4(n_1,n_*)$ are taken for $n\geq N_0+1$ and since $\tau\leq \kappa_1$, we can replace $K_1$ by $4\rho_0$ whenever the former appears. 

For every $n'=N_0+1,N_0+2$, we compute the sums required in the discrete uniform Gr\"{o}nwall Lemma \ref{lemma:DUGL}. We use the fact that
\[
\tau(n_2+1)=\tau(N_r-1) = \tau (\lfloor r/\tau\rfloor-1)\leq r.
\]
First, we note that
    \begin{align}\label{sum1}
    \sum_{n = n'}^{n'+n_2}\tau \zeta_n 
    = 
    \sum_{n = n'}^{n'+n_2} \tau \frac{18}{5\nu} \|f\|^2_{\infty}
    =
    \tau \frac{18}{5\nu} \|f\|^2_{\infty}(n_2+1)
    \leq
    r \frac{18}{5\nu} \|f\|^2_{\infty}.
\end{align}
\begin{confidential}
    {\color{red}
    \begin{align}
    \sum_{n = n'}^{n'+n_2}\tau \zeta_n &
    = 
    \sum_{n = n'}^{n'+n_2} \tau \frac{18}{5\nu} \|f\|^2_{\infty}
    \\
    &
    =
    \tau \frac{18}{5\nu} \|f\|^2_{\infty}\sum_{n = n'}^{n'+n_2}1
    \\
    &
    =
    \tau \frac{18}{5\nu} \|f\|^2_{\infty}(n_2+1)
    \\
    &
    =
    \tau \frac{18}{5\nu} \|f\|^2_{\infty}(N_r-1)
    \\
    &
    \leq
    \tau \frac{18}{5\nu} \|f\|^2_{\infty}\left\lfloor \frac{r}{\tau} \right\rfloor
    \\
    &
    \leq
    r \frac{18}{5\nu} \|f\|^2_{\infty}
\end{align}
    }
\end{confidential}
Secondly, using the $L^2(H^1)$ bound \eqref{eq:boundnablaunp1.1}, and the finite-time bound \eqref{2star}, we get
\begin{align}\label{sum2}
   \sum_{n = n'}^{n'+n_2} \tau \xi_n 
   &
   =
    \sum_{n = n'}^{n'+n_2} \tau\|\nabla u^{n}\|^2
    \\
    &
    \notag
    \leq
    4K_4 \frac{16\theta}{\nu(2\theta-1)}\rho_0 + \frac{16\theta}{\nu(2\theta-1)}K_3(n_2+1)\tau + \frac{16\theta}{\nu(2\theta-1)}\frac{\tau\nu}{8\theta}(4\theta^2-6\theta+3)\|\nabla u^{n'}\|^2
    \\
    &
    \notag
    \leq
    4K_4 \frac{16\theta}{\nu(2\theta-1)}\rho_0 + \frac{16\theta}{\nu(2\theta-1)}K_3r + \frac{16\theta}{\nu(2\theta-1)}\frac{\tau\nu}{8\theta}(4\theta^2-6\theta+3)K_5(\|\nabla u_0\|,\|f\|_{\infty},T_0+r).
\end{align}
\begin{confidential}
    {\color{red}
We have to justify why, in the last inequality, we can bound $\|\nabla u^{n'}\|^2$ by $K_5(\|\nabla u_0\|,\|f\|_{\infty},T_0+r)$. From the assumption \eqref{taubound} of the theorem, assumption \eqref{eq:2.36} of \eqref{prop2} holds for $T=T_0+r$. Therefore, $\|\nabla u^n\|^2\leq K_5(\|\nabla u_0\|,\|f\|_{\infty},n\tau)$ for all $n=1,\dots,N:=\lfloor (T_0+r)/\tau\rfloor$. Since $K_5$ is monotonically increasing in all of its arguments, we have that $\|\nabla u^n\|^2\leq K_5(\|\nabla u_0\|,\|f\|_{\infty},T_0+r)$ for all $n=1,\dots,N:=\lfloor (T_0+r)/\tau\rfloor$.\\
Note that by assumption, $n'\leq n_*-n_2 = N_0+N_r-N_r+2=N_0+2$. So for my inequality to hold, I need that $(N_0+2)\leq N_0+N_r$. This is true if and only if $2\leq N_r=\lfloor\frac{r}{\tau}\rfloor$. And recall that we had fixed $r\geq 4\kappa_1$, where $\tau\leq\kappa_1$. So $\frac{r}{\tau}\geq\frac{4\kappa_1}{\tau}\geq\frac{4\kappa_1}{\kappa}=4$. And $\lfloor\frac{r}{\tau}\rfloor\geq \frac{r}{\tau}-1\geq 3>2$. So my inequality is true.
    }
\end{confidential}
Similarly, using the $L^2(H^1)$ bound \eqref{eq:boundnablaunp1.1}, and the finite-time bound \eqref{2star}, we get 
\begin{align}\label{sum3}
   \sum_{n = n'}^{n'+n_2} \tau \alpha_n  
   &
   =
   \sum_{n = n'}^{n'+n_2} \tau \frac{108}{\nu^3}K_1\|\nabla u^{n}\|^2
   \leq
   C_9 \rho_0^2 +C_{10}\rho_0r+\tau C_{11}\rho_0K_5(\|\nabla u_0\|,\|f\|_{\infty},T_0+r).
\end{align}
Finally, using the fact that for $a,b,c\geq0$, 
\[
(a+b+c)^2\leq 2(a+b)^2+2c^2\leq 4a^2+4b^2+2c^2,
\]
the fact that if $a_i\geq0$, $\sum_{i=0}^{n-1}a_i^2\leq \left(\sum_{i=0}^{n-1}a_i\right)^2$, and the $L^2(H^1)$ bound \eqref{sum3}, we get 
\begin{align}\label{sum4}
    \sum_{n = n'}^{n'+n_2} \tau \eta_n 
    &
    =
    \sum_{n = n'}^{n'+n_2} \tau  2\frac{108}{\nu^3}K_1\|\nabla u^{n}\|^2+ \sum_{n = n'}^{n'+n_2}2\left(\tau\frac{108}{\nu^3}K_1\|\nabla u^{n}\|^2\right)^2
    \notag
    \\
    &
    \leq
    2\sum_{n = n'}^{n'+n_2} \tau  \frac{108}{\nu^3}K_1\|\nabla u^{n}\|^2+ 2\left(\sum_{n = n'}^{n'+n_2}\tau\frac{108}{\nu^3}K_1\|\nabla u^{n}\|^2\right)^2
    \\
    &
    \notag
    \leq
    2C_9 \rho_0^2 +2C_{10}\rho_0r+2\tau C_{11}\rho_0K_5(\|\nabla u_0\|,\|f\|_{\infty},T_0+r)
    \\
    & 
    \notag
    \qquad +
    8 C_9^2\rho_0^4 + 8 C_{10}^2\rho_0^2r^2+4\tau^2C_{11}^2\rho_0^2K_5^2(\|\nabla u_0\|,\|f\|_{\infty},T_0+r).
\end{align}
\begin{confidential}
    {\color{red}
    So we have 

\begin{align}
    a_1(n_1,n_*)
    &= 2C_9 \rho_0^2 +2C_{10}\rho_0r+2\tau C_{11}\rho_0K_5(\|\nabla u_0\|,\|f\|_{\infty},T_0+r)
    +
    8 C_9^4\rho_0^2 + 8 C_{10}^2\rho_0^2r^2+4\tau^2C_{11}^2\rho_0^2K_5^2(\|\nabla u_0\|,\|f\|_{\infty},T_0+r)
    \\
    a_2 (n_1,n_*)
    &
    =
    r \frac{18}{5\nu} \|f\|^2_{\infty}
    \\
    a_3(n_1,n_*) 
    &
    =
    4K_4 \frac{16\theta}{\nu(2\theta-1)}\rho_0 + \frac{16\theta}{\nu(2\theta-1)}K_3r + \frac{16\theta}{\nu(2\theta-1)}\frac{\tau\nu}{8\theta}(4\theta^2-6\theta+3)K_5(\|\nabla u_0\|,\|f\|_{\infty},T_0+r)
    \\
    a_4(n_1,n_*) 
    &
    =
    C_9 \rho_0^2 +C_{10}\rho_0r+\tau C_{11}\rho_0K_5(\|\nabla u_0\|,\|f\|_{\infty},T_0+r)
\end{align}
    }
\end{confidential}

We now note that the small time step assumption \eqref{taubound} implies
\[
\tau\leq \kappa_3(\|\nabla u_0\|,\|f\|_{\infty},T_0+r),\qquad\tau\leq \kappa_4(\|\nabla u_0\|,\|f\|_{\infty},T_0+r),
\]
and therefore, we have the following estimates
\begin{align*}
    \tau\frac{108}{\nu^3}\rho_0 K_5(\|\nabla u_0\|,\|f\|_{\infty},T_0+r) &\leq \frac{1}{15},
    \\
    \tau^2\frac{108^2}{\nu^6}\rho_0^2 K_5^2(\|\nabla u_0\|,\|f\|_{\infty},T_0+r) &\leq \frac{1}{15^2},
    \\
    \tau K_5(\|\nabla u_0\|,\|f\|_{\infty},T_0+r) &\leq \frac{\nu^3}{15\cdot 108 \rho_0}.
\end{align*}
Using \eqref{sum1}, \eqref{sum2}, \eqref{sum3}, and \eqref{sum4}, the above estimates yield the sums defined in Lemma \ref{lemma:DUGL} 
\begin{align*}
    a_1(n_1,n_*)
    &= 2C_9 \rho_0^2 +2C_{10}\rho_0r+\frac{4}{15(2\theta-1)}(4\theta^2-6\theta+3)
    +
    8 C_9^2\rho_0^4 + 8 C_{10}^2\rho_0^2r^2+\frac{16}{15^2(2\theta-1)^2}(4\theta^2-6\theta+3)^2,
    \\
    a_2 (n_1,n_*)
    &
    =
    r \frac{18}{5\nu} \|f\|^2_{\infty},
    \\
    a_3(n_1,n_*) 
    &
    =
    4K_4 \frac{16\theta}{\nu(2\theta-1)}\rho_0 + \frac{16\theta}{\nu(2\theta-1)}K_3r + \frac{2}{(2\theta-1)}(4\theta^2-6\theta+3)\frac{\nu^3}{15\cdot 108\rho_0},
    \\
    a_4(n_1,n_*) 
    &
    =
    C_9 \rho_0^2 +C_{10}\rho_0r+\frac{2}{15(2\theta-1)}(4\theta^2-6\theta+3).
\end{align*}

 By Proposition \ref{prop2} regarding the $V$-bound on a finite time interval, we get that \eqref{eq:2.22} holds for $n+1 = N_0+N_r$. Therefore, Lemma \ref{lemma:DUGL} gives
\begin{align}\label{eq:3.75}
    \|\nabla u^{N_0+N_r}\|^2 \leq \left(\frac{a_3(n_1,n_*) }{\tau n_2}+a_2(n_1,n_*) \right)\exp(a_1(n_1,n_*) )\exp(a_4(n_1,n_*) ). 
\end{align}
Since by assumption $r\geq 4\kappa_1$,  
\[
\frac{1}{\tau n_2}=\frac{1}{\tau(N_r-2)} \leq \frac{5}{2r}. 
\]
\begin{confidential}
    {\color{red}
    Why is that?
    Actually we have $\frac{1}{\tau n_2}\leq\frac{A}{r}$ for all $A\geq\frac{5}{2}$. By plotting this inequality, it becomes clear that the case $A=2$ (which Tone uses) is not true in general. For example, if we pick $\nu$ such that $\kappa_1=\frac{1}{\nu\lambda_1}\geq 3$, $r=3$, and $\tau=3$, we satisfy that $12\leq 4\kappa_1\leq r$ and $\tau = 3\leq \kappa_1$. And it is a straightforward calculation to show that the intended inequality does not hold for these values.

    Let's prove it for $A\geq \frac{5}{2}$. We want to show $\frac{1}{\tau n_2}\leq\frac{A}{r}$ for all $A\geq\frac{5}{2}$. Since $r\geq4\kappa_1\geq 4\tau$, $N_r-2=\lfloor r/\tau\rfloor-2\geq \lfloor 4\tau/\tau\rfloor =2>0$, so we can multiply by $N_r-2$ to get that the above is equivalent to showing
    \[
    \frac{1}{\tau}\leq\frac{A}{r}\left(\lfloor r/\tau\rfloor-2\right) \Leftrightarrow \frac{1}{\tau}-\frac{A}{r}\lfloor r/\tau\rfloor +\frac{2A}{r}\leq 0
    \]
    Multiplying by $r$, this is equivalent to showing
    \[
    \frac{r}{\tau}-A\lfloor r/\tau\rfloor+2A\leq0 \Leftrightarrow \frac{r}{\tau}-\lfloor r/\tau\rfloor+(1-A)\lfloor r/\tau\rfloor+2A\leq 0
    \]
    And it is true that 
    \[
    \frac{r}{\tau}-\lfloor r/\tau\rfloor<1,
    \]
     so
     \[
     \frac{r}{\tau}-\lfloor r/\tau\rfloor+(1-A)\lfloor r/\tau\rfloor+2A\leq 1+(1-A)\lfloor r/\tau\rfloor+2A.
     \]
     Now $r\geq 4\tau$, so $\lfloor r/\tau\rfloor\geq 4$, and so $-\lfloor r/\tau\rfloor\leq -4$. So
     \[
     1+(1-A)\lfloor r/\tau\rfloor+2A=1-(A-1)\lfloor r/\tau\rfloor+2A\leq 1-4(A-1)+2A=5-2A
     \]

    I want $5-2A\leq 0$, which happens iff $A\geq\frac{5}{2}$.
    }
\end{confidential}
So the bound \eqref{eq:3.75} becomes
\begin{align*}
    \|\nabla u^{N_0+N_r}\|^2 
    &
    \leq
    \rho_1^2(\|f\|_{\infty};r).
\end{align*}

Note that we assumed in \eqref{taubound} that $\tau\leq \kappa_3(\rho_1,\|f\|_{\infty},r)$ and $\tau\leq \kappa_4(\rho_1,\|f\|_{\infty},r)$, where $\kappa_3,\kappa_4$ are decreasing functions of their arguments. So we can think of $u^{N_0+N_r}$ as our initial data and apply Proposition \ref{prop2} with $T=r$. We then get that \eqref{eq:2.22} holds for all $n= N_0+N_r+1,\dots,N_0+2N_r$, and
\begin{confidential}
    {\color{red}
    So

\begin{align}
    \kappa_3(\rho_1,\|f\|_{\infty},r) & \leq \kappa_3(\|\nabla u^{N_0+N_r}\|,\|f\|_{\infty},r)
    \\
    \kappa_4(\rho_1,\|f\|_{\infty},r) & \leq \kappa_4(\|\nabla u^{N_0+N_r}\|,\|f\|_{\infty},r)
\end{align}

    }
\end{confidential}
\begin{align*}
    \|\nabla u^n\|^2\leq K_5(\|\nabla u^{N_0+N_r}\|,\|f\|_{\infty},N_r\tau)
\end{align*}
for all $n= N_0+N_r+1,\dots,N_0+2N_r$.

Since $\|\nabla u^{N_0+N_r}\|^2\leq \rho_1^2(\|f\|_{\infty};r)$ and $K_5(\cdot,\cdot,\cdot)$ is an increasing function of all its arguments, we get that
\begin{align*}
    \|\nabla u^n\|^2\leq K_5(\rho_1,\|f\|_{\infty},N_r\tau)
\end{align*}
for all $n= N_0+N_r+1,\dots,N_0+2N_r$.

Now applying Lemma \ref{lemma:DUGL} with $n_1=N_0+N_r+1, n_2=N_r-2$, and $n_*=N_0+2N_r$,
\begin{align*}
    \|\nabla u^{N_0+2N_r}\|^2\leq \rho_1^2(\|f\|_{\infty};r).
\end{align*}

Iterating Proposition \ref{prop2} and Lemma \ref{lemma:DUGL} in this manner, we get 
\begin{align*}
    \|\nabla u^n\|^2\leq K_5(\rho_1,\|f\|_{\infty},r)= K_6(\|f\|_{\infty})
\end{align*}
for all $n\geq N_0+N_r$, and recalling our initial bound \eqref{2star} on a finite interval, we finally obtain
\begin{align*}
    \|\nabla u^n\|^2 \leq \max\{K_5(\|\nabla u_0\|, \|f\|_{\infty}, T_0+r), K_6(\|f\|_{\infty})\} = K_7(\|\nabla u_0\|, \|f\|_{\infty})
\end{align*}
for all $n\geq 1$.
\end{proof}

% Section 6 - Conclusions
\section{Conclusions}\label{sec:conclusion}
We proved the uniform-in-time stability of the Cauchy one-leg time-stepping method for the two-dimensional Navier-Stokes equations. In particular, we established conditions under which this stability holds for the semi-discrete-in-time formulation, including a time-step restriction \eqref{taubound} and a constraint on the method parameter $\theta\in\left(\frac{1}{2},1\right)$. 

There are two key arguments that allowed us to extend the result in \cite{MR2217369}. First, in Lemma \ref{lemma:gradu1}, we established an $L^2(H^1)$ bound at the integer time levels. A similar bound at the fractional time steps $u^{n+\theta}$, which arises naturally as an extension of the BE case, is not sufficient to prove the $V$-stability argument. The second key result is the reformulation of the energy estimate in Lemma \ref{lemma:BEstab} as presented in Lemma \ref{lemma:energy-equality}. This step is essential, as it shows that this type of argument cannot be applied to the case $\theta=\frac{1}{2}$ without imposing a time-step restriction. Therefore, proving long-time stability of the semi-discretization in time for the case $\theta=\frac{1}{2}$, which corresponds to the midpoint method, remains an open problem.

% Section 6 - Funding
\section{Funding}
Isabel Barrio Sanchez and Catalin Trenchea were partially supported by the National Science Foundation under grant DMS-2208220.

\begin{confidential}
    {\color{red}
\section{Appendix: Computing $\alpha,\varepsilon, a,b$}\label{sec:appendix1}
We want to find 
$\alpha, \varepsilon > 0$ and $a,b\in \mathbb{R}$ such that
\begin{align*}
& 
\alpha + \varepsilon + a^2 
= 
 \theta^2  \Big( 2 + \lambda_1 \nu \theta \tau \Big) 
\\
& 
b^2 - \alpha 
= 
	( \theta - 1) \bigg[ 
\Big( 1 + \lambda_1 \nu \theta \tau \Big) ( \theta - 1 ) + ( \theta + 1 ) \bigg] 
\\
& 
2 ab 
= 
2\theta \bigg[  \theta - (1 - \theta) \Big( 1 + \lambda_1 \nu \theta \tau \Big) \bigg].
\end{align*}

Adding these relations we get

    {\color{red}
    \begin{align*}
& 
(a + b)^2 + \varepsilon
= 
 \theta^2  \Big( 2 + \lambda_1 \nu \theta \tau \Big) 
+
	( \theta - 1) \bigg[ 
\Big( 1 + \lambda_1 \nu \theta \tau \Big) ( \theta - 1 ) + ( \theta + 1 ) \bigg] 
\\
&
 + 2\theta \bigg[  \theta - (1 - \theta) \Big( 1 + \lambda_1 \nu \theta \tau \Big) \bigg]
 \\
 & 
 = 
 \theta^2 
 +  \theta^2  \Big( 1 + \lambda_1 \nu \theta \tau \Big) 
	+ \Big( 1 + \lambda_1 \nu \theta \tau \Big) ( 1 - \theta)^2 + ( \theta^2 - 1 )
\\
&
 + 2\theta^2 
 - 2\theta (1 - \theta) \Big( 1 + \lambda_1 \nu \theta \tau \Big)
 \\
 & 
 = 
4 \theta^2 -1
 + \big[  \theta^2  + ( 1 - \theta)^2  - 2\theta (1 - \theta)  \big] 
 \Big( 1 + \lambda_1 \nu \theta \tau \Big) 
 \\
 &
 = 
4 \theta^2 - 1
 + ( 2\theta - 1 )^2   \Big( 1 + \lambda_1 \nu \theta \tau \Big),
\end{align*}
    }

\begin{align*}
& 
(a + b)^2 + \varepsilon
= 
4 \theta^2 - 1
 + ( 2\theta - 1 )^2   \Big( 1 + \lambda_1 \nu \theta \tau \Big),
\end{align*}

i.e., 
\begin{align*}
a + b
& 
= 
\sqrt{ 8 \theta^2  - 4 \theta 
 - \varepsilon
+ ( 2\theta - 1 )^2 \lambda_1 \nu \theta \tau 
 }
\\
ab 
& 
= 
\theta \bigg[  2 \theta - 1 - (1 - \theta)\lambda_1 \nu \theta \tau  \bigg]
.
\end{align*}

    {\color{red}
    \begin{align*}
a + b
& 
= 
\sqrt{ 4 \theta^2 -1 
 + ( 2\theta - 1 )^2   \Big( 1 + \lambda_1 \nu \theta \tau \Big)
 - \varepsilon }
\\
& 
= 
\sqrt{ 4 \theta^2 -1
 + ( 2\theta - 1 )^2  
 + ( 2\theta - 1 )^2 \lambda_1 \nu \theta \tau 
 - \varepsilon }
 \\
 &
= 
\sqrt{ 8 \theta^2  - 4 \theta 
 - \varepsilon
+ ( 2\theta - 1 )^2 \lambda_1 \nu \theta \tau 
 }
\\
ab 
& 
= 
\theta \bigg[  \theta - (1 - \theta) \Big( 1 + \lambda_1 \nu \theta \tau \Big) \bigg]
\\
& 
= 
\theta \bigg[  2 \theta - 1 - (1 - \theta)\lambda_1 \nu \theta \tau  \bigg]
.
\end{align*}
    }

If $a+b=c$ and $ab=d$, then $a$ and $b$ are roots of $x^2-cx+d=0$. Therefore, $a$ and $b$ are the roots of the quadratic equation

    {\color{red}
    This follows directly from using the quadratic formula and plugging $a+b$ in place of $c$ and $ab$ in place of d.
    }

\begin{align*}
x^2 - \bigg( 8 \theta^2 - 4 \theta  - \varepsilon
 + ( 2\theta - 1 )^2    \lambda_1 \nu \theta \tau
 \bigg)^{1/2} x
 + \theta \bigg[  2\theta - 1 - (1 - \theta ) \lambda_1 \nu \theta \tau \Big) \bigg] 
 = 0,
\end{align*}
which has discriminant

    {\color{red}
    \begin{align*}
\Delta 
& 
= 
\bigg( 8 \theta^2 - 4 \theta - \varepsilon 
 + ( 2\theta - 1 )^2 \lambda_1 \nu \theta \tau
\bigg)
- 4 \theta \bigg[ 2 \theta - 1 - ( 1 - \theta) \lambda_1 \nu \theta \tau \bigg] 
\\
& 
= 
\cancel{8 \theta^2} - \bcancel{4 \theta} - \varepsilon 
 + ( 2\theta - 1 )^2 \lambda_1 \nu \theta \tau
- \cancel{8 \theta^2}  + \bcancel{4 \theta} + 4\theta ( 1 - \theta) \lambda_1 \nu \theta \tau 
\\
& 
= 
- \varepsilon 
 + \big( 4\theta^ 2 - 4 \theta  + 1 + 4\theta ( 1 - \theta)\big) \lambda_1 \nu \theta \tau
\\
& 
= 
- \varepsilon 
 + \lambda_1 \nu \theta \tau
 ,
 \end{align*}
    }

\begin{align*}
\Delta 
= 
- \varepsilon 
 + \lambda_1 \nu \theta \tau
 ,
 \end{align*}
 
The roots $x_{1,2}$ of the quadratic polynomial are
\begin{align*}
x_{1,2} 
& 
= 
\frac{1}{2} \Big\{
 \Big( 8 \theta^2 - 4 \theta  - \varepsilon
 + ( 2\theta - 1 )^2    \lambda_1 \nu \theta \tau \Big)^{1/2}
\pm 
\Big( - \varepsilon  + \lambda_1 \nu \theta \tau \Big)^{1/2}
  \Big\},
\end{align*}
so

\begin{align*}
x_{1,2}^2
&
 = 
2 \theta^2 - \theta  - \frac{1}{2} \varepsilon
 + ( \theta^2 - \theta + \frac{1}{2} )    \lambda_1 \nu \theta \tau 
\\
& \qquad
\pm 
\frac{1}{2}
\Big( 8 \theta^2 - 4 \theta  - \varepsilon
 + ( 2\theta - 1 )^2    \lambda_1 \nu \theta \tau \Big)^{1/2}
\Big( - \varepsilon  + \lambda_1 \nu \theta \tau \Big)^{1/2}
.
\end{align*}

    {\color{red}
    \begin{align*}
x_{1,2}^2
&
 = 
\frac{1}{4} \Big\{
 8 \theta^2 - 4 \theta  - \varepsilon
 + ( 2\theta - 1 )^2    \lambda_1 \nu \theta \tau 
+
\Big( - \varepsilon  + \lambda_1 \nu \theta \tau \Big)
\\
& \qquad
\pm 2
\Big( 8 \theta^2 - 4 \theta  - \varepsilon
 + ( 2\theta - 1 )^2    \lambda_1 \nu \theta \tau \Big)^{1/2}
\Big( - \varepsilon  + \lambda_1 \nu \theta \tau \Big)^{1/2}
 \Big\}
\\
&
= 
\frac{1}{4} \Big\{
 8 \theta^2 - 4 \theta  - 2 \varepsilon
 + ( 4\theta^2 - 4 \theta + 2 )    \lambda_1 \nu \theta \tau 
\\
& \qquad
\pm 2
\Big( 8 \theta^2 - 4 \theta  - \varepsilon
 + ( 2\theta - 1 )^2    \lambda_1 \nu \theta \tau \Big)^{1/2}
\Big( - \varepsilon  + \lambda_1 \nu \theta \tau \Big)^{1/2}
 \Big\}
\\
&
= 
2 \theta^2 - \theta  - \frac{1}{2} \varepsilon
 + ( \theta^2 - \theta + \frac{1}{2} )    \lambda_1 \nu \theta \tau 
\\
& \qquad
\pm 
\frac{1}{2}
\Big( 8 \theta^2 - 4 \theta  - \varepsilon
 + ( 2\theta - 1 )^2    \lambda_1 \nu \theta \tau \Big)^{1/2}
\Big( - \varepsilon  + \lambda_1 \nu \theta \tau \Big)^{1/2}
.
\end{align*} 
    }

We set 
\begin{align}
\label{eq:a,b}
& 
a 
:=
\frac{1}{2} \Big\{
 \Big( 8 \theta^2 - 4 \theta  - \varepsilon
 + ( 2\theta - 1 )^2    \lambda_1 \nu \theta \tau \Big)^{1/2}
- 
\Big( - \varepsilon  + \lambda_1 \nu \theta \tau \Big)^{1/2}
  \Big\}
,
\\
& 
b
:= 
\frac{1}{2} \Big\{
 \Big( 8 \theta^2 - 4 \theta  - \varepsilon
 + ( 2\theta - 1 )^2    \lambda_1 \nu \theta \tau \Big)^{1/2}
+ 
\Big( - \varepsilon  + \lambda_1 \nu \theta \tau \Big)^{1/2}
  \Big\}
,
\end{align}
so 
\begin{align*} 
b^2
& 
= 
2 \theta^2 - \theta  - \frac{1}{2} \varepsilon
 + ( \theta^2 - \theta + \frac{1}{2} )    \lambda_1 \nu \theta \tau 
\\
& \qquad
+ 
\frac{1}{2}
\Big( 8 \theta^2 - 4 \theta  - \varepsilon
 + ( 2\theta - 1 )^2    \lambda_1 \nu \theta \tau \Big)^{1/2}
\Big( - \varepsilon  + \lambda_1 \nu \theta \tau \Big)^{1/2}
,
\\
a^2
& 
= 
2 \theta^2 - \theta  - \frac{1}{2} \varepsilon
 + ( \theta^2 - \theta + \frac{1}{2} )    \lambda_1 \nu \theta \tau 
\\
& \qquad
-
\frac{1}{2}
\Big( 8 \theta^2 - 4 \theta  - \varepsilon
 + ( 2\theta - 1 )^2    \lambda_1 \nu \theta \tau \Big)^{1/2}
\Big( - \varepsilon  + \lambda_1 \nu \theta \tau \Big)^{1/2}
.
\end{align*}
${\empty}$
\\

Therefore

\begin{align*}
\alpha 
& 
= 
b^2 
- 
	( \theta - 1) \bigg[ 
\Big( 1 + \lambda_1 \nu \theta \tau \Big) ( \theta - 1 ) + ( \theta + 1 ) \bigg] 
\\
& 
= 
 \theta 
- \frac{1}{2} \varepsilon	
 + \Big( \theta -  \frac{1}{2} \Big)    \lambda_1 \nu \theta \tau 
\\
& \qquad
+ 
\frac{1}{2}
\Big( 8 \theta^2 - 4 \theta  - \varepsilon
 + ( 2\theta - 1 )^2    \lambda_1 \nu \theta \tau \Big)^{1/2}
\Big( - \varepsilon  + \lambda_1 \nu \theta \tau \Big)^{1/2}
.
\end{align*}

    {\color{red}
    \begin{align*}
\alpha 
& 
= 
b^2 
- 
	( \theta - 1) \bigg[ 
\Big( 1 + \lambda_1 \nu \theta \tau \Big) ( \theta - 1 ) + ( \theta + 1 ) \bigg] 
\\
& 
= 
2 \theta^2 - \theta  - \frac{1}{2} \varepsilon
 + ( \theta^2 - \theta + \frac{1}{2} )    \lambda_1 \nu \theta \tau 
\\
& \qquad
+ 
\frac{1}{2}
\Big( 8 \theta^2 - 4 \theta  - \varepsilon
 + ( 2\theta - 1 )^2    \lambda_1 \nu \theta \tau \Big)^{1/2}
\Big( - \varepsilon  + \lambda_1 \nu \theta \tau \Big)^{1/2}
\\ 
& \qquad
- 
	( \theta - 1) \bigg[ 
\Big( 1 + \lambda_1 \nu \theta \tau \Big) ( \theta - 1 ) + ( \theta + 1 ) \bigg] 
\\
& 
= 
2 \theta^2 - \theta  - \frac{1}{2} \varepsilon
 + ( \theta^2 - \theta + \frac{1}{2} )    \lambda_1 \nu \theta \tau 
\\
& \qquad
+ 
\frac{1}{2}
\Big( 8 \theta^2 - 4 \theta  - \varepsilon
 + ( 2\theta - 1 )^2    \lambda_1 \nu \theta \tau \Big)^{1/2}
\Big( - \varepsilon  + \lambda_1 \nu \theta \tau \Big)^{1/2}
\\ 
& \qquad
	- ( \theta - 1)^2 
	- ( \theta - 1)^2 \lambda_1 \nu \theta \tau 
	- ( \theta^2 - 1 )
\\
& 
= 
2 \theta^2 - \theta  
	- ( \theta - 1)^2 
	- ( \theta^2 - 1 )
- \frac{1}{2} \varepsilon	
\\ 
& \qquad
 - ( \theta - 1)^2 \lambda_1 \nu \theta \tau 
 + ( \theta^2 - \theta + \frac{1}{2} )    \lambda_1 \nu \theta \tau 
\\
& \qquad
+ 
\frac{1}{2}
\Big( 8 \theta^2 - 4 \theta  - \varepsilon
 + ( 2\theta - 1 )^2    \lambda_1 \nu \theta \tau \Big)^{1/2}
\Big( - \varepsilon  + \lambda_1 \nu \theta \tau \Big)^{1/2}
\\
& 
= 
\cancel{2 \theta^2} - \theta  
	- \cancel{\theta^2} + 2 \theta - \xcancel{1}
	- \cancel{\theta^2} + \xcancel{1}
- \frac{1}{2} \varepsilon	
\\ 
& \qquad
 + \Big( - \bcancel{\theta^2} + 2 \theta - 1 + \bcancel{\theta^2} - \theta + \frac{1}{2} \Big)    \lambda_1 \nu \theta \tau 
\\
& \qquad
+ 
\frac{1}{2}
\Big( 8 \theta^2 - 4 \theta  - \varepsilon
 + ( 2\theta - 1 )^2    \lambda_1 \nu \theta \tau \Big)^{1/2}
\Big( - \varepsilon  + \lambda_1 \nu \theta \tau \Big)^{1/2}
\\
& 
= 
 \theta 
- \frac{1}{2} \varepsilon	
 + \Big( \theta -  \frac{1}{2} \Big)    \lambda_1 \nu \theta \tau 
\\
& \qquad
+ 
\frac{1}{2}
\Big( 8 \theta^2 - 4 \theta  - \varepsilon
 + ( 2\theta - 1 )^2    \lambda_1 \nu \theta \tau \Big)^{1/2}
\Big( - \varepsilon  + \lambda_1 \nu \theta \tau \Big)^{1/2}
.
\end{align*}
    }

A direct computation shows that
\begin{align}
\label{eq:check}
& 
\alpha + \varepsilon + a^2 
= 
 \theta^2  \Big( 2 + \lambda_1 \nu \theta \tau \Big) .
\end{align}

{\color{red}
Indeed,
\begin{align*}
\alpha + \varepsilon + a^2 
&
= 
 \theta 
- \frac{1}{2} \varepsilon	
 + \Big( \theta -  \frac{1}{2} \Big)    \lambda_1 \nu \theta \tau 
\\
& \qquad
+ 
\frac{1}{2}
\Big( 8 \theta^2 - 4 \theta  - \varepsilon
 + ( 2\theta - 1 )^2    \lambda_1 \nu \theta \tau \Big)^{1/2}
\Big( - \varepsilon  + \lambda_1 \nu \theta \tau \Big)^{1/2}
\\
& \qquad
+ \varepsilon 
\\
& \qquad
+
2 \theta^2 - \theta  - \frac{1}{2} \varepsilon
 + ( \theta^2 - \theta + \frac{1}{2} )    \lambda_1 \nu \theta \tau 
\\
& \qquad
-
\frac{1}{2}
\Big( 8 \theta^2 - 4 \theta  - \varepsilon
 + ( 2\theta - 1 )^2    \lambda_1 \nu \theta \tau \Big)^{1/2}
\Big( - \varepsilon  + \lambda_1 \nu \theta \tau \Big)^{1/2}
\\
&
= 
 \bcancel{\theta }
- \cancel{\frac{1}{2} \varepsilon	}
 + \Big( \xcancel{\theta -  \frac{1}{2}} \Big)    \lambda_1 \nu \theta \tau 
\\
& \qquad
+ \cancel{\varepsilon }
\\
& \qquad
+
2 \theta^2 - \bcancel{\theta}  - \cancel{\frac{1}{2} \varepsilon}
 + ( \theta^2 - \xcancel{\theta + \frac{1}{2}} )    \lambda_1 \nu \theta \tau 
\\
&
= 
2 \theta^2  + \theta^2  \lambda_1 \nu \theta \tau 
.
\end{align*}
which verifies \eqref{eq:check}.
}

Since we are working with square roots, we need 
$$
- \varepsilon +\lambda_1 \nu \theta \tau
\geq 0,
$$
and 
$$ 8 \theta^2 - 4 \theta  - \varepsilon
 + ( 2\theta - 1 )^2    \lambda_1 \nu \theta \tau \geq 0$$
That is, we need 
$$
\varepsilon 
\leq \lambda_1 \nu \tau\theta,
$$
and
$$\varepsilon\leq 8 \theta^2 - 4 \theta 
 + ( 2\theta - 1 )^2    \lambda_1 \nu \theta \tau $$
The following choice of $\varepsilon$ satisfies the above two inequalities (provided that \eqref{eq:step-condition} holds).\\
\begin{align*}
& 
\varepsilon 
:= 
\lambda_1 \nu ( 2\theta - 1)  \tau
,
\end{align*}

{\color{red}
Proof that $\varepsilon$ satisfies the inequalities:\\
The first inequality is straightforward since $(2\theta-1)\leq 1$ for all $\theta\in(\frac{1}{2},1]$
\begin{align*}
    & \lambda_1 \nu ( 2\theta - 1)  \tau \leq 8 \theta^2 - 4 \theta 
 + ( 2\theta - 1 )^2    \lambda_1 \nu \theta \tau\\
 &\iff \frac{\nu\tau}{\frac{1}{\lambda_1}} \Bigg[(2\theta-1)-(2\theta-1)^2\theta\Bigg]\leq 4\theta(2\theta-1)\\
 &\iff \frac{\nu\tau}{\frac{1}{\lambda_1}} \Bigg[(2\theta-1)[1-(2\theta-1)\theta]\Bigg]\leq 4\theta(2\theta-1)\\
  &\iff \frac{\nu\tau}{\frac{1}{\lambda_1}} \Bigg[(2\theta-1)(1-2\theta^2+\theta)\Bigg]\leq 4\theta(2\theta-1)\\
   &\iff \frac{\nu\tau}{\frac{1}{\lambda_1}} \Bigg[(2\theta-1)(-2\theta-1)(\theta-1)\Bigg]\leq 4\theta(2\theta-1)\\
   &\iff \frac{\nu\tau}{\frac{1}{\lambda_1}} \leq\frac{4\theta}{(-2\theta-1)(\theta-1)} \\
\end{align*}
From \eqref{eq:step-condition}, $\frac{\nu\tau}{\frac{1}{\lambda_1}}\in[0,1]$. And $\frac{4\theta}{(-2\theta-1)(\theta-1)}\geq 1$ for $\theta\in(\frac{1}{2},1]$. So the second inequality is also true, and the $\varepsilon$ meets the condition.\\
}

Plugging $\varepsilon$ into the definitions of $a$, $b$, and $\alpha$, we obtain the desired result:

\begin{align*}
\alpha 
&
=
 \theta 
 - \frac{1}{2} ( 2\theta - 1 ) ( 1 -  \theta )    \lambda_1 \nu \tau 
\\
& \qquad
+ 
\frac{1}{2}
\Big[ ( 2 \theta  - 1 )(1-\theta) \Big( 4\theta 
 - ( 1 - \theta )  ( 2 \theta  + 1 )   \lambda_1 \nu  \tau \Big) 
\lambda_1 \nu  \tau 
\Big]^{1/2}
,\\
a 
&
= 
\frac{1}{2} \Big\{
( 2 \theta - 1 ) \Big( 4 \theta  - \lambda_1 \nu ( 2\theta + 1 ) ( 1 - \theta )  \tau
 \Big)^{1/2}
- 
\Big( 
\lambda_1 \nu ( 1 - \theta )  \tau
\Big)^{1/2}
\Big\}
,
\\
b 
& 
=
\frac{1}{2} \Big\{
( 2 \theta - 1 ) \Big( 4 \theta  - \lambda_1 \nu ( 2\theta + 1 ) ( 1 - \theta )  \tau
 \Big)^{1/2}
+ 
\Big( 
\lambda_1 \nu ( 1 - \theta )  \tau
\Big)^{1/2}
\Big\}
.
\end{align*}
}
\end{confidential}
\begin{confidential}
    {\color{red}
\begin{align*} 
\alpha 
& 
= 
 \theta 
- \frac{1}{2} \varepsilon	
 + \Big( \theta -  \frac{1}{2} \Big)    \lambda_1 \nu \theta \tau 
\\
& \qquad
+ 
\frac{1}{2}
\Big( 8 \theta^2 - 4 \theta  - \varepsilon
 + ( 2\theta - 1 )^2    \lambda_1 \nu \theta \tau \Big)^{1/2}
\Big( - \varepsilon  + \lambda_1 \nu \theta \tau \Big)^{1/2}
\\
& 
= 
 \theta 
- \frac{1}{2} \lambda_1 \nu ( 2\theta - 1)  \tau
 + \Big( \theta -  \frac{1}{2} \Big)    \lambda_1 \nu \theta \tau 
\\
& \qquad
+ 
\frac{1}{2}
\Big( 8 \theta^2 - 4 \theta  
- \lambda_1 \nu ( 2\theta - 1)  \tau
 + ( 2\theta - 1 )^2    \lambda_1 \nu \theta \tau \Big)^{1/2}
 \\
 &
 \cdot
\Big( 
- \lambda_1 \nu ( 2\theta - 1)  \tau
+ \lambda_1 \nu \theta \tau 
\Big)^{1/2}
\\
& 
= 
 \theta 
 + \Big( - \theta + \frac{1}{2} + \theta^2 -  \frac{1}{2}\theta \Big)    \lambda_1 \nu \tau 
\\
& \qquad
+ 
\frac{1}{2}
\Big( 4\theta ( 2 \theta  - 1 )
 + (  - 2\theta + 1 + 4\theta^3  - 4\theta^2 + \theta )    \lambda_1 \nu  \tau \Big)^{1/2}
 \\
 &
 \cdot
\Big( 
\lambda_1 \nu ( - 2\theta + 1 + \theta ) \tau 
\Big)^{1/2}
\\
& 
= 
 \theta 
 + \Big( \theta^2 -  \frac{3}{2}\theta  + \frac{1}{2} \Big)    \lambda_1 \nu \tau 
\\
& \qquad
+ 
\frac{1}{2}
\Big( 4\theta ( 2 \theta  - 1 )
 + (  4\theta^3  - 4\theta^2 - \theta  + 1 )    \lambda_1 \nu  \tau \Big)^{1/2}
\Big( 
\lambda_1 \nu ( 1 - \theta ) \tau 
\Big)^{1/2}
\\
& 
= 
 \theta 
 + \frac{1}{2} ( 2\theta^2 -  3 \theta  + 1 )    \lambda_1 \nu \tau 
\\
& \qquad
+ 
\frac{1}{2}
\Big( 4\theta ( 2 \theta  - 1 )
 + \big (  4\theta^2 ( \theta  -  1 ) - \theta  + 1 \big)    \lambda_1 \nu  \tau \Big)^{1/2}
\Big( 
\lambda_1 \nu ( 1 - \theta ) \tau 
\Big)^{1/2}
\\
& 
= 
 \theta 
 + \frac{1}{2} ( 2\theta - 1 ) ( \theta  - 1 )    \lambda_1 \nu \tau 
\\
& \qquad
+ 
\frac{1}{2}
\Big( 4\theta ( 2 \theta  - 1 )
 + ( \theta  -  1 ) ( 2\theta - 1 ) ( 2 \theta  + 1 )    \lambda_1 \nu  \tau \Big)^{1/2}
\Big( 
\lambda_1 \nu ( 1 - \theta ) \tau 
\Big)^{1/2}
\\
& 
= 
 \theta 
 - \frac{1}{2} ( 2\theta - 1 ) ( 1 -  \theta )    \lambda_1 \nu \tau 
\\
& \qquad
+ 
\frac{1}{2}
\Big[ ( 2 \theta  - 1 ) \Big( 4\theta 
 - ( 1 - \theta )  ( 2 \theta  + 1 )   \lambda_1 \nu  \tau \Big) \Big]^{1/2}
\Big( 
\lambda_1 \nu ( 1 - \theta ) \tau 
\Big)^{1/2}
\\
& 
= 
 \theta 
 - \frac{1}{2} ( 2\theta - 1 ) ( 1 -  \theta )    \lambda_1 \nu \tau 
\\
& \qquad
+ 
\frac{1}{2}
\Big[ ( 2 \theta  - 1 )(1-\theta) \Big( 4\theta 
 - ( 1 - \theta )  ( 2 \theta  + 1 )   \lambda_1 \nu  \tau \Big) 
\lambda_1 \nu  \tau 
\Big]^{1/2}
.
\end{align*}

Also from \eqref{eq:a,b} we have
\begin{align*}
a 
& 
:=
\frac{1}{2} \Big\{
 \Big( 8 \theta^2 - 4 \theta  - \varepsilon
 + ( 2\theta - 1 )^2    \lambda_1 \nu \theta \tau \Big)^{1/2}
- 
\Big( - \varepsilon  + \lambda_1 \nu \theta \tau \Big)^{1/2}
  \Big\}
\\
=
& 
\frac{1}{2} \Big\{
 \Big( 8 \theta^2 - 4 \theta  
 - \lambda_1 \nu ( 2\theta - 1)  \tau
 + ( 2\theta - 1 )^2    \lambda_1 \nu \theta \tau \Big)^{1/2}
- 
\Big( 
- \lambda_1 \nu ( 2\theta - 1)  \tau
+ \lambda_1 \nu \theta \tau \Big)^{1/2}
  \Big\}
\\
=
& 
\frac{1}{2} \Big\{
\Big( 4 \theta  ( 2 \theta - 1 )
 + \lambda_1 \nu ( -  2\theta + 1 + 4 \theta^3 - 4 \theta^2 + \theta )  \tau
 \Big)^{1/2}
- 
\Big( 
\lambda_1 \nu (  - 2\theta + 1 + \theta )  \tau
\Big)^{1/2}
\Big\}
\\
=
& 
\frac{1}{2} \Big\{
\Big( 4 \theta  ( 2 \theta - 1 )
 + \lambda_1 \nu \big( 4 \theta^2 ( \theta - 1 )  - \theta + 1 \big)  \tau
 \Big)^{1/2}
- 
\Big( 
\lambda_1 \nu ( 1 - \theta )  \tau
\Big)^{1/2}
\Big\}
\\
=
& 
\frac{1}{2} \Big\{
\Big( 4 \theta  ( 2 \theta - 1 )
 + \lambda_1 \nu \big( ( 2\theta - 1 ) ( 2\theta + 1 ) ( \theta - 1 ) \big)  \tau
 \Big)^{1/2}
- 
\Big( 
\lambda_1 \nu ( 1 - \theta )  \tau
\Big)^{1/2}
\Big\}
\\
=
& 
\frac{1}{2} \Big\{
( 2 \theta - 1 ) \Big( 4 \theta  - \lambda_1 \nu ( 2\theta + 1 ) ( 1 - \theta )  \tau
 \Big)^{1/2}
- 
\Big( 
\lambda_1 \nu ( 1 - \theta )  \tau
\Big)^{1/2}
\Big\}
,
\end{align*}
namely
\begin{align*}
a 
&
= 
\frac{1}{2} \Big\{
( 2 \theta - 1 ) \Big( 4 \theta  - \lambda_1 \nu ( 2\theta + 1 ) ( 1 - \theta )  \tau
 \Big)^{1/2}
- 
\Big( 
\lambda_1 \nu ( 1 - \theta )  \tau
\Big)^{1/2}
\Big\}
,
\\
b 
& 
=
\frac{1}{2} \Big\{
( 2 \theta - 1 ) \Big( 4 \theta  - \lambda_1 \nu ( 2\theta + 1 ) ( 1 - \theta )  \tau
 \Big)^{1/2}
+ 
\Big( 
\lambda_1 \nu ( 1 - \theta )  \tau
\Big)^{1/2}
\Big\}
.
\end{align*}
}
\end{confidential}
\newpage
% References
\bibliographystyle{plain}
\bibliography{refs,database2}

\end{document}